\documentclass[smallextended]{svjour3}       
\smartqed  
\newtheorem{assumption}{Assumption}

\usepackage{amsmath,amsfonts,amssymb,tabularx,booktabs,cite,nicefrac}
\usepackage[dvipsnames]{xcolor}
\usepackage{algorithm,algorithmic,graphicx}
\usepackage{epsfig,epstopdf} 
\graphicspath{{eps/Inpaper/}}

\newcommand{\ve}[2]{\langle #1 ,  #2 \rangle}
\newcommand{\eqdef}{:=}
\newcommand{\R}{\mathbb{R}}

\newcommand{\kp}{_{k+1}}
\newcommand{\tnorm}[1]{\lVert{#1}\rVert^2}

\newcommand{\prox}[2]{\mbox{prox}_{h,#2}\left(#1\right)}

\graphicspath{
    {./eps/}{./}{./eps/Inpaper/}
}
\usepackage{mathtools}
\usepackage{changes}
\usepackage{soul}

\usepackage{booktabs}
\usepackage{tabu}
\usepackage{array}
\usepackage{arydshln}
\title{Fast and Safe: Accelerated gradient methods with optimality certificates and underestimate sequences}

\titlerunning{Fast and Safe: Accelerated gradient methods}        

\author{Majid Jahani\and
	 Naga~Venkata~C.~Gudapati \and
	  Chenxin Ma\and
	Rachael Tappenden \and Martin Tak\'a\v{c}
}

\institute{
M.~Jahani, N.~V.~C.~Gudapati, C.~Ma, M.~Tak\'a\v{c} \at
Department of Systems and Industrial Engineering, Lehigh University, H.S. Mohler Laboratory, 200 West Packer Avenue, Bethlehem, PA 18015, USA.\\
\email{maj316@lehigh.edu,nag415@lehigh.edu, chm514@lehigh.edu, Takac.MT@gmail.com}
	\and
R.~Tappenden \at
School of Mathematics and Statistics, University of Canterbury, Private Bag 4800, Christchurch 8140, New Zealand.
\email{rachael.tappenden@canterbury.ac.nz}
	\and
	Martin Tak\'a\v{c} was supported by NSF Grants CCF-1618717 and CMMI-1663256.
}

\date{Received: date / Accepted: date}

\begin{document}

\maketitle

\begin{abstract}
	In this work we introduce the concept of an Underestimate Sequence (UES), which is motivated by Nesterov's estimate sequence. Our definition of a UES utilizes three sequences, one of which is a lower bound (or under-estimator) of the objective function. The question of how to construct an appropriate sequence of lower bounds is addressed, and we present lower bounds for strongly convex smooth functions and for strongly convex composite functions, which adhere to the UES framework. Further, we propose several first order methods for minimizing strongly convex functions in both the smooth and composite cases. The algorithms, based on efficiently updating lower bounds on the objective functions, have natural stopping conditions that provide the user with a certificate of optimality. Convergence of all algorithms is guaranteed through the UES framework, and we show that all presented algorithms converge linearly, with the accelerated variants enjoying the optimal linear rate of convergence.
\end{abstract}

\section{Introduction}

In this work we are interested in solving the strongly convex, composite, optimization problem,
\begin{equation}\label{eq:Problem}
	\min_{x\in \R^n} \{F(x) \eqdef f(x) + h(x)\}.
\end{equation}
We use $x^*$ to denote the optimal solution of \eqref{eq:Problem}, and $F^*:= F(x^*)$ to denote the associated optimal function value. It is assumed that $h$ is a convex and possibly nonsmooth function.
\begin{assumption}\label{A_SCL}
The function $f(\cdot)$ is $\mu$-strongly convex and $L$-smooth, i.e., for all $x,y\in\R^n$, it holds that
\begin{align}
	f(x) \geq f(y) + \ve{\nabla f(y)}{x-y} + \tfrac{\mu}{2} \tnorm{x-y}, \label{eq:ass1} \\
	f(x) \leq f(y) + \ve{\nabla f(y)}{x-y} + \tfrac{L}{2} \tnorm{x-y}. \label{eq:ass2}
\end{align}
\end{assumption}
In words, Assumption~\ref{A_SCL} explains that $f$ is assumed to be continuously differentiable on $\R^n$ with $L\geq \mu>0$.
It is straightforward to show that strong convexity of $f$ implies strong convexity of $F$.

For problems of the form \eqref{eq:Problem}, which satisfy Assumption~\ref{A_SCL}, it is well known that Nesterov's methods \cite{Nesterov04,Nesterov07,Nesterov13} converge linearly, with the accelerated variants converging at the optimal rate of $(1-\sqrt{\mu/ L})$.

Nesterov's acceleration approach, and the idea of adding momentum, has led to the extensive analysis of accelerated first order methods in a variety of settings. This includes a recent surge of interest in investigating stochastic gradient methods \cite{robbins1951stochastic,schmidt2017minimizing,johnson2013accelerating,Ghadimi12} and their accelerated variants \cite{cotter2011better,shalev2013accelerated,kingma2014adam,nitanda2014stochastic}. Coordinate descent methods \cite{nesterov2012efficiency,richtarik2014iteration,Fountoulakis18,Tappenden16} are another class of algorithms that have proved extremely popular, largely because they can take advantage of modern parallel computing architecture \cite{jaggi2014communication,ma2015adding}, and this has also inspired much research into studying their accelerated versions \cite{fercoq2015accelerated,allen2016even,ACOCOA}.
However, while the theoretical and practical performance of Nesterov's methods is well established, a satisfactory geometric interpretation of these approaches has been elusive.

Recently the authors of \cite{Bubeck15,Drusvyatskiy16} proposed algorithms for smooth functions (i.e., $h \equiv 0$ in \eqref{eq:Problem}) that enjoy the same optimal rate of convergence as Nesterov's accelerated method, but also have a novel geometric intuition. Specifically, the geometric descent algorithm \cite{Bubeck15} achieves the optimal linear convergence rate, and shares a geometric intuition similar to that of ellipsoidal methods. The authors illustrate that the optimal rate is obtained by appropriately shrinking two balls that contain $x^*$ (the minimizer of $f(x)$) at each iteration.

Motivated by \cite{Bubeck15}, the paper \cite{Drusvyatskiy16} proposed the Optimal Quadratic Averaging (OQA) algorithm. Indeed, \cite{Drusvyatskiy16} show that their OQA algorithm generates the same iterate sequence as that generated by the algorithm in \cite{Bubeck15}, although the two schemes are slightly different. The OQA algorithm also maintains a sequence of quadratic lower bounds on the objective function, and at each iteration the new quadratic lower bound is formed as the optimal average of the current lower bound and the lower bound from the previous iteration. The gap between the function value $f(x_k)$ and the minimum value of lower bound, say $\phi_k^*$, converges to zero at the optimal rate. Importantly, the lower bound also acts a natural stopping criterion for the algorithm, and when $f(x_k) - \phi_k^* \leq \epsilon$, where $\epsilon >0$ is some stopping tolerance, then the user has a certificate of $\epsilon$-optimality, i.e., it is guaranteed that $f(x_k) - f^* \leq \epsilon$. In practice, the OQA algorithm can be equipped with historical information to achieve further speed up. However, the OQA algorithm and its history based variant need at least two calls of a line search process at every iteration, which can pose a heavy computational burden in terms of function evaluations. The authors in \cite{Drusvyatskiy16} also briefly describe how their \emph{unaccelerated OQA algorithm} can be extended to composite functions, and left as an open problem the possibility of deriving \emph{accelerated} proximal variants.

More recently, the authors of \cite{Chen16} successfully addressed the open problem in \cite{Drusvyatskiy16} and presented an accelerated algorithm for composite problems of the form \eqref{eq:Problem}, that achieves the optimal linear rate of convergence. Their algorithm, called the geometric proximal gradient (GeoPG) method also has a satisfying geometrical interpretation similar to that in \cite{Bubeck15}. Unfortunately, a major drawback of GeoPG in \cite{Chen16} is that the algorithm is rather complicated, and requires a couple of inner loops to determine necessary algorithm parameters. For example, for GeoPG one must find the root of a specific function and one is also required to compute a minimum enclosing ball via some iterative process; both of these steps must be carried out at every iteration, which is expensive.

Another relevant work is that in \cite{Ghadimi12} where the authors propose an accelerated stochastic approximation algorithm. That algorithm is based upon a modification of Nesterov's optimal smooth method \cite{Nesterov83} to fit a convex composite setting where only a noisy gradient is available. Section~5 of \cite{Ghadimi12} describes how certain stochastic lower bound can be incorporated into their method, although some additional computational effort is required to compute these. \textcolor{black}{Furthermore, in \cite{diakonikolas2019} the authors provide a general scheme for the analysis of first-order methods by construction of a `duality gap' involving an approximation of the objective function. In the continuous-time setting, the authors in \cite{diakonikolas2019} show that the approximate duality gap decreases. They also characterize the discretization errors incurred by different discretization methods (please see \cite{diakonikolas2019} for more details).}

In this paper we propose several new algorithms to solve problem \eqref{eq:Problem} that are motivated by, and extend, the previously mentioned works. In particular, we present four algorithms: a proximal Gradient Descent (GD) type algorithm for composite problems, an accelerated proximal GD type algorithm for composite problems, a GD type algorithm for smooth problems, and an accelerated GD type algorithm for smooth problems. Our algorithms all converge linearly, and the accelerated variants converge at the optimal linear rate. These algorithms blend the positive features of Nesterov's methods \cite{Nesterov04,Nesterov07,Nesterov13} and the OQA algorithm \cite{Drusvyatskiy16}, and thus enjoy the advantages of both approaches. First, similarly to Nesterov's methods, no line search is needed by any of our algorithms as long as we make the standard assumption that the Lipschitz constant $L$ is known or is easily computable. Hence, there are no `inner-loops' in any of our algorithm variants, which ensures that the computational cost is low and is fixed at every iteration. Secondly, our algorithms incorporate quadratic lower bounds so they have natural stopping conditions; a feature that is similar to OQA. Each iteration of the OQA method in \cite{Drusvyatskiy16} requires two `optimal steps': an optimal line-search in a given direction and the optimal choice of the averaging of successive quadratics. In this work, we show that instead of making such optimal choices, one can use carefully coupled convex combinations in the two steps, and this approach is computationally cheaper. The resulting scheme maintains the optimal linear rate of convergence and naturally extends to the proximal setting.

Furthermore, in this work we propose the definition of an UnderEstimate Sequence (UES), which was motivated by Nesterov's Estimate Sequence \cite{Nesterov83}. Perhaps surprisingly, early work on estimate sequences appeared to be largely overlooked, but since Nesterov's work on smoothing techniques in the early 2000s \cite{Nesterov05}, they have seen a revival in popularity. For example, the work of Baes in \cite{Baes09}, the development of a randomized estimate sequence in \cite{Lu15} and an approximate estimate sequence in \cite{Lin15}.
The definition we introduce is different from previous works because an underestimate sequence involves \emph{lower bounds} on the objective function. To the best of our knowledge, this is the first work which proposes estimate sequences that form \emph{lower bounds} on the objective function.
The UES definition is the powerhouse of our convergence analysis; we prove that each of our proposed algorithms generates a UES, and consequently the algorithms converge (linearly) to the optimal solution of problem \eqref{eq:Problem}. While we describe 4 new algorithms in this work, we stress that the UES definition is general, and it allows a plethora of algorithms to be developed. Moreover, any developed algorithm whose iterates generate a UES is guaranteed to converge to the optimal value $F^*$ (under a mild assumption on one of the algorithm parameters).

The algorithms presented in this work can be used to solve problem \eqref{eq:Problem}. However, they are also more widely applicable as a subproblem solver in existing optimization algorithms. In particular, the Inexact Coordinate Descent algorithm \cite{Tappenden2016}, and the Universal Catalyst algorithm \cite{Lin15}, are guaranteed to converge if at every iteration, the inexact solution to a certain subproblem is $\epsilon$-optimal. In both cases, the algorithms solve optimization problems where the objective function is convex (but not necessarily strongly convex) \emph{but the arising subproblems are strongly convex}. Neither paper explains how to verify the subproblem stopping conditions. This work builds the link, bridging theory and practice, by providing algorithms that can be used to solve subproblems and which return solutions that are easily verified to be $\epsilon$-optimal.

\subsection{Contributions}
The main contributions of this paper are stated now (listed in no particular order).
\begin{itemize}
\item \textbf{Underestimate Sequence.} We introduce the concept of an UnderEstimate Sequence (UES).
    The UES consists of three sequences $\{x_k\}_{k=0}^\infty$, $\{\phi_k(x)\}_{k=0}^\infty$ and $\{\alpha_k\}_{k=0}^\infty$, where for all $k$, $\phi_k(x)$ is a \emph{global lower bound} on the objective function $F(x)$. While there have been several extensions and variants of Nesterov's work on estimate sequences \cite{Nesterov83}, the definition of a UES involves \emph{lower bounds} or \emph{under}estimates of $F(x)$, which is new. The UES framework is general, conceptually simple, and it allows the construction of a wide variety of algorithms to solve \eqref{eq:Problem}.
\item \textbf{New algorithms.} Four new algorithms are presented that are computationally efficient and whose iterates generate a UES. Crucially, two of our algorithms solve the \emph{composite} problem \eqref{eq:Problem}. The algorithms are: (i) CUESA, a proximal GD type algorithm for composite problems; (ii) ACUESA, an accelerated proximal GD type algorithm for composite problems; (iii) SUESA, a GD type algorithm for smooth problems; and (iv) ASUESA, an accelerated GD type algorithm for smooth problems.
\item \textbf{Algorithms with optimal convergence rate.} Each of the four algorithms generate iterates that form a UES, and all are guaranteed to converge to the optimal solution of \eqref{eq:Problem}. Moreover, all algorithms converge linearly, and the accelerated algorithms (ACUESA and ASUESA) converge \emph{at the optimal rate.}
\item \textbf{Algorithms with convergence certificates.}  The underestimate sequence builds a global lower bound of $F(x)$ at each iteration, and the gap between the (minimum of the) lower bound and $F(x_k)$ tends to zero. Thus, this difference acts as a kind of surrogate ``duality gap'', and once this gap falls below some (user defined) stopping tolerance $\epsilon$, it is guaranteed that the point returned by the algorithm is $\epsilon$-optimal. The algorithms can be used as subproblem solvers within existing algorithms, to ensure that the solutions to the subproblems are $\epsilon$-accurate.
\item \textbf{No line search.} The algorithms developed in this work are computationally efficient and do not involve any `inner loops'. In contrast, the methods in \cite{Bubeck15,Drusvyatskiy16,Chen16} all involve an exact linesearch or a root finding process to determine necessary algorithmic parameters, which comes with an additional computational cost.
\end{itemize}

\subsection{Paper Outline}
The paper is organized as follows. In the next section the definition of an Underestimate Sequence (UES) is presented, along with a proposition which shows that if one has a UES, then, under a mild assumption, it is guaranteed that $F(x_k) - F^* \to 0$ linearly. Section~\ref{sec:lb} discusses lower bounds for function $F$ (in both the smooth and composite cases), and these lower bounds are a critical part of the underestimate sequences framework. In Section~\ref{sec:composite} we propose two algorithms for solving \eqref{eq:Problem} in the composite case ($h\neq 0$) and in Section~\ref{sec:smooth} we present two algorithms for solving \eqref{eq:Problem} in the smooth case ($h\equiv0$). The algorithms are supported by convergence theory, which shows that they are all guaranteed to converge linearly, and the accelerated algorithms achieve the optimal rate. In Section~\ref{sec:adaptiveL} an adaptive Lipschitz constant updating scheme is presented that can be used as an inner loop within the four previously mentioned algorithms, so that the true Lipschitz constant is not explicitly needed. Section~\ref{sec:numericalexperiments} presents numerical experiments to demonstrate the practical advantages of our proposed algorithms, and concluding remarks are given in Section~\ref{sec:conclusion}.

\section{Underestimate Sequence}\label{sec:UES}

In this section, the definition of an Underestimate Sequence (UES) is presented, as well as a proposition showing that if one has a UES then $F(x_k) - F^* \to 0$.

\begin{definition}\label{def:UES}
	A series of sequences $\{x_k\}_{k=0}^\infty$, $\{\phi_k(x)\}_{k=0}^\infty$ and $\{\alpha_k\}_{k=0}^\infty$, where $\alpha_k\in (0,1)$ for all $k\geq0$, 
is called an Underestimate Sequence (UES) of the function $F(x)$ if, for all $x\in\R^n$ and for all $k\geq 0$ we have,
	\begin{align}
		\phi_k(x) &\leq F(x),  \label{eq:def1}\\
		F(x_{k+1}) -\phi_{k+1}^* &\leq (1-\alpha_k) (F(x_{k}) - \phi_{k}^*), \label{eq:def2}
	\end{align}
	where $\phi_k^* := \min_x \phi_k(x)$.
\end{definition}

\begin{proposition}\label{prop:UES}
	If $\{x_k\}_{k=0}^\infty$, $\{\phi_k(x)\}_{k=0}^\infty$ and $\{\alpha_k\}_{k=0}^\infty$ is a UES of $F(x)$, then
	\begin{equation}\label{eq:linconverge}
		F(x_k) -\phi_k^* \leq \lambda_k(F(x_0) - \phi_0^*),
	\end{equation}
	where $\lambda_0 = 1$ and $\lambda_{k+1} =  (1-\alpha_k)\lambda_k$. 
Furthermore, since $\phi_k^*\leq \phi_k(x^*) \leq F^*$ for all $k\geq0$, if $\lambda_k \to 0$, then \eqref{eq:linconverge} implies that $\{F(x_k) - F^*\}_{k=0}^\infty$ converges to $0$.
\end{proposition}
\begin{remark}
  Note that if $\sum_{k=0}^\infty \alpha_k = \infty$ then $\lambda_k\to 0$; also see \cite[pg.72--73]{Nesterov04}.
\end{remark}


Definition~\ref{def:UES} is different from Nesterov's Estimate Sequence (ES) in several ways. Although both a UES and an ES contain a sequence of estimators $\{\phi_k(x)\}_{k=0}^\infty$ for $F$, Definition~\ref{def:UES} shows that $\phi_k$ must be a \emph{lower/under estimator} of $F$ for all $k\geq0$, but this necessarily does not hold for an ES. Nesterov's ES convergence guarantees rely upon $F(x_{k}) \leq \phi^*_{k}$ holding, but this \emph{does not hold} in our case. Moreover, the definition of an ES only contains two sequences, while the UES has an extra sequence of points $\{x_k\}_{k=0}^\infty$. This enables us to show that the gap between the function value at $x_k$ and $\phi_k^*$ decreases in the $k$th iteration. The quantities in Definition~\ref{def:UES} and Proposition~\ref{prop:UES} are all computable; they do not require knowledge of $x^*$. (A more detailed comparison can be found in Appendix~\ref{sec:appendix}.)

Proposition~\ref{prop:UES} shows that, if $\lambda_k\to 0$ (where $\lambda_{k+1} = \prod_{j=0}^{k}(1-\alpha_j)$), then any sequences that form a UES (i.e., satisfy Definition~\ref{def:UES}) are guaranteed to converge to the optimal solution of problem \eqref{eq:Problem}, and $F(x_k) -\phi_k^* \to 0$. Additionally, if $\alpha_k \in (0,1)$ is constant for all $k\geq 0$, then a linear rate of convergence holds. Thus, the UES construction provides a general framework for determining whether an optimization algorithm for problem \eqref{eq:Problem} will converge (linearly). In particular, if the iterates generated by an optimization algorithm satisfy Definition~\ref{def:UES} and $\lambda_k \to 0$, then that algorithm converges.

The UES framework is not only interesting from a theoretical perspective, but it also has an important practical advantage. In particular, $F(x_k) -\phi_k^*$ provides a natural stopping criterion when designing algorithms, because $F(x_k)$ and $\phi_k^*$ are upper and lower bounds for $F^*$, respectively. This difference is a kind of surrogate for the duality gap, so algorithms that adhere to the UES framework are provided with a certificate of optimality, which is a desirable attribute.

%
%

\section{Lower Bounds via Quadratic Averaging}\label{sec:lb}

The purpose of this section is to introduce (global) lower bounds for the function $F$ defined in \eqref{eq:Problem}, in both the smooth ($h=0$) and nonsmooth cases. Lower bounds are the cornerstone of the UES set up, as seen in \eqref{eq:def1} in Definition~\ref{def:UES}. The efficient construction of global lower bounds for $F$ allows the development of practical algorithms whose convergence can be analyzed via the UES framework.

Before stating the lower bounds, several technical results are presented that will be used throughout this paper.

\subsection{Preliminary Technical Results}

The proximal map is defined as
\begin{equation}\label{eq:proxmap}
\prox{x}{\gamma} \eqdef \arg\min_{u} \{ h(u) + \tfrac \gamma 2 \tnorm{x-u} \},
\end{equation}
and the proximal gradient is
\begin{equation}\label{eq:proxgrad}
	G_\gamma(x) \eqdef \gamma\left(x - \prox{x- \tfrac{1}{\gamma} \nabla f(x) }{\gamma}\right).
\end{equation}
Definitions \eqref{eq:proxmap} and \eqref{eq:proxgrad} will be used with $\gamma \equiv L$. Given some point $x\in \R^n$, a short step and a long step are denoted by
\begin{eqnarray}
    x^{+}  &\eqdef& x - \tfrac{1}{L} G_L(x),\label{eq:shortproxstep}\\
	x^{++} &\eqdef& x - \tfrac{1}{\mu} G_L(x). \label{eq:longproxstep}
\end{eqnarray}
In the smooth case ($h\equiv0$), the proximal gradient is simply the gradient $\nabla f(\cdot)$, so the short and long steps (\eqref{eq:shortproxstep} and \eqref{eq:longproxstep}) simplify as
\begin{eqnarray}
  x^+ &=& x - \tfrac1 L \nabla f(x) \label{eq:shortstep}\\
  x^{++} &=& x - \tfrac{1}{\mu} \nabla f(x).\label{eq:longstep}
\end{eqnarray}

For a function $h:\R^n \to \R$, the elements $s \in \R^n$ that satisfy
\begin{equation*}
  h(y) \geq h(x) + \langle s,y-x\rangle, \qquad \forall y \in \R^n,
\end{equation*}
are called the subgradients of $h$ at the point $x$. In words, all elements defining a linear function that supports
$h$ at a point $x$ are subgradients. The set of all $s$ at a point $x$ is called the subdifferential of $h$ and it is
denoted by $\partial h(x)$. The following Lemma characterizes elements of $\partial h(x^+)$.
\begin{lemma}\label{smalllemma}
Let $G_L(x)$ and $x^+$ be defined in \eqref{eq:proxgrad} and \eqref{eq:shortproxstep}, respectively. Then, for all $x\in\R^n$, $G_L(x) - \nabla f(x)\in \partial h(x^+)$.
\end{lemma}
	
\begin{proof} For a given point $x\in\R^n$, combining \eqref{eq:proxmap}, \eqref{eq:proxgrad} and \eqref{eq:shortproxstep} gives
\begin{eqnarray*}
		x^+ = \arg\min_u \left\{h(u) + \tfrac L2 \tnorm{u- (x - \tfrac{1}{L} \nabla f(x))}  \right\}.
	\end{eqnarray*}
	Then
	$$ 0\in  \tfrac 1L \partial h(x^+) + x^+ - (x - \tfrac{1}{L} \nabla f(x) ) \overset{\eqref{eq:shortproxstep}}{=} \tfrac 1L \partial h(x^+) - \tfrac 1L (G_L(x) -\nabla f(x)).$$
Multiplying through by $L$, and rearranging, gives the result.\qed
\end{proof}

\subsection{A Lower Bound for Composite Functions}\label{s:lbcomposite}

Here, a lower bound is developed for composite function, i.e., it is assumed that $h$ is not equivalent to the zero function. The following Lemma defines a lower bound for $F(x)$ in \eqref{eq:Problem}. The lower bound is the same as that presented in \cite{Chen16} and \cite{Drusvyatskiy16}, with the roles of $x$ and $y$ reversed here; the proof is included for completeness.
\begin{lemma}[Lemma~3.1 in \cite{Chen16}; Lemma~6.1 in \cite{Drusvyatskiy16}]
  Given a point $y \in \R^n$, let $G_L(y)$ and $y^+$ be defined in \eqref{eq:proxgrad} and \eqref{eq:shortproxstep}, respectively. Then for all $x\in \R^n$
  \begin{eqnarray}\label{eq:lbcomposite}
    \varphi(x;y) \eqdef F(y^+) + \langle G_L(y),x-y\rangle + \tfrac{\mu}2\tnorm{x-y} + \tfrac1{2L}\tnorm{G_L(y)} \leq F(x).
  \end{eqnarray}
\end{lemma}
\begin{proof}
   By Assumption~\ref{A_SCL} ($\mu$-strongly convex)
\begin{eqnarray}\label{eq:mustrong}
 f(y) + \ve{\nabla f(y) }{x-y} + \tfrac \mu 2 \tnorm{x-y} \leq f(x), \qquad \forall x,y\in\R^n,
\end{eqnarray}
and ($L$-smooth)
\begin{eqnarray}\label{eq:Lsmooth}
\notag
  f(y^+) &\leq& f(y) + \langle \nabla f(y), y^+ - y \rangle + \tfrac L{2 }\tnorm{y^+-y }\\
  &\overset{\eqref{eq:shortproxstep}}{=}& f(y) -\tfrac1L \langle \nabla f(y), G_L(y)\rangle + \tfrac1{2L}\tnorm{G_L(y)}.
\end{eqnarray}
Combining \eqref{eq:mustrong} and \eqref{eq:Lsmooth} gives
\begin{eqnarray*}
\notag
  F(y^+) &\leq&  F(x) - \ve{\nabla f(y) }{x-y} - \tfrac \mu 2 \tnorm{x-y} -\tfrac1L \langle \nabla f(y), G_L(y)\rangle\\
  && + \tfrac1{2L}\tnorm{G_L(y)} + (h(y^+) - h(x))\\
  &=&  F(x) - \ve{\nabla f(y) }{x-y^+} - \tfrac \mu 2 \tnorm{x-y} \\
  && + \tfrac1{2L}\tnorm{G_L(y)} + (h(y^+) - h(x))\\
  &=&  F(x) - \ve{\nabla f(y) - G_L(y)}{x-y^+} - \tfrac \mu 2 \tnorm{x-y}\\
  && + \tfrac1{2L}\tnorm{G_L(y)} + (h(y^+) - h(x)) - \ve{G_L(y)}{x-y^+}\\
  &\overset{{\rm Lemma}~\ref{smalllemma}}{\leq}&  F(x)- \tfrac \mu 2 \tnorm{x-y} + \tfrac1{2L}\tnorm{G_L(y)} - \ve{G_L(y)}{x-y^+}\\
  &=&  F(x)- \tfrac \mu 2 \tnorm{x-y} - \tfrac1{2L}\tnorm{G_L(y)} - \ve{G_L(y)}{x-y}.
\end{eqnarray*}
Rearranging gives the result.\qed
\end{proof}

Before stating the next result, which shows that $\varphi(x;y)$ is a quadratic lower bound, we give the following equivalence,
\begin{eqnarray}\label{eq:normequivcomposite}
  \tfrac{\mu}2\tnorm{x- y^{++}} \overset{\eqref{eq:longproxstep}}{=} \tfrac{\mu}2\tnorm{x-y} + \langle x-y,G_L(y)\rangle + \tfrac1{2\mu}\tnorm{G_L(y)}.
\end{eqnarray}

\begin{lemma}\label{lem:varphixycanon}
  For all $x,y\in\R^n$, the lower bound \eqref{eq:lbcomposite} has the canonical form
  \begin{equation}\label{eq:varphixycanon}
  \varphi(x;y) = \varphi^* + \tfrac{\mu}{2} \tnorm{x-y^{++}},
\end{equation}
where
 \begin{equation}\label{eq:varphiminval}
  \varphi^* = F(y^+) + \left(\tfrac1{2L} - \tfrac1{2\mu} \right)\tnorm{G_L(y)}.
\end{equation}
\end{lemma}
\begin{proof}
Minimizing $\varphi(x;y)$ in \eqref{eq:lbcomposite} w.r.t. $x$, and using the definition in \eqref{eq:proxmap}, yields the minimizer
\begin{equation}\label{eq:varphiminimizer}
  y^{++} = \arg\min_x \varphi(x;y).
\end{equation}
The corresponding minimal value is
\begin{align}\label{eq:varphistar}
 \varphi^* &\eqdef \min_x\varphi(x;y) = \varphi(y^{++};y)\notag\\
 &\overset{\eqref{eq:lbcomposite}}{=} F(y^+) + \langle G_L(y),y^{++}-y\rangle + \tfrac{\mu}2 \tnorm{y^{++}-y} + \tfrac1{2L}\tnorm{G_L(y)}\notag\\
 &\overset{\eqref{eq:longproxstep}}{=} F(y^+) - \tfrac1{\mu}\langle G_L(y),G_L(y) + \tfrac{\mu}2\tnorm{\tfrac1{\mu}G_L(y)} + \tfrac1{2L}\tnorm{G_L(y)}\notag\\
 &= F(y^+) + \left(\tfrac1{2L} - \tfrac1{2\mu} \right)\tnorm{G_L(y)},
\end{align}
which is equivalent to \eqref{eq:varphiminimizer}. (Note also that \eqref{eq:varphiminimizer} and \eqref{eq:varphiminval} are the minimizer and minimum value of \eqref{eq:varphixycanon}, respectively.) Furthermore,
\begin{eqnarray*}
    \varphi(x;y) &\overset{\eqref{eq:lbcomposite}}{=}& F(y^+) + \langle G_L(y),x-y\rangle + \tfrac{\mu}2\tnorm{x-y} + \tfrac1{2L}\tnorm{G_L(y)}\\
    &\overset{\eqref{eq:normequivcomposite}}{=}& F(y^+) + \left(\tfrac1{2L} - \tfrac1{2\mu} \right)\tnorm{G_L(y)} + \tfrac{\mu}2\tnorm{x-y^{++}},
  \end{eqnarray*}
which, by \eqref{eq:varphistar}, confirms that \eqref{eq:varphixycanon} is equivalent to \eqref{eq:lbcomposite}.\qed
\end{proof}
\begin{remark}
  Lemma~\ref{lem:varphixycanon} shows that the lower bound \eqref{eq:lbcomposite} (equivalently \eqref{eq:varphixycanon}) is a quadratic lower bound for $F(x)$.
\end{remark}
Now, a sequence of lower bounds $\{\varphi_k(x)\}_{k=0}^\infty$ can be defined in the following way. Using \eqref{eq:lbcomposite} and a given initial point $x_0$, define the function
\begin{equation}\label{eq:compositecanonical}
  \varphi_0(x) \eqdef \varphi(x;x_0) = \varphi_0^* + \tfrac{\mu}2\tnorm{x-v_0},
\end{equation}
where
\begin{eqnarray}\label{eq:minvalminmize}
  \varphi_0^* \eqdef F(x_0^+) + \left(\tfrac1{2L} - \tfrac1{2\mu} \right)\tnorm{G_L(x_0)} \quad \text{and} \quad v_0 = x_0^{++}.
\end{eqnarray}
Differentiating \eqref{eq:compositecanonical} w.r.t. $x$ shows that the minimum value and minimizer of $\varphi_0(x)$ are given by \eqref{eq:minvalminmize}. This motivates the following construction: for a given $x_0 \in \R^n$,
\begin{enumerate}
  \item $\varphi_0(x) \eqdef \varphi(x;x_0) = \varphi_0^* + \tfrac{\mu}2\tnorm{x-v_0},$
  \item For $k\geq 0$, $\alpha_k \in (0,1)$, and some point $y_k$, recursively define
  \begin{equation}\label{eq:varphiconvexcomb}
    \varphi_{k+1}(x) \eqdef (1-\alpha_k) \varphi_k(x) + \alpha_k \varphi(x;y_k).
  \end{equation}
\end{enumerate}

\begin{lemma}\label{lem:varphicanonical}
  For all $k \geq 0$, $\varphi_{k+1}$ can be written in the canonical form
  \begin{eqnarray}\label{eq:canon}
    \varphi_{k+1}(x) \eqdef \varphi_{k+1}^* + \tfrac{\mu}2\tnorm{x-v_{k+1}},
  \end{eqnarray}
  where
  \begin{eqnarray}
    v_{k+1} &\eqdef&(1-\alpha_k)v_k + \alpha_ky_k^{++}\label{eq:vkcomp}\\
    \varphi_{k+1}^* &\eqdef& (1-\alpha_k)(\varphi_k^* + \tfrac{\mu}2\tnorm{v_{k+1}-v_{k}})\label{eq:phikcomp}\\
    \notag
     &&+\; \alpha_k\left(F(y_k^+) + \left(\tfrac1{2L} - \tfrac1{2\mu}\right)\tnorm{G_L(y_k)} + \tfrac{\mu}2\tnorm{v_{k+1} - y_k^{++}}\right).
  \end{eqnarray}
\end{lemma}
\begin{proof}
The proof is by induction and the induction hypothesis is
\begin{equation}\label{eq:inductionhypoth}
  \varphi_{k}(x) = \varphi_k^* + \tfrac{\mu}2\tnorm{x - v_k}.
\end{equation}
The following holds:
\begin{eqnarray}\label{eq:normintermediate}
  &&(1-\alpha_k)\tnorm{x - v_k}+ \alpha_k\tnorm{x-y_k^{++}}\notag\\
  &=&  (1-\alpha_k)\tnorm{x -v_{k+1}+ v_{k+1}- v_k}+ \alpha_k\tnorm{x-v_{k+1}+ v_{k+1}-y_k^{++}}\notag\\
  &=&  (1-\alpha_k)\left(\tnorm{x -v_{k+1}}+ \tnorm{v_{k+1}- v_k} + 2\langle x -v_{k+1},v_{k+1}- v_k\rangle\right)\notag\\
  &&+ \alpha_k\left(\tnorm{x-v_{k+1}}+ \tnorm{v_{k+1}-y_k^{++}} + 2\langle x-v_{k+1},v_{k+1}-y_k^{++}\rangle\right) \notag\\
  &\overset{\eqref{eq:vkcomp}}{=}& \tnorm{x -v_{k+1}}+ (1-\alpha_k)\tnorm{v_{k+1}- v_k} + \alpha_k\tnorm{x-v_{k+1}}\notag\\
  && + 2(1-\alpha_k)\langle x -v_{k+1},-\alpha_k(v_k-y_k^{++})\rangle + 2\alpha_k\langle x-v_{k+1},(1-\alpha_k)(v_k-y_k^{++})\rangle\notag\\
  &=& \tnorm{x -v_{k+1}}+ (1-\alpha_k)\tnorm{v_{k+1}- v_k} + \alpha_k\tnorm{x-v_{k+1}}.
\end{eqnarray}
Now,
\begin{eqnarray*}
  \varphi_{k+1}(x) &\overset{\eqref{eq:varphiconvexcomb}}{=}& (1-\alpha_k) \varphi_k(x) + \alpha_k \varphi(x;y_k)\\
  &\overset{\eqref{eq:inductionhypoth}}{=}& (1-\alpha_k)(\varphi_k^* + \tfrac{\mu}2\tnorm{x - v_k}) + \alpha_k \varphi(x;y_k)\\
  &\overset{\eqref{eq:varphixycanon}+\eqref{eq:varphiminval}}{=}& (1-\alpha_k)(\varphi_k^* + \tfrac{\mu}2\tnorm{x - v_k})\\
   &&+ \alpha_k\left(F(y_k^+) + \left(\tfrac1{2L} - \tfrac1{2\mu} \right)\tnorm{G_L(y_k)} + \tfrac{\mu}{2} \tnorm{x-y_k^{++}}\right)\\
   &\overset{\eqref{eq:normintermediate}}{=}& (1-\alpha_k)(\varphi_k^* + \tfrac{\mu}2\tnorm{v_{k+1}- v_k}) + \tfrac{\mu}2\tnorm{x -v_{k+1}}\\
   &&+ \alpha_k\left(F(y_k^+) + \left(\tfrac1{2L} - \tfrac1{2\mu} \right)\tnorm{G_L(y_k)} + \tfrac{\mu}{2} \tnorm{x-v_{k+1}}\right)\\
   &\overset{\eqref{eq:phikcomp}}{=}& \varphi_{k+1}^* + \tfrac{\mu}2\tnorm{x -v_{k+1}}.
\end{eqnarray*}
The proof is complete.\qed
\end{proof}

Lemma~\ref{lem:varphicanonical} shows that the lower bound is quadratic. This feature is important because it is easy to find the minimizer and minimum value of a quadratic. Indeed, Lemma~\ref{lem:varphicanonical} shows that the minimizer of the quadratic lower bound $\phi_{k+1}(x)$ is $v_{k+1}$, which is easily computed/updated using \eqref{eq:vkcomp}. Moreover, the minimum value of the quadratic is $\phi_{k+1}^*$ in \eqref{eq:phikcomp}, which is also readily available. Both $v_{k+1}$ and $\phi_{k+1}^*$ are used in the algorithms presented later in this work, and the expressions given in the previous lemma show that they can be computed cheaply and only depend upon quantities that are computed at iteration $k$ (no long memory is needed).

The following lemma gives an equivalent expression for $\varphi_{k+1}^*$, which is slightly more convenient in terms of ease of computation.
\begin{lemma}
  An equivalent expression for $\varphi_{k+1}^*$ in \eqref{eq:phikcomp} is
\begin{eqnarray}\label{eq:varphistarequiv}
    \varphi_{k+1}^* &=& (1-\alpha_k)\left(\varphi_k^* + \alpha_k\tfrac{\mu}2\tnorm{v_{k}-y_k^{++}}\right)\\
    \notag
    &&+ \alpha_k\left(F(y_k^+) +  \left(\tfrac1{2L} -\tfrac1{2\mu} \right)\tnorm{G_L(y_k)}\right).
\end{eqnarray}
\end{lemma}
\begin{proof}
Using \eqref{eq:vkcomp} gives the equivalences
\begin{eqnarray}\label{eq:vkp1vk}
  \tnorm{v_{k+1}-v_{k}} &=& \tnorm{(1-\alpha_k)v_k + \alpha_k y_k^{++}-v_{k}} = \alpha_k^2\tnorm{v_{k}-y_k^{++}}
\end{eqnarray}
and
\begin{eqnarray}
\notag
  \tnorm{v_{k+1} - y_k^{++}} &=& \tnorm{(1-\alpha_k)v_k + \alpha_k y_k^{++} - y_k^{++}}\\
  &=& (1-\alpha_k)^2\tnorm{v_k - y_k^{++}}.\label{eq:vkp1xkpp}
\end{eqnarray}
Combining \eqref{eq:vkp1vk} and \eqref{eq:vkp1xkpp} gives
\begin{eqnarray}
\notag
  &&\hspace{-5mm}(1-\alpha_k)\tnorm{v_{k+1}-v_{k}} + \alpha_k\tnorm{v_{k+1} - y_k^{++}}\\
  \notag
  &=&\alpha_k^2(1-\alpha_k)\tnorm{v_{k}-y_k^{++}} + \alpha_k(1-\alpha_k)^2\tnorm{v_k - y_k^{++}}\\
  \notag
  &=&\alpha_k(1-\alpha_k)(\alpha_k + (1-\alpha_k)) \tnorm{v_{k}-y_k^{++}} \\
  &=&\alpha_k(1-\alpha_k)\tnorm{v_{k}-y_k^{++}}.\label{eq:vkequiv}
\end{eqnarray}
Substituting \eqref{eq:vkequiv} into \eqref{eq:phikcomp} gives the result.\qed
\end{proof}

\begin{lemma}\label{lem:varphilessthanF}
Let $\alpha_k \in (0,1)$ $\forall k$. Then, for all $x\in\R^n$, $\varphi_k(x) \leq F(x)$ for all $k\geq 0$.
\end{lemma}
\begin{proof}
  The proof follows by a simple induction argument so is omitted for brevity.
\end{proof}
Lemma~\ref{lem:varphilessthanF} is fundamentally important for this work, because it shows that the quadratic bounds presented in the previous lemmas are \emph{lower bounds} for $F(x_k)$ for all $k\geq 0$.  If we can develop an algorithm that satisfies both conditions \eqref{eq:def1} and \eqref{eq:def2} in Definition~\eqref{def:UES}, then the algorithm is guaranteed to converge by Proposition~\ref{prop:UES}. But the quadratic lower bounds do indeed satisfy \eqref{eq:def1}, so it remains to develop an algorithm whose iterates satisfy \eqref{eq:def2}.

\subsection{A Lower Bound for Smooth Functions}\label{s:lbsmooth}

In this section, only smooth functions are considered, so here it is assumed that $h=0$. Note that when $h\equiv 0$, \eqref{eq:proxmap} and \eqref{eq:proxgrad} show that $G_L(y) = \nabla f(y)$. Now, for any point $y\in \R^n$, one can define a lower bound
\begin{equation}\label{eq:lowerboundDF}
	\phi(x;y) := f(y) - \tfrac{1}{2\mu}\tnorm{\nabla f(y)} + \tfrac{\mu}{2} \tnorm{x- y^{++} } \leq f(x),
\end{equation}
which holds with equality $\phi(x;y) = f(x)$ if and only if $x = y$. The lower bound in \eqref{eq:lowerboundDF} is a consequence of Assumption~\ref{A_SCL} ($\mu$-strongly convex) and the equivalence \eqref{eq:normequivcomposite} used with $G_L(y)=\nabla f(y)$.

\begin{remark}
  Notice that $\varphi(x;y)$ in \eqref{eq:lbcomposite} does not generalize $\phi(x;y)$ in \eqref{eq:lowerboundDF} in the sense that \eqref{eq:lbcomposite} used with $h\equiv 0$ \emph{does not} recover \eqref{eq:lowerboundDF}. (This will be discussed further in Section~\ref{sec:compare}.)
\end{remark}

Now, a sequence of lower bounds $\{\phi_k(x)\}_{k=0}^\infty$ can be defined in the following way. Using \eqref{eq:lowerboundDF} and a given initial point $x_0$, define the function
\begin{equation}\label{eq:phi0}
  \phi_0(x) \eqdef \phi(x; x_0) = \phi_0^* + \tfrac{\mu}{2}\tnorm{x-v_0},
\end{equation}
where
\begin{eqnarray}\label{eq:c0v0}
  \phi_0^* = f(x_0) - \tfrac{1}{2\mu}\tnorm{\nabla f(x_0)} \quad \text{and}\quad  v_0 = x_0^{++}.
\end{eqnarray}
Differentiating the expression in \eqref{eq:phi0} w.r.t. $x$ shows that $\phi_0^*$ and $v_0$ in \eqref{eq:c0v0} are the minimum value and minimizer of $\phi_0(x)$, respectively. This motivates the following construction:
\begin{enumerate}
  \item $\phi_0(x) \eqdef \phi(x; x_0) = \phi_0^* + \tfrac{\mu}{2}\tnorm{x-v_0}$
  \item For $k\geq 0$, $\alpha_k \in (0,1)$, and some point $y_k$, recursively define
  \begin{equation}\label{eq:zzzzz4}
	\phi_{k+1}(x) := (1- \alpha_k) \phi_k(x) + \alpha_k \phi(x; y_k).
\end{equation}
\end{enumerate}

\begin{lemma}\label{l:equivphikstar}
  For all $k\geq0$, $\phi\kp$ has the canonical form
  \begin{equation}\label{eq:phicanonical}
    \phi\kp(x) = \phi\kp^* + \tfrac{\mu}2 \tnorm{x - v\kp},
  \end{equation}
  where $\alpha_k \in (0,1)$, $v_{k+1}$ is defined in \eqref{eq:vkequiv} and
  \begin{align}\label{eq:phistarequiv}
    \phi_{k+1}^* &= (1-\alpha_k)\left(\phi_k^* + \alpha_k\tfrac{\mu}2\tnorm{v_{k}-y_k^{++}}\right) +\alpha_k( f(y_k) - \tfrac{1}{2\mu}\tnorm{\nabla f(y_k)}).
\end{align}
\end{lemma}

The following Lemma shows that $\phi_k$ is a global lower bound for $f$.
\begin{lemma}\label{lem:lb}
Let $\alpha_k \in (0,1)$ $\forall k$. Then, for all $x\in\R^n$, $\phi_k(x) \leq F(x)$ for all $k\geq 0$.
\end{lemma}

\subsection{Geometry of the lower bounds}
Here we reproduce the example in Section 2 of \cite{Drusvyatskiy16} to show how the lower bounds are combined to generate new quadratic lower bounds. All parameters used are identical to those in \cite{Drusvyatskiy16}, but using the notation of this paper. The left plot shows the quadratics $\phi_k(x) = 1+0.5(x+2)^2$ and $\phi(x;x_k) = 3+0.5(x-4)^2$, and the resulting quadratics generated by taking convex combinations as described previously.
\begin{figure}[H]\centering
\includegraphics[width=5.8cm]{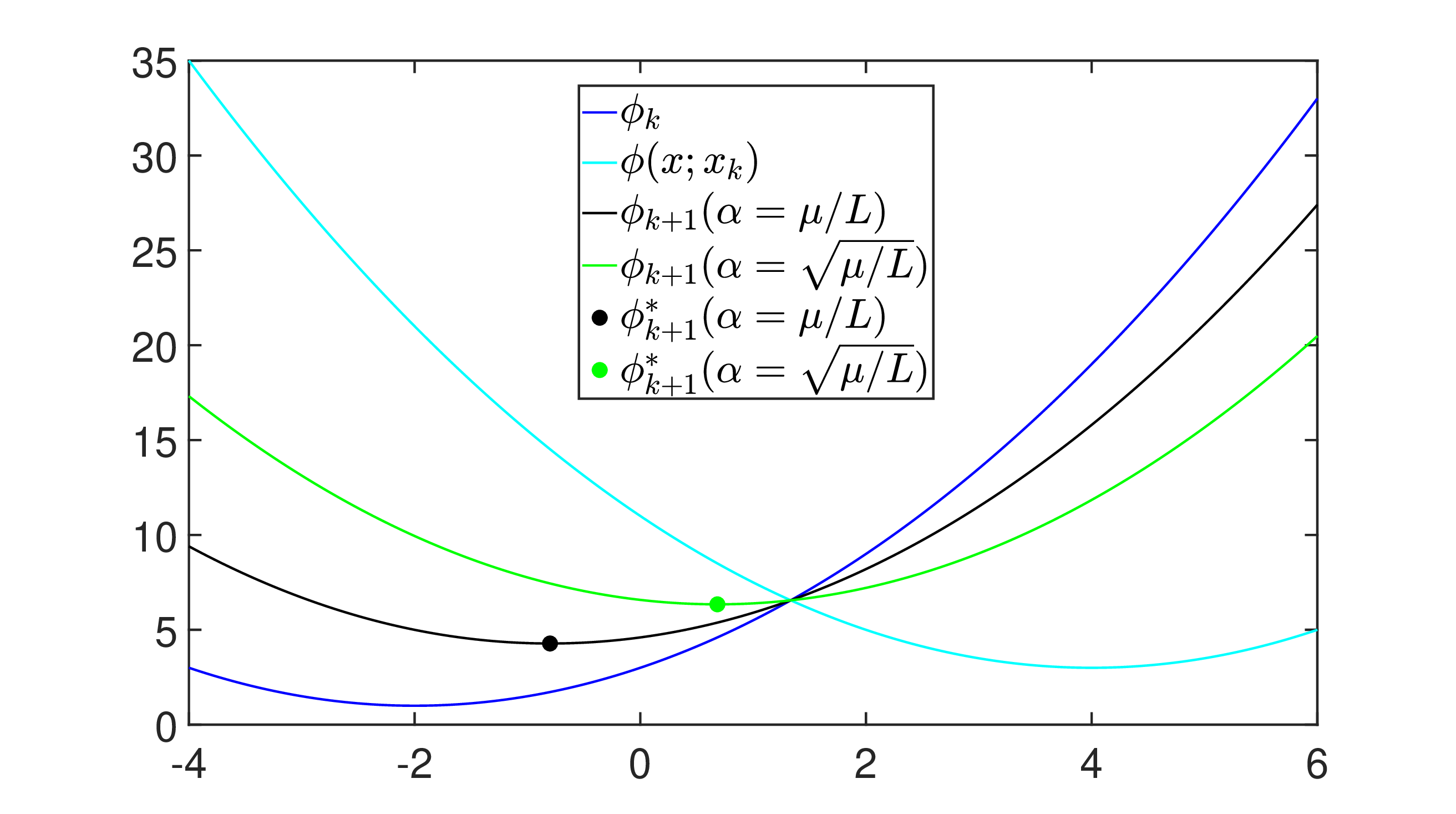}
\includegraphics[width=5.8cm]{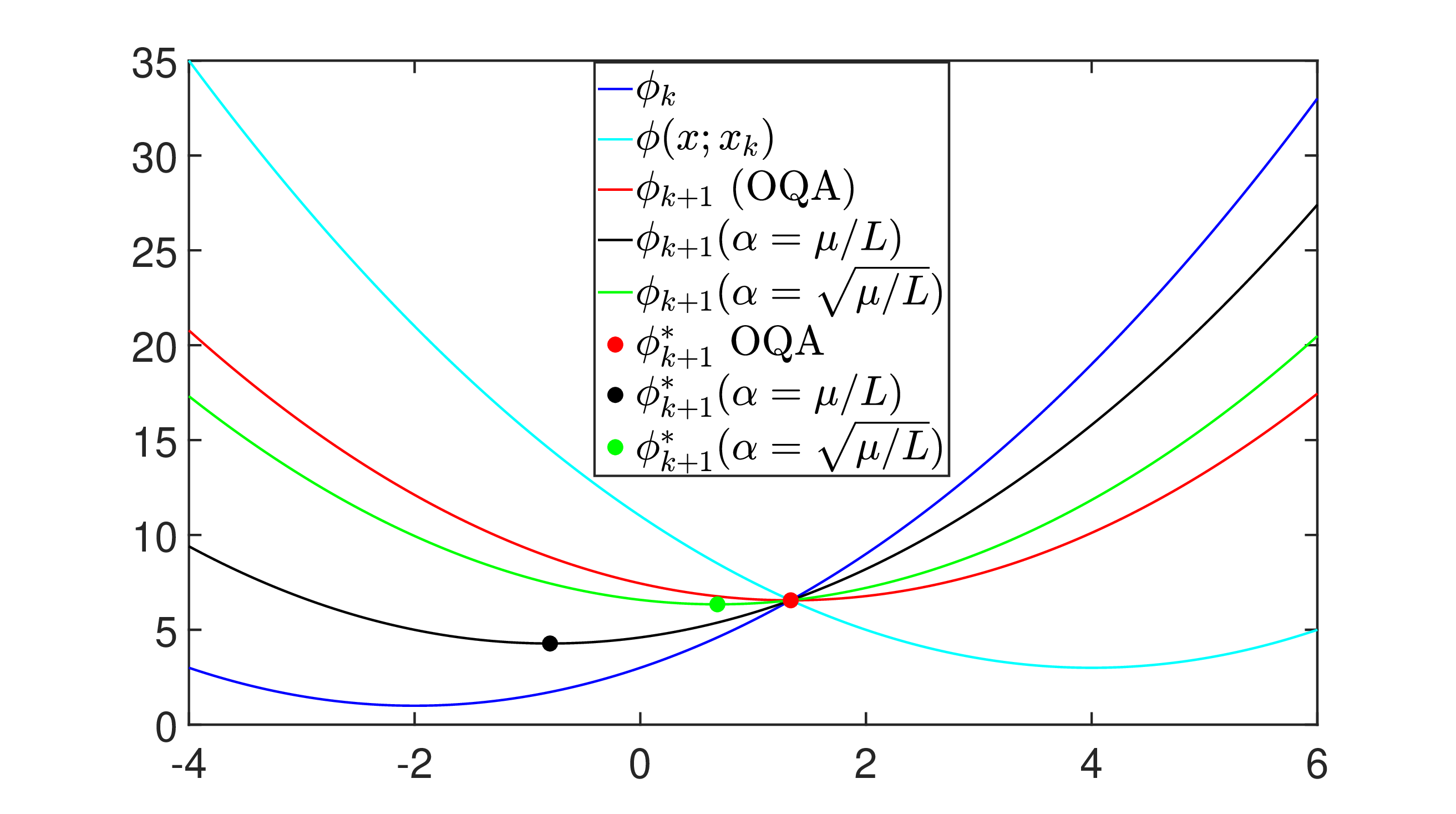}
\caption{Left: The quadratic lower bounds $\phi_k(x) = 1+0.5(x+2)^2$ and $\phi(x;x_k) = 3+0.5(x-4)^2$. The new lower bounds $\phi_{k+1}(x)$ for $\alpha = \mu/L=0.2$ (corresponding to the unaccelerated algorithm in this work) and $\phi_{k+1}(x)$ for $\alpha = \sqrt{\mu/L}=\sqrt{0.2}$ (corresponding to the accelerated algorithm in this work). Right: This is the same plot as the left but it also shows the optimal quadratic average (corresponding to the OQA algorithm) from \cite{Drusvyatskiy16}. Finally, the minimizer of each quadratic model is indicated by the round marker.}
\label{fig:QLB}
\end{figure}

\section{Algorithms and Convergence Guarantees for Composite Functions}\label{sec:composite}

The purpose of this section is to demonstrate that the UES framework, and the previously presented lower bounds, are \emph{useable} definitions that give rise to \emph{efficient implementable algorithms}. Throughout this section we consider composite optimization problems (problems of the form \eqref{eq:Problem} with $h \neq 0$) and, as for all results in this work, we suppose that Assumption~\ref{A_SCL} holds.

We present two algorithms whose iterates fit the Underestimate Sequence framework described in Section~\ref{sec:UES}, and use the lower bounds developed in Section~\ref{s:lbcomposite}. Both algorithms fit the composite setting very naturally; the first algorithm is a proximal gradient descent type method, while the second algorithm is an accelerated proximal gradient variant. Both algorithms incorporate verifiable stopping conditions, and convergence guarantees are established via the UES framework.

\subsection{A Composite Underestimate Sequence Algorithm}

We now present an algorithm to solve \eqref{eq:Problem}, which is based on the UES framework. A brief description will follow.

\begin{algorithm}
	\caption{Composite UES Algorithm (CUESA)}
	\label{alg:CUESA}
	\begin{algorithmic}[1]
		\STATE Initialization: Set $k=0$, $\epsilon >0$, initial point $x_0\in\R^n$ and compute $\mu$, $L$.
		\STATE Set $v_0$ and $\varphi_0^*$ as in \eqref{eq:minvalminmize} and let $\alpha_k= \frac{\mu}{L}$.
		\WHILE {$F(x_k) - \varphi_k^* > \epsilon$}
        \STATE Set $y_k = x_k$, $y_k^+ = x_k^+$, and $y_k^{++} = x_k^{++}$
		\STATE Set $x_{k+1} = x_k - \tfrac{1}{L} G_L (x_k),$
		\STATE Update $v_{k+1}$ and $\varphi_{k+1}^*$ as in \eqref{eq:vkcomp} and \eqref{eq:phikcomp} respectively.
		\STATE $k=k+1$.
		\ENDWHILE
	\end{algorithmic}
\end{algorithm}

The Composite (functions) UnderEstimate Sequence Algorithm (CUESA) presented in Algorithm~\ref{alg:CUESA} solves problem~\eqref{eq:Problem} when $h \neq 0$. The algorithm uses a fixed step size $\alpha_k = \tfrac{\mu}{L}$, and note that $y_k$ is not explicitly used in CUESA ($y_k = x_k$ for all $k\geq 0$). Moreover, the lower bound $\varphi_{k+1}(x)$ is not explicitly constructed; the algorithm simply computes and uses the minimizer of the lower bound $\varphi_{k+1}^*$.

Considering Step~5 in isolation shows that each iteration of CUESA is simply a proximal gradient descent step. However, CUESA is distinct from a standard proximal gradient method due to the inclusion of additional ingredients related to the lower bound $\varphi_k(x)$, which guarantee an $\epsilon$-optimal solution.

In addition to the proximal gradient, the algorithm utilizes two vectors at every iteration, $x_k$ and $v_k$. By \eqref{eq:phikcomp}, no additional vectors are required to update $\varphi_{k}^*$, but the function value $F(x_k^+) \equiv F(x_{k+1})$ is required. Overall, at every iteration of CUESA, three vectors must be stored, and one function value must be computed.

\begin{theorem}\label{thm3}
  Let Assumption~\ref{A_SCL} hold. The sequences $\{x_k\}_{k=0}^\infty$, $\{\varphi_k(x)\}_{k=0}^\infty$ and $\{\alpha_k\}_{k=0}^\infty$ generated by CUESA (Algorithm~\ref{alg:CUESA}) form a UES.
\end{theorem}
\begin{proof}
It must be shown that the iterates generated by Algorithm~\ref{alg:CUESA} satisfy the conditions of Definition~\ref{def:UES}. Because $\alpha_k = \mu/L \in (0,1)$ for all $k \geq 0$, by Lemma~\ref{lem:varphilessthanF}, \eqref{eq:def1} holds. It remains to prove \eqref{eq:def2}. From Step~4 in CUESA, one sees that $y_k = x_k$ for all $k$, so it also follows that $y_k^+ = x_{k+1}$ for all $k$. Using $y = x= x_k$ in the lower bound \eqref{eq:lbcomposite} gives
\begin{eqnarray}\label{eq:compFreduce}
  F(x_{k+1}) \leq F(x_k) - \tfrac1{2L} \tnorm{G_L(x_k)}.
\end{eqnarray}
Thus,
\begin{align*}
&F(x_{k+1})-\varphi_{k+1}^*\\
&=(1-\alpha_k) F(x_{k+1}) +\alpha_k F(x_{k+1}) -\varphi_{k+1}^*\\
&\overset{\eqref{eq:varphistarequiv}}{=} (1-\alpha_k) F(x_{k+1}) +\alpha_k F(x_{k+1}) -
    (1-\alpha_k)\left(\varphi_k^* + \alpha_k\tfrac{\mu}2\tnorm{v_{k}-y_k^{++}}\right)\\
    &\qquad-\alpha_k\left(F(y_k^+) +  \left(\tfrac1{2L} -\tfrac1{2\mu} \right)\tnorm{G_L(y_k)}\right)\\
&= (1-\alpha_k) F(x_{k+1})  -
    (1-\alpha_k)\left(\varphi_k^* + \alpha_k\tfrac{\mu}2\tnorm{v_{k}-y_k^{++}}\right)\\
    &\qquad-\alpha_k\left(\tfrac1{2L} -\tfrac1{2\mu} \right)\tnorm{G_L(y_k)}\\
    &\leq (1-\alpha_k) \left(F(x_{k+1}) - \varphi_k^*\right)-\alpha_k\left(\tfrac1{2L} -\tfrac1{2\mu} \right)\tnorm{G_L(y_k)}
    \\
&\overset{\eqref{eq:compFreduce}}{\leq} 
(1-\alpha_k)\left( F(x_{k}) - \tfrac1{2L} \tnorm{G_L(x_k)} - \phi_k^*\right)- \alpha_k\left(\tfrac1{2L}- \tfrac{1}{2\mu}\right)\tnorm{G_L(x_k)}\\
&\leq 
(1-\alpha_k)
\left( F(x_{k})- \phi_k^*\right)+ \left(
  \tfrac{\alpha_k}{2\mu} -\tfrac{\alpha_k}{2L}
 -(1-\alpha_k) \tfrac1{2L}
  \right) \tnorm{ G_L(x_k)}\\
&\leq (1-\alpha_k)(F(x_k)-\varphi_k^*),
\end{align*}
where the last step follows because $\alpha_k = \tfrac{\mu}{L}$ so
$\tfrac{\alpha_k}{2\mu} -\tfrac1{2L}=  0.$\qed
\end{proof}
Theorem~\ref{thm3} confirms that the iterates generated by Algorithm~\ref{alg:CUESA} form a UES. Recalling Proposition~\ref{prop:UES}, and noting that $\sum_{k=0}^\infty \alpha_k = \infty$ because $\alpha_k = \mu/L$ for all $k\geq 0$, CUESA is guaranteed to converge linearly to the solution of \eqref{eq:Problem}, as is stated in the following corollary.
\begin{corollary}\label{coro3}
Let Assumption~\ref{A_SCL} hold. Then, the sequence of iterates $\{x_k\}_{k\geq0}$ generated by Algorithm \ref{alg:CUESA} exhibits a linear rate of convergence
	\begin{align*}
		 F(x_k) -\varphi_k^* \leq \left(1- \tfrac{\mu}{L}\right)^k (F(x_0) -\varphi_0^*).
	\end{align*}
\end{corollary}

\subsection{An Accelerated Composite UES Algorithm}

CUESA has a linear rate of convergence, but the rate constant is suboptimal. Thus, the Accelerated Composite UnderEstimate Sequence Algorithm (ACUESA) for solving \eqref{eq:Problem} when $h \neq 0$ is presented in Algorithm~\ref{alg:ACUESA}, and ACUESA achieves the best possible rate of convergence.
\begin{algorithm}[h!]
	\caption{Accelerated Composite UES Algorithm (ACUESA)}
	\label{alg:ACUESA}
	\begin{algorithmic}[1]
		\STATE Initialization: Set $k=0$, $\epsilon >0$, initial point $x_0\in\R^n$ and compute $\mu$, $L$.
		\STATE Set $v_0$ and $\varphi_0^*$ as in \eqref{eq:minvalminmize}.	
                Let $\alpha_k= \sqrt{\frac{\mu}{L}}$, $\beta_k = \frac{1}{1+\alpha_k}$.
		\WHILE {$F(x_k) - \varphi_k^* > \epsilon$}
		\STATE Set $y_k = \beta_k x_k + (1- \beta_k) v_k$.
		\STATE Set $x_{k+1} = y_k - \frac{1}{L}  G_L(y_k).$	
		\STATE Update $v_{k+1}$ and $\varphi_{k+1}^*$ as in \eqref{eq:vkcomp} and \eqref{eq:phikcomp} respectively.
		\STATE $k=k+1$.
		\ENDWHILE
	\end{algorithmic}
\end{algorithm}

ACUESA uses a fixed step size $\alpha_k = \sqrt{\mu/L}$ and a fixed acceleration parameter $\beta_k = \frac{1}{1+\alpha_k}$. ACUESA explicitly uses a point $y_k (\neq x_k)$, which is a convex combination of the points $x_k$ and $v_k$. At each iteration, a proximal gradient descent step is taken \emph{from} $y_k$, with the proximal gradient computed at $y_k$. In addition to the proximal gradient, ACUESA constructs \emph{three} points at every iteration, $x_k$, $v_k$ and $y_k$. Thus, at every iteration of ACUESA, four $n$-dimensional vectors must be stored, and one function value must be computed.

\begin{theorem}\label{thm4}
	Let Assumption~\ref{A_SCL} hold. The sequences $\{x_k\}_{k=0}^\infty$, $\{\varphi_k(x)\}_{k=0}^\infty$ and $\{\alpha_k\}_{k=0}^\infty$ generated by ACUESA (Algorithm~\ref{alg:ACUESA}) form a UES.
\end{theorem}
\begin{proof}
From Step~5 in ACUESA,
\begin{equation}\label{eq:xkp1yp}
  x_{k+1} = y_k - \tfrac1LG_L(y_k) \overset{\eqref{eq:shortproxstep}}{\equiv} y_k^+.
\end{equation}
Hence,
  \begin{eqnarray}
    F(x_{k+1}) - \varphi_{k+1}^* &=& (1-\alpha_k)F(x_{k+1}) + \alpha_k F(x_{k+1})  - \varphi_{k+1}^*\notag\\
    &\overset{\eqref{eq:varphistarequiv}}{=}&(1-\alpha_k)F(x_{k+1}) + \alpha_k F(x_{k+1}) - (1-\alpha_k)\varphi_k^* -  \alpha_kF(y_k^+)\notag\\
    && -\alpha_k(1-\alpha_k)\tfrac{\mu}2\tnorm{v_{k} - y_k^{++}} - \alpha_k\left(\tfrac1{2L} - \tfrac1{2\mu} \right)\tnorm{G_L(y_k)}\notag\\
     &\overset{\eqref{eq:xkp1yp}}{=}&(1-\alpha_k)F(x_{k+1})- (1-\alpha_k)\varphi_k^*\notag\\
     && -\alpha_k(1-\alpha_k)\tfrac{\mu}2\tnorm{v_{k} - y_k^{++}}+\left(\tfrac{\alpha_k}{2\mu} -\tfrac{\alpha_k}{2L}\right)\tnorm{G_L(y_k)}\notag\\
     &=&(1-\alpha_k)\left(F(x_{k+1})-\varphi_k^* -\alpha_k\tfrac{\mu}2\tnorm{v_{k} - y_k^{++}}\right)\notag\\
     &&+\left(\tfrac{\alpha_k}{2\mu} -\tfrac{\alpha_k}{2L}\right)\tnorm{G_L(y_k)}. \label{eq:Fkp11}
  \end{eqnarray}
  Rearranging the update for $y_k$ in Step~4 of Algorithm~\ref{alg:ACUESA} gives
  \begin{equation}\label{eq:vk}
    v_k = \tfrac1{1-\beta_k}(y_k - \beta_k x_k) = y_k + \tfrac{\beta_k}{1-\beta_k}(y_k - x_k),
  \end{equation}
and because $\beta_k = \tfrac{1}{1+\alpha_k}$,
  \begin{eqnarray}\label{eq:alphabeta1}
    1-\beta_k =  1-\tfrac{1}{1+\alpha_k} = \tfrac{\alpha_k}{1+\alpha_k} =\alpha_k\beta_k \quad \Rightarrow \quad \tfrac{\alpha_k\beta_k}{1-\beta_k}= 1.
  \end{eqnarray}
Thus, combining \eqref{eq:vk} and \eqref{eq:longproxstep} gives
  \begin{align}
    &\tfrac{\alpha_k\mu}2\tnorm{v_k - y_k^{++}}\notag \\
    &= \tfrac{\alpha_k\mu}2\tnorm{\tfrac{\beta_k}{1-\beta_k}(y_k - x_k) + \tfrac1{\mu}G_L(y_k)}\notag\\
    &= \tfrac{\alpha_k\mu}2\tfrac{\beta_k^2}{(1-\beta_k)^2}\tnorm{x_k-y_k} + \tfrac{\alpha_k}{2\mu}\tnorm{G_L(y_k)}
    - \tfrac{\alpha_k\beta_k}{1-\beta_k}\langle G_L(y_k),x_k-y_k\rangle\notag\\
    &\overset{\eqref{eq:alphabeta1}}{=} \tfrac{\mu}2\tfrac{\beta_k}{1-\beta_k}\tnorm{x_k-y_k} + \tfrac{\alpha_k}{2\mu}\tnorm{G_L(y_k)} - \langle G_L(y_k),x_k-y_k\rangle.\label{eq:vkyknorm}
  \end{align}
  Substituting \eqref{eq:vkyknorm} into \eqref{eq:Fkp11} results in
  \begin{eqnarray*}
    && \hspace{-5mm} F(x_{k+1}) - \varphi_{k+1}^*\\
    &=& (1-\alpha_k)\left(F(x_{k+1})-\varphi_k^* -\tfrac{\mu}2\tfrac{\beta_k}{1-\beta_k}\tnorm{x_k-y_k}  + \langle G_L(y_k),x_k-y_k\rangle \right)\\
     && -(1-\alpha_k)\tfrac{\alpha_k}{2\mu}\tnorm{G_L(y_k)}+\left(\tfrac{\alpha_k}{2\mu} -\tfrac{\alpha_k}{2L}\right)\tnorm{G_L(y_k)}\\
     &=&(1-\alpha_k)\left(F(x_{k+1})-\varphi_k^* -\tfrac{\mu}2\tfrac{\beta_k}{1-\beta_k}\tnorm{x_k-y_k} + \langle G_L(y_k),x_k-y_k\rangle \right)\\
     && + (1-\alpha_k)\tfrac{1}{2L}\tnorm{G_L(y_k)}.
  \end{eqnarray*}
  \textcolor{black}{The second equality uses the fact that $\alpha_k =\sqrt{\mu/L}$.} 
 Using a rearrangement of the lower bound \eqref{eq:lbcomposite}, and \eqref{eq:xkp1yp}, gives
  \begin{eqnarray*}
    && \hspace{-5mm}F(x_{k+1}) - \varphi_{k+1}^*\\
     &\leq&(1-\alpha_k)\left(F(x_k)- \langle G_L(y_k),x_k-y_k\rangle - \tfrac{\mu}2\tnorm{x_k-y_k} - \tfrac1{2L}\tnorm{G_L(y_k)} -\varphi_k^*\right) \\
     && +(1-\alpha_k)\left(-\tfrac{\mu}2\tfrac{\beta_k}{1-\beta_k}\tnorm{x_k-y_k}  + \langle G_L(y_k),x_k-y_k\rangle  + \tfrac{1}{2L}\tnorm{G_L(y_k)}\right)\\
     &=&(1-\alpha_k)(F(x_k) -\varphi_k^*) +(1-\alpha_k)\left( - \tfrac{\mu}2\tnorm{x_k-y_k} -\tfrac{\mu}2\tfrac{\beta_k}{1-\beta_k}\tnorm{x_k-y_k}\right)\\
     &\leq&(1-\alpha_k)(F(x_k) -\varphi_k^*)  - \tfrac{\mu}2(1-\alpha_k)\tfrac{1}{1-\beta_k}\tnorm{x_k-y_k}\\
     &\leq&(1-\alpha_k)(F(x_k) -\varphi_k^*).
  \end{eqnarray*}
  Thus, the iterates generated by ACUESA form a UES.\qed
\end{proof}

Theorem~\ref{thm4} confirms that the iterates generated by Algorithm~\ref{alg:ACUESA} form a UES, and because $\alpha_k = \sqrt{\mu/L}$, ACUESA is guaranteed to converge linearly at the optimal rate to the solution of problem \eqref{eq:Problem}, as is shown in the following corollary.

\begin{corollary}\label{coro4}
Let Assumption~\ref{A_SCL} hold. Then, the sequence of iterates $\{x_k\}_{k\geq0}$ generated by Algorithm \ref{alg:ACUESA} exhibits the optimal linear rate of convergence
	\begin{align*}
		 F(x_k) -\varphi_k^* \leq \left(1- \sqrt{\tfrac{\mu}{L}}\right)^k (F(x_0) -\varphi_0^*).
	\end{align*}
\end{corollary}

The difference in convergence rates between Algorithms~\ref{alg:CUESA}~and~\ref{alg:ACUESA} is essentially explained by the quadratic term $\tnorm{v_k -y_k^{++}}$, which is entirely ignored in the proof of Theorem~\ref{thm3}. In the proof of Theorem~\ref{thm4}, one is able to incorporate an additional term containing $\tnorm{G_L(y_k)}$, which leads to a larger allowable value of $\alpha_k$, and ultimately, a tighter bound for Algorithm~\ref{alg:ACUESA}.

\section{Algorithms and Convergence Guarantees for Smooth Functions}\label{sec:smooth}

Here smooth optimization problems (problem \eqref{eq:Problem} with $h\equiv 0$) are considered, and two gradient descent type algorithms whose iterates form a UES are presented. The algorithms are similar to those for composite problems, but they use the lower bounds developed in Section~\ref{s:lbsmooth}.

\subsection{UES Algorithms for Smooth Functions}
\begin{algorithm}[h!]
	\caption{Smooth Underestimate Sequence Algorithm (SUESA)}
	\label{alg:SUESA}
	\begin{algorithmic}[1]
        \STATE Initialization: Set $k=0$, $\epsilon >0$, initial point $x_0\in\R^n$ and compute $\mu$, $L$.
		\STATE Set $v_0$ and $\phi_0^*$ as in \eqref{eq:c0v0}, and let $\alpha_k= \frac{\mu}{L}$.
		\WHILE {$f(x_k) - \phi_k^* > \epsilon$}
        \STATE Set $y_k = x_k$, $y_k^{+} = x_k^{+}$ and $y_k^{++} = x_k^{++}$.
		\STATE Set $x_{k+1} = x_k - \frac{1}{L} \nabla f(x_k)$.
        \STATE Update $v_{k+1}$ and $\phi_{k+1}^*$ as in \eqref{eq:vkcomp} and \eqref{eq:phistarequiv}, respectively.
		\STATE $k=k+1$.
		\ENDWHILE
	\end{algorithmic}
\end{algorithm}
\begin{algorithm}
	\caption{Accelerated Smooth Underestimate Sequence Algorithm (ASUESA)}
	\label{alg:ASUESA}
	\begin{algorithmic}[1]
		\STATE Initialization: Set $k=0$, $\epsilon >0$, initial point $x_0\in\R^n$ and compute $\mu$, $L$.
\STATE Set $v_0$ and $\phi_0^*$ as in \eqref{eq:c0v0}. Let $\alpha_k= \sqrt{\frac{\mu}{L}}$ and $\beta_k = \frac{1}{1+\alpha_k}$.	
		\WHILE {$f(x_k) - \phi_k^* > \epsilon$}
		\STATE Set $y_k = \beta_k x_k + (1- \beta_k) v_k$.
		\STATE Set $x_{k+1} = y_k - \frac{1}{L} \nabla f(y_k).$
		\STATE Update $v_{k+1}$ and $\phi_{k+1}^*$ as in \eqref{eq:vkcomp} and \eqref{eq:phistarequiv}, respectively.
		\STATE $k=k+1$.
		\ENDWHILE
	\end{algorithmic}
\end{algorithm}

SUESA (Algorithm~\ref{alg:SUESA}) is similar to CUESA, and ASUESA (Algorithm~\ref{alg:ASUESA}) is similar to ACUESA. In the smooth case the proximal gradient is simply the gradient so that the main update in SUESA (see Step~5) is a gradient descent step. The minimizer of the lower bound $\phi_{k+1}^*$ is updated in place of $\varphi_{k+1}^*$ --- this variation means that SUESA is \emph{not} an instance of CUESA when $h = 0$, nor is ASUESA and instance of ACUESA (see Section~\ref{sec:compare} for further details).

\begin{theorem}\label{thm1}
	Let Assumption~\ref{A_SCL} hold. The sequences $\{x_k\}_{k=0}^\infty$, $\{\phi_k(x)\}_{k=0}^\infty$ and $\{\alpha_k\}_{k=0}^\infty$ generated by SUESA (Algorithm~\ref{alg:SUESA}) form a UES.
\end{theorem}
\begin{proof}
By Lemma~\ref{lem:lb}, \eqref{eq:def1} holds. To prove \eqref{eq:def2}, combining the definition of $x_{k+1}$ (Step~5 in Algorithm~\ref{alg:SUESA}) with \eqref{eq:ass2}, and noting that $x_k=y_k$, gives
	\begin{eqnarray}\label{eq:zzzz3}
		f(x_{k+1}) = f(y_k^+) &\leq& f(y_k)  - \tfrac{1}{2L} \tnorm{\nabla f(y_k)}.
	\end{eqnarray}
It is straightforward to establish that $f(x_{k+1}) - \phi_{k+1}^*\leq (1- \alpha_k) (f(x_k) -\phi_k^*)$ using \eqref{eq:phistarequiv} and \eqref{eq:zzzz3}, and noting that $\alpha_k = \mu/L\in (0,1)$ for all $k\geq0$.
\end{proof}
\begin{corollary}\label{coro1}
Let Assumption~\ref{A_SCL} hold.  Then the sequences $\{x_k\}_{k=0}^\infty$, $\{\phi_k(x)\}_{k=0}^\infty$ and $\{\alpha_k\}_{k=0}^\infty$ generated by SUESA (Algorithm~\ref{alg:SUESA}) form a UES, so SUESA converges at a linear rate: $f(x_k) -\phi_k^* \leq (1- \tfrac{\mu}{L})^k (f(x_0) -\phi_0^*).$
\end{corollary}

\begin{theorem}\label{thm2}
Let Assumption~\ref{A_SCL} hold. The series of sequences $\{x_k\}_{k=0}^\infty$, $\{\phi_k(x)\}_{k=0}^\infty$ and $\{\alpha_k\}_{k=0}^\infty$ generated by ASUESA in Algorithm~\ref{alg:ASUESA} form a UES.
\end{theorem}
\begin{proof}
Note that, rearranging the expression for $y_k$ in Step~4 of Algorithm~\ref{alg:ASUESA} gives \eqref{eq:vk} and because $\beta_k = \tfrac{1}{1+\alpha_k}$, \eqref{eq:alphabeta1} holds. Thus, by the convexity of $f$ we have
	\begin{eqnarray}\label{eq:pqr}
		-\alpha_k\langle\nabla f(y_k),v_k -y_k\rangle \notag
		&\overset{\eqref{eq:vk}}{=}& -\alpha_k\langle\nabla f(y_k),y_k + \tfrac{\beta_k}{1- \beta_k} (y_k-x_k) -y_k\rangle \notag\\
		&\leq& \tfrac{\alpha_k\beta_k}{1- \beta_k} \big(f(x_k) - f(y_k)  \big) \overset{\eqref{eq:alphabeta1}}{=} f(x_k) - f(y_k) .
	\end{eqnarray}
To confirm that $f(x_{k+1})-\phi_{k+1}^*\leq (1-\alpha_k) (f(x_k)-\phi_k^*)$, combine \eqref{eq:phistarequiv} and \eqref{eq:zzzz3}, complete the square, apply \eqref{eq:pqr} and note that $\alpha_k = \sqrt{\frac{\mu}{L}}$.\qed	
\end{proof}

\begin{corollary}\label{coro2}
Let Assumption~\ref{A_SCL} hold.  Then the sequences $\{x_k\}_{k=0}^\infty$, $\{\phi_k(x)\}_{k=0}^\infty$ and $\{\alpha_k\}_{k=0}^\infty$ generated by ASUESA (Algorithm~\ref{alg:ASUESA}) form a UES, so ASUESA converges at the optimal linear rate: $f(x_k) -\phi_k^* \leq (1- \sqrt{\tfrac{\mu}{L}})^k (f(x_0) -\phi_0^*).$
\end{corollary}

\subsection{A Toy Example}\label{sec:toyeg}


Notice that for SUESA, at iteration $k\geq 0$, the distance between $x_k$ and the minimizer of the lower bound shrinks at a fixed rate. That is, since $\alpha_k = \tfrac{\mu}{L}$, the following equality holds:
\begin{align}\label{eq:xkvk}
	x_{k+1} - v_{k+1} &= \left(x_k - \tfrac{1}{L} \nabla f(x_k)\right) - \left((1- \alpha_k)v_k + \alpha_k (x_k - \tfrac{1}{\mu} \nabla f(x_k) )\right) \notag\\
	& = (1- \tfrac{\mu}{L}) (x_k - v_k).
\end{align}
Equation \eqref{eq:xkvk} illustrates that, after each iteration of Algorithm~\ref{alg:SUESA}, the line joining $x_{k+1}$ and $v_{k+1}$ is parallel to the line joining $x_k$ and $v_k$. Moreover, the distance between the two points is reduced by precisely $(1- \tfrac{\mu}{L})$ at every iteration. Intuitively, $x_k$ and the minimizer $v_k$ are becoming ever closer, and eventually they both converge to $x^*$ (recall Theorem~\ref{thm1}).

This fact can be visualized via the following toy example. Consider the (smooth) regularized logistic regression problem, i.e., problem \eqref{eq:Problem} with $h=0$ and
\begin{align*}
	f(x) = \sum_{i=1}^m \log(1+\exp(-y_i \ve{x}{a_i})) + \tfrac \lambda 2 \tnorm{x},
\end{align*}
where $a_i \in \R^n$ is the $i$th feature vector with corresponding (binary) label $y_i\in \R$.
For this example, 100 two dimensional data points with binary labels $\{a_i, y_i\}$ were randomly generated (so $m=100$ and $n=2$), and these points and labels are shown in the top left plot in Figure~\ref{fig:toyexample}. (Each point $a_i$ is plotted on a 2D grid, and the point is colored green or red to highlight its label $y_i$). Parameter $\lambda= 0.01$, so the strong convexity constant is $\mu=0.01$. SUESA and ASUESA were used to solve this problem, starting from the point $x_0 = (-20,10)^T$. Both algorithms were stopped after 100 iterations and the results of this experiment are shown in Figure~\ref{fig:toyexample}. The optimal solution is $x^* = (-0.2610,0.2407)^T$ with corresponding optimal value $F^* \equiv f^* = 0.6779$ (all to 4dp).


\begin{figure}[H]\centering
\hspace{-4mm}
\includegraphics[scale=.26]{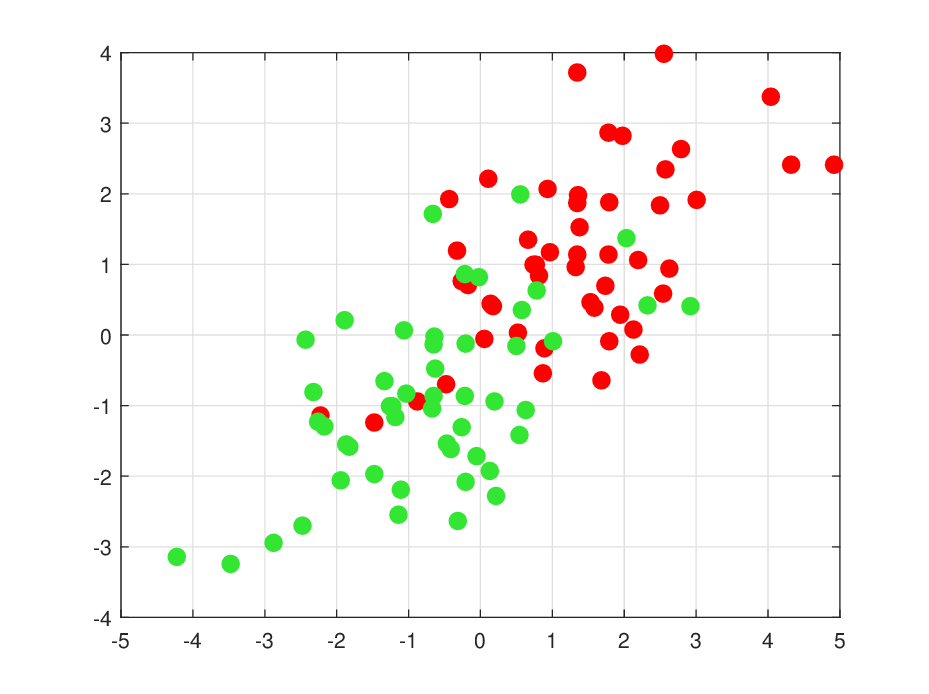}\hspace{-2mm}
\includegraphics[scale=.26]{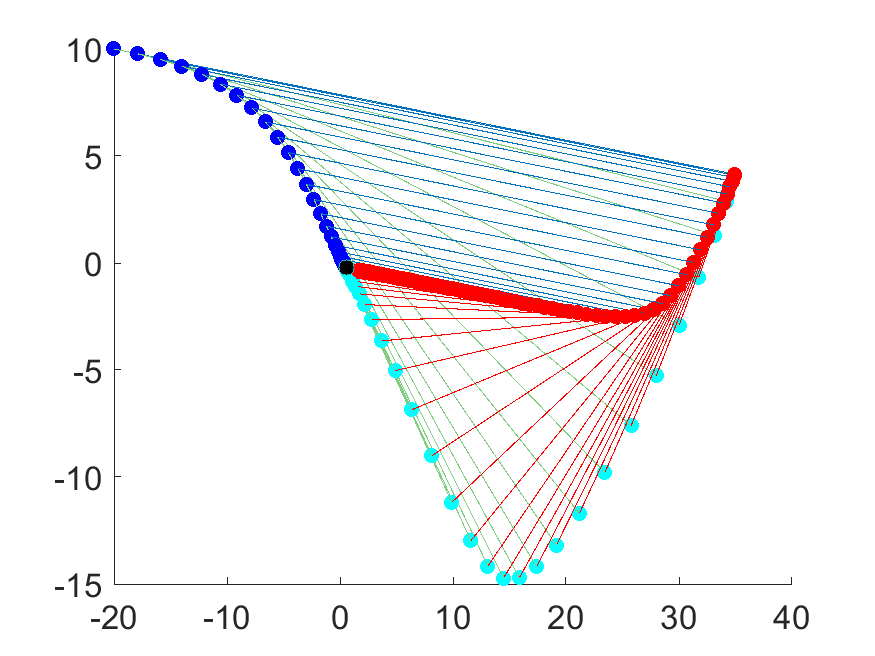}\hspace{-2mm}
\includegraphics[scale=.26]{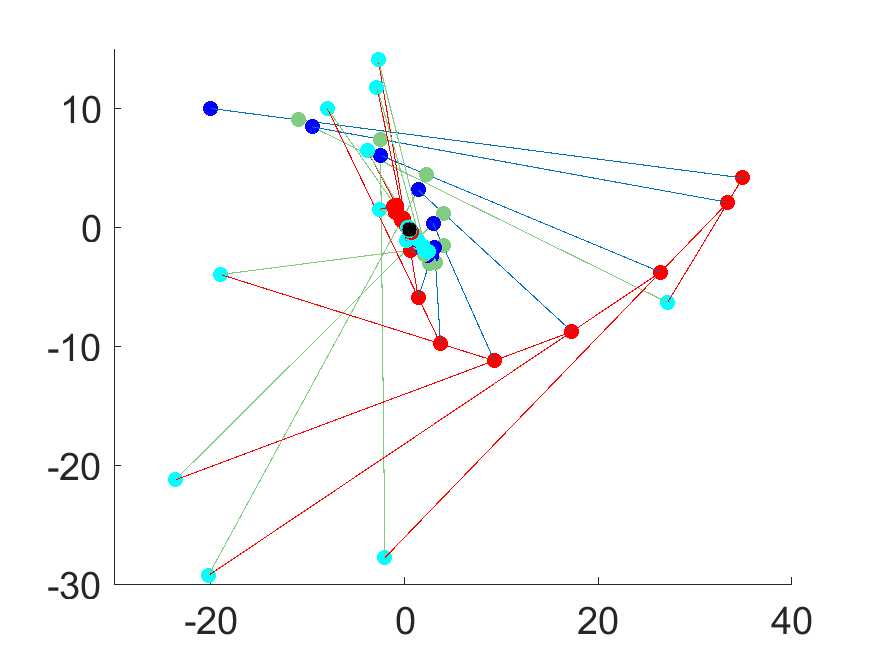}\\
\includegraphics[scale=.28]{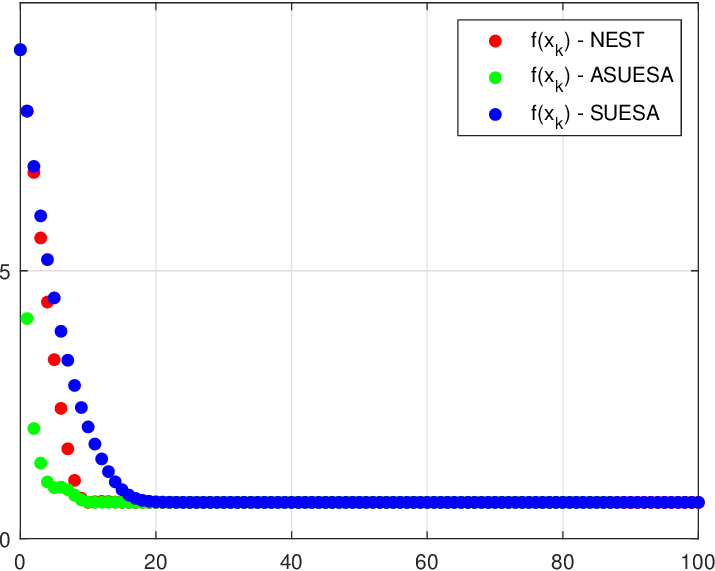}\hspace{3mm}
\includegraphics[scale=.28]{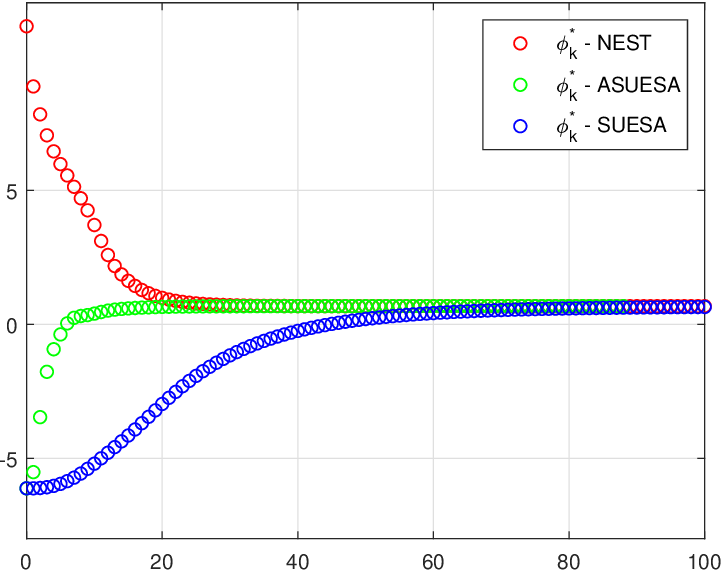}\hspace{3mm}
\includegraphics[scale=.28]{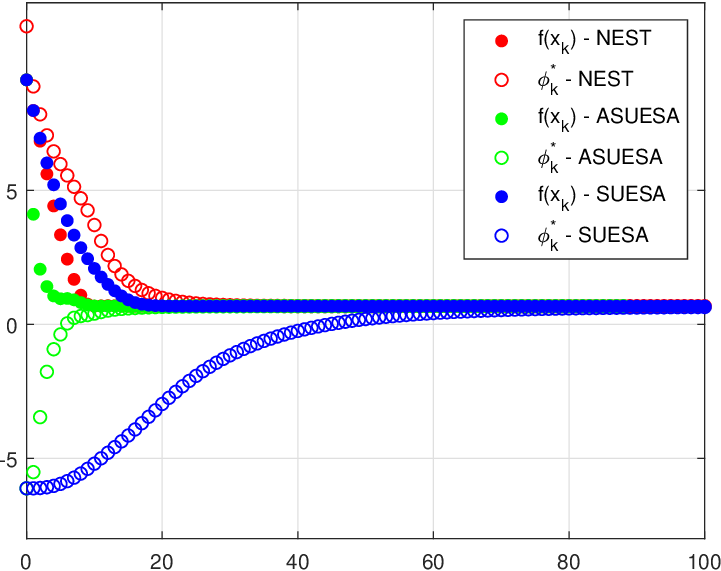}\\
\includegraphics[scale=.28]{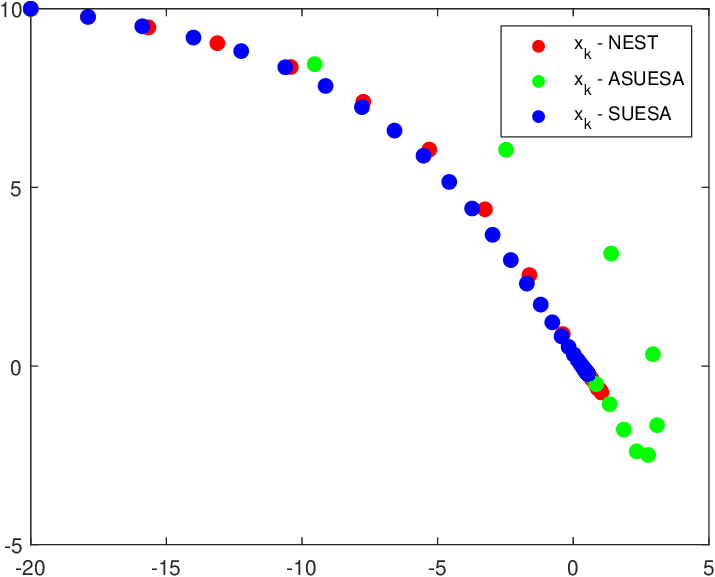}\hspace{3mm}
\includegraphics[scale=.28]{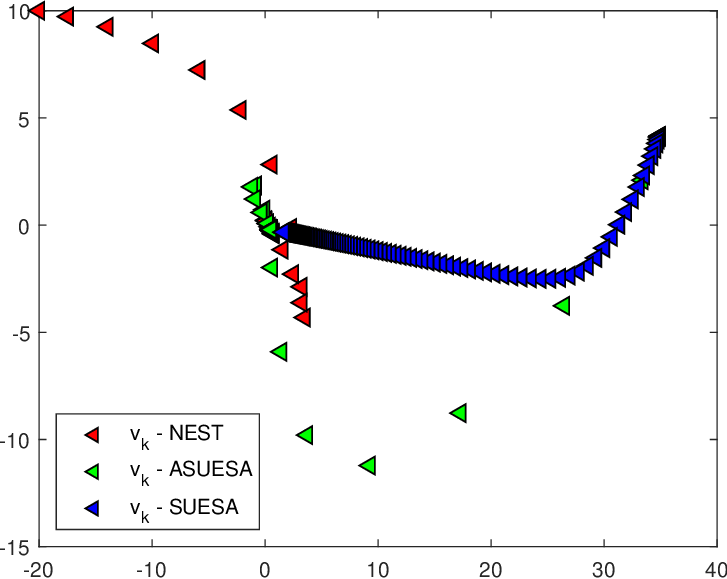}\hspace{3mm}
\includegraphics[scale=.28,trim={1cm 0 0 0},clip]{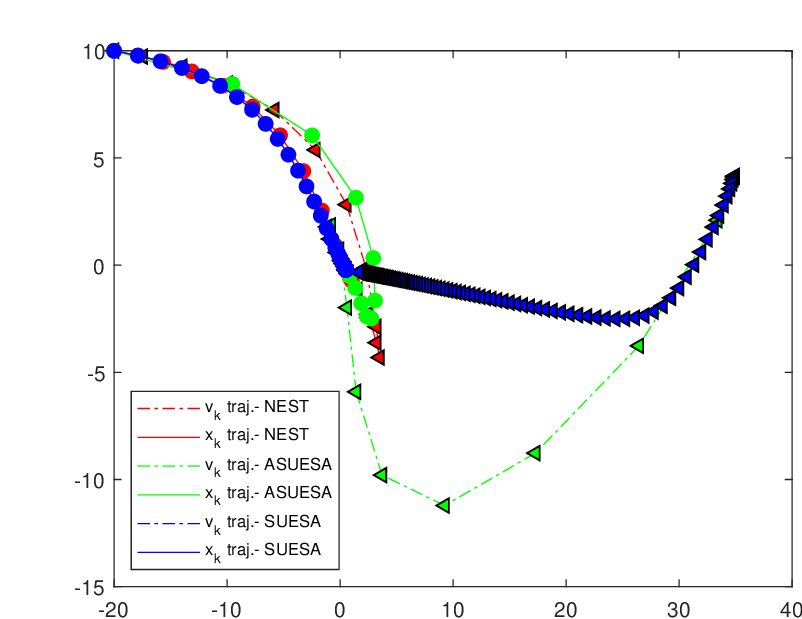}
\caption{SUESA and ASUESA applied to the toy example described in Section~\ref{sec:toyeg}. Top row from left: the simulated data; the iterates generated by SUESA; the iterates generated by ASUESA.
Middle row from left: the function values $f(x_k)$ for SUESA, ASUESA and NEST; the values $\phi_k^*$ for SUESA, ASUESA and NEST algorithm; both $f(x_k)$ and $\phi_k^*$ for SUESA, ASUESA and NEST.
Bottom row from left: the iterates $x_k$ generated by SUESA, ASUESA and NEST; the points $v_k$ generated by SUESA, ASUESA and NEST; both $x_k$ and $v_k$ for SUESA, ASUESA and NEST.}
\label{fig:toyexample}
\end{figure}

%
%
%
%
%

Consider the iterates generated by SUESA in Figure~\ref{fig:toyexample}. The blue dots represent the iterates $x_k$, starting from $x_0 = (-20,10)^T$. The aqua dots are the long steps $x_k^{++}$. Notice that the iterates (short steps) $x_k$, $x_k^{+} (\equiv x_{k+1})$ and long steps $x_k^{++}$ all lie on the same (aqua) line. This is simply because $x_k^+$ and $x_k^{++}$ are the result of starting from $x_k$ and then moving in the direction $-\nabla f(x_k)$, scaled by $1/L$ or $1/\mu$, respectively. So all the aqua lines correspond to the negative gradient directions. The red dots correspond to the points $v_k$. Initially $v_0 = x_0^{++} \approx (35,5)^T$ so the first red and aqua point coincide. The blue lines join $x_k$ and $v_k$, and they are all parallel. As SUESA progresses, the points $x_k$ (blue) and $v_k$ (red) eventually converge to the minimizer (black dot).

The last plot on the first row in Figure~\ref{fig:toyexample} shows the iterates of ASUESA. Additionally for ASUESA, the green points correspond to the iterates $y_k$, and one notices that these points lie on the line joining $x_k$ and $v_k$.


The second row of Figure~\ref{fig:toyexample} shows (left) the function values $f(x_k)$, (middle) the bounds $\phi_k^*$, and (right) both the function values $f(x_k)$ and bounds $\phi_k^*$ for SUESA, ASUESA and NEST for the toy problem. The middle plot shows that the bounds $\phi_k^*$ for SUESA and ASUESA are \emph{lower bounds for} $f(x_k)$, while the bound $\phi_k^*$ for NEST is an upper bound for $f(x_k)$. This is why a verifiable stopping condition can be used for SUESA and ASUESA (the function values decrease, the lower bounds increase, and the gap sandwiches the optimal function value $f(x^*)$, squeezing closer to it). On the other hand, for NEST, $\phi_k^*$ pushes $f(x_k)$ down from above, but provides no information about how close the current value $f(x_k)$ is from the optimal value $f(x^*)$.

The third row of Figure~\ref{fig:toyexample} shows (left) the iterates $x_k$, (middle) the points $v_k$, and (right) both the iterates $x_k$ and the points $v_k$ for SUESA, ASUESA and NEST for the toy problem. The left and middle plots clearly show that the iterates $x_k$ and points $v_k$ are different for each algorithm. \emph{ASUESA is not the same as NEST}. The trajectory of the iterates $x_k$ (and the points $v_k$) are distinct for each algorithm. This confirms that the algorithms presented here are new, and are not simply an existing algorithm with an added lower bound.

\subsection{A comparison of the composite and smooth algorithms when $h\equiv0$}\label{sec:compare}

A natural question to ask is, `Do the composite algorithms (CUESA and ACUESA) recover the smooth algorithms (SUESA and ASUESA) when $h \equiv 0$?'. Ultimately, the answer to this question is no, but this is a consequence of the termination condition (i.e., the lower bounds), rather than the path of the iterates.

Table~\ref{table:compare} compares the vectors generated by ASUESA (Algorithm~\ref{alg:ASUESA}) and ACUESA (Algorithm~\ref{alg:ACUESA}) in the smooth case for a given initial point $x_0\in \R^n$. Note that the same values $\alpha_k$ and $\beta_k$ are used by both algorithms, and for $h\equiv 0$, $G_L(y_k) \equiv \nabla f(y_k)$. Table~\ref{table:compare} confirms that the iterates generated by ASUESA and ACUESA follow the same path.

\begin{table}[h!]\centering
  \begin{tabular}{|c|c|c|c|c|c|}
  \hline
  Variable & ACUESA($h\equiv 0$)&  &ASUESA &   \\
  \hline
  $v_0$ & $x_0^{++}$ & \eqref{eq:minvalminmize}& $x_0^{++}$& \eqref{eq:c0v0}  \\
  $y_k$ &$\beta_k x_k + (1- \beta_k) v_k$ & Step~4 & $\beta_k x_k + (1- \beta_k) v_k$ & Step~4  \\
  $x_{k+1}$ & $y_k - \frac1L \nabla f(y_k)$& Step~5 & $y_k - \frac1L \nabla f(y_k)$& Step~5 \\
  $v_{k+1}$ & $(1-\alpha_k)v_k + \alpha_k y_k^{++}$& \eqref{eq:vkcomp}& $(1- \alpha_k) v_k + \alpha_k y_k^{++}$ &  \eqref{eq:vkcomp}\\
  \hline
\end{tabular}
\caption{The vectors generated by ACUESA and ASUESA when $h\equiv 0$.}
\label{table:compare}
\end{table}

However, consider the stopping gap for each algorithm, i.e., compare
\begin{eqnarray}\label{eq:Gap_ASUESA}
  f(x_{k+1}) - \phi_{k+1}^* &\overset{\eqref{eq:phistarequiv}}{=}& f(x_{k+1}) - (1-\alpha_k)\alpha_k\tfrac{\mu}2\tnorm{v_{k}-y_k^{++}}  + \tfrac{\alpha_k}{2\mu}\tnorm{\nabla f(y_k)}
    \notag\\
    && - (1-\alpha_k)\phi_k^* - \alpha_k f(y_k).
\end{eqnarray}
with
\begin{eqnarray}\label{eq:Gap_ACUESA}
    f(x_{k+1}) - \varphi_{k+1}^* &\overset{\eqref{eq:varphistarequiv}}{=}& f(x_{k+1}) - (1-\alpha_k)\alpha_k\tfrac{\mu}2\tnorm{v_{k}-y_k^{++}}  + \tfrac{\alpha_k}{2\mu}\tnorm{\nabla f(y_k)}
    \notag\\
    && - (1-\alpha_k)\varphi_k^* - \alpha_k f(y_k^+) - \alpha_k\tfrac1{2L}\tnorm{\nabla f(y_k)}.
\end{eqnarray}
At every iteration, by \eqref{eq:zzzz3} it holds that $- f(y_k)\leq -f(y_k^+) -\tfrac{\alpha_k}{2L} \tnorm{\nabla f(y_k)}$ and so we conclude that  $f(x_{k+1}) - \phi_{k+1}^* \leq f(x_{k+1}) - \varphi_{k+1}^*$ for all $k\geq 0$. This confirms that, in the smooth case with $h \equiv 0$, ACUESA does not recover ASUESA. Moreover, it is advantageous to use ASUESA because the algorithm will terminate sooner than ACUESA for $h\equiv 0$.

Similar arguments can be used to show that CUESA and SUESA are not the same algorithm in the smooth case, and that it is advantageous to use SUESA because SUESA will terminate earlier than CUESA.

\section{An algorithm with adaptive $L$}\label{sec:adaptiveL}

In all of the algorithms presented so far, the Lipschitz constant $L$ is explicitly used. However, by studying the convergence proofs for Algorithms~\ref{alg:CUESA}--\ref{alg:ASUESA} one observes that the role of $L$ is to enforce a reduction in the function value from one iteration to the next. Thus, it is natural to ask the question, `Can an \emph{adaptive} Lipschitz constant, say $L_k$, be used in place of the true Lipschitz constant $L$, while preserving convergence guarantees?'. This is discussed now.

When the Lipschitz constant $L$ is unknown, or is expensive to compute, it may be preferable to employ an `adaptive' Lipschitz constant, say $L_k$, that approximates $L$ locally. This approach, studied by Nesterov in \cite{Nesterov07,Nesterov13}, has the additional advantage that $L_k$ may be smaller than the true Lipschitz constant $L$, which can lead to large step sizes. To maintain convergence properties for Algorithms~\ref{alg:CUESA}--\ref{alg:ASUESA}, certain inequalities must hold; the relevant inequalities are as follows.

\paragraph{Non-smooth case.}
For composite functions, \eqref{eq:compFreduce} and \eqref{eq:Fkp11} must hold for CUESA and ACUESA, respectively. So, if $L_k$ satisfies
\begin{eqnarray}
\label{eq:comAdapL}
  F\left(y_k - \tfrac1{L_k} G_{L_k}(y_k)\right) \leq F(y_k) - \tfrac1{2L_k} \tnorm{G_{L_k}(y_k)},
\end{eqnarray}
then the algorithms are still guaranteed to converge. If $L_k$ satisfies \eqref{eq:comAdapL}, then one obtains the improvement $\alpha_k = \mu/L_k$ (or $\alpha_k = \sqrt{\mu/L_k}$ for the accelerated case) at every iteration.

\paragraph{Smooth case.} For smooth functions, \eqref{eq:zzzz3} must hold for SUESA and ASUESA. This means that at every iteration, if $L_k$ satisfies
\begin{equation}\label{eq:smoAdapL}
f\left(y_k - \tfrac1{L_k} \nabla f(y_k) \right)
\leq f(y_k ) - \tfrac1{2 L_k} \tnorm{\nabla f(y_k)},
\end{equation}
then convergence guarantees for SUESA and ASUESA are maintained.
If $L_k$ satisfies \eqref{eq:smoAdapL}
then we have the improvement $\alpha_k = \mu/L_k$ (or $\alpha_k = \sqrt{\mu/L_k}$ for the accelerated case) at every iteration.

With these two inequalities in mind, the adaptive Lipschitz process is described now. Initialize Algorithms~\ref{alg:CUESA}--\ref{alg:ASUESA} with the estimate $0 < \mu \leq L_0\leq L$, and increase and decrease factors $u>1$ and $d\geq 1$, respectively. At each iteration $k\geq1$ of Algorithms~\ref{alg:CUESA}--\ref{alg:ASUESA}, an inner loop (described in Algorithm~\ref{alg:adapL}) is employed to find the approximation $L_k$. The inner loop (Algorithm~\ref{alg:adapL}) replaces Steps~4--5 in Algorithms~\ref{alg:CUESA}--\ref{alg:ASUESA}. The psuedocode is presented now, and the inner iteration counter $s$ in used in Algorithm~\ref{alg:adapL} to distinguish it from the outer loop iteration counter $k$.

\begin{algorithm}
	\caption{Finding $L_k$ in iteration $k$ of Algorithms~\ref{alg:ASUESA} and \ref{alg:ACUESA}.}
	\label{alg:adapL}
	\begin{algorithmic}[1]
		\STATE Input: $x_k, v_k$ $u>1$, $d \geq 1$, $0 < \mu < L_0\leq L$ and $L_{k-1}$.
        \STATE Initialize: $L_k^{(s)} = \max\{L_0,L_{k-1}/d\}$.
		\FOR{$s = 0,1,2,\dots$}
		\STATE  $\alpha_k^{(s)} = \sqrt{\tfrac{\mu}{L_k^{(s)}}}$, $\beta_k^{(s)}  = \tfrac{1}{1+\alpha_k^{(s)} }$.\label{ref:mmmmm0}
		\STATE Set $y_k^{(s)}  = \beta_k^{(s)}  x_k + (1- \beta_k^{(s)} ) v_k$.
		\STATE Set $x_{k+1}^{(s)}  = y_k^{(s)} - \frac{1}{L_k^{(s)}} G_{L_k^{(s)}}(y_k^{(s)}).$
\IF{\eqref{eq:smoAdapL} or \eqref{eq:comAdapL} holds}
    \STATE Break.
  \ELSE
    \STATE $L_k^{(s+1)}$ = $u \cdot L_k^{(s)}$.\label{ref:mmmmm1}
  \ENDIF
\ENDFOR
\STATE Output: $L_k = L_k^{(s)}$, $\alpha_k = \alpha_k^{(s)}$, $\beta_k = \beta_k^{(s)}$, $y_k = y_k^{(s)}$ and $x_{k+1} = x_{k+1}^{(s)}$.
	\end{algorithmic}
\end{algorithm}

Algorithm~\ref{alg:adapL} begins with the estimate $L_k^{(0)}=\max\{L_0,L_{k-1}/d\}$. This ensures that (i) $L_k \geq L_0\geq \mu > 0$ so that convergence of the outer loop is maintained (because $\alpha_k = \mu/L_k \in (0,1)$) must hold); and allows that possibility that (ii) $L_k^{(0)}$ is $\frac1d$ times smaller than $L_{k-1}$ so that large step sizes might be used. Next, $L_k^{(s)}$ is (possibly repeatedly) multiplied by the increase factor $u$ until \eqref{eq:smoAdapL} (or \eqref{eq:comAdapL}) is satisfied at inner iteration $S$, at which point, $L_k = L_k^{(S)}$ is passed to the outer loop and iteration $k$ continues with $L_k$ used in place of $L$. Following this process, it possibly occurs that at some iteration $k$, $L_k < L$. In this case, the stepsize $1/L_k$ is used, which is larger than $1/L$.

The strategy above holds for  Algorithms~\ref{alg:ACUESA} and \ref{alg:ASUESA}, but it is straightforward to adapt it to Algorithms~\ref{alg:CUESA} and \ref{alg:SUESA}  by modifying the variables as $\alpha_s = \mu/L_s$ and $\beta_s =1$.

Using the arguments in \cite[Section~3]{Nesterov13}, the value $L_k$ can increase only if $L_k \leq L$. Thus, by Assumption~\ref{A_SCL} (Lipschitz continuity), the following chain of inequalities holds
\begin{equation}
  0 < \mu \leq L_0 \leq L_k \leq u \cdot L.
\end{equation}
The following lemma bounds the number of inner iterations.

\begin{lemma}[Lemma~4 in \cite{Nesterov13}]
Let Assumption~\ref{A_SCL} hold, and let $u >1$, $d \geq 1$ and $ 0<\mu \leq L_0 \leq L$.
Then the maximum number of times that Lines~\ref{ref:mmmmm0}~to~\ref{ref:mmmmm1} of Algorithm~\ref{alg:adapL} are executed during the first $K$ iterations of Algorithm~\ref{alg:ACUESA} (or Algorithm~\ref{alg:ASUESA}), say $N_k$, is bounded by
\begin{equation}
  N_K \leq \left(1+\frac{\ln d}{\ln u}\right)\cdot(K+1) + \frac1{\ln u} \max\left\{\ln \frac{u \cdot \ln L}{d \cdot \ln L_0},0\right\}.
\end{equation}
\end{lemma}

\section{Numerical Experiments}\label{sec:numericalexperiments}

In this section, we present numerical results to compare our proposed algorithms with several other methods that have an optimal convergence rate. The algorithms are as follows, and are summarized in Table~\ref{table-DiscAlgs}.

\emph{OQA/OQA+.} The Optimal Quadratic Averaging algorithm (OQA) \cite{Drusvyatskiy16}, which builds upon the work in \cite{Bubeck15}, maintains a quadratic lower bound on the objective function value at every iteration. The quadratic lower bound is called `optimal' because it is the `best' lower bound that can be obtained as a convex combination of the previous 2 quadratic lower bounds. In OQA, $x_{k+1}$ is set to be the minimizer of $f(x)$ on the line joining the points $x_k^+$ and the minimizer of the current quadratic lower bound. In \cite{Drusvyatskiy16} the author suggest a variant of OQA, which we call OQA+ here, that computes $x_{k}^+$ via a line search that does not use the true Lipschitz constant $L$.

\emph{NEST/NEST+.} NEST denotes the algorithm described in Chapter 2 of \cite{Nesterov04}, which is for \emph{smooth} functions. Further, NEST+ is a variant of NEST in which the Lipschitz constant $L$ is adaptively update via the strategy in \cite{Nesterov07,Nesterov13}. Both variants are \emph{accelerated} algorithms.

\emph{CNEST/CNEST+.} CNEST denotes Algorithm (4.9) in \cite{Nesterov07}, which can be applied to \emph{composite nonsmooth} functions. Further, CNEST+ is a variant of CNEST in which the Lipschitz constant $L$ is adaptively update via the strategy in \cite{Nesterov07,Nesterov13}. Both variants are \emph{accelerated} algorithms.

\emph{GD.} We also implement a Gradient Descent (GD) method which uses a fixed stepsize of $\tfrac{1}{L}$. Note that this is similar to Algorithm~\ref{alg:SUESA}, although GD does not maintain any kind of lower bound. As the only non-optimal algorithm, Gradient Descent provides a benchmark that will enable us to observe any performance advantages of the optimal methods.

\begin{table}[H]\tiny
	\centering
	\begin{tabular}{l|l} \toprule
		{Algorithm} & {Description}  \\ \midrule
		{OQA} &  {Optimal Quadratic Averaging Algorithm} \\
		{OQA+}  & {Optimal Quadratic Averaging Algorithm with $x_k^+= \text{line-search} (x_{k},x_k-\nabla f(x_k))$}   \\
        {SUESA}  & {Smooth Underestimate Sequence Algorithm}   \\
		{ASUESA}  & {Accelerated Smooth Underestimate Sequence Algorithm}\\
		{ASUESA+}  & {Accelerated Smooth Underestimate Sequence Algorithm with adaptive Lipschitz constant} \\
		{NEST}   & {Accelerated algorithm described in Chapter 2 of \cite{Nesterov04}} \\
		{NEST+}  & {Accelerated algorithm described in Chapter 2 of \cite{Nesterov04} with adaptive Lipschitz constant}  \\\midrule
		{CUESA}  &  {Composite Underestimate Sequence Algorithm}     \\
		{CUESA+} &  {Composite Underestimate Sequence Algorithm with adaptive Lipschitz constant} \\
		{ACUESA} &  {Accelerated Composite Underestimate Sequence Algorithm}  \\
		{ACUESA+} &  {Accelerated Composite Underestimate Sequence Algorithm with adaptive Lipschitz constant} \\
		{CNEST}  &  {Accelerated Algorithm (4.9) in \cite{Nesterov07} with fixed Lipschitz constant}     \\
		{CNEST+}   & {Accelerated Algorithm (4.9) in \cite{Nesterov07} with adaptive Lipschitz constant} \\\bottomrule
	\end{tabular}
	\normalsize
	\caption{Description of implemented algorithms}
\label{table-DiscAlgs}
\end{table}

\subsection{Empirical Risk Minimization}

We consider two Empirical Risk Minimization (ERM) problems, which are popular in the machine learning literature. In particular, we study ERM with a squared hinge loss
\begin{equation}\label{eq:squaHingeLoss}
f(x) =  \frac{1}{m} \sum_{i=1}^{m} \big(\max \{0,1 - y_ia_i^Tx\}\big)^2 +  \tfrac{\lambda}{2}\tnorm{x},
\end{equation}
and ERM with a logistic loss (also called logistic regression)
\begin{equation}\label{eq:logisticRe}
f(x) = \frac{1}{m} \sum_{i=1}^{m} \log\big( 1+ e^{-y_ia_i^Tx}\big) +  \tfrac{\lambda}{2}\tnorm{x}.
\end{equation}
In each case $y_i \in\{-1,+1\}$ is the label and $a_i \in \R^n$ represents the training data for $i = 1,2,...,m$. All the datasets in our experiments come from LIBSVM database \cite{Chang11}. Also note that for all experiments we have $\mu = \lambda$.

Note that in these experiments, when implementing a line search in direction $v$ (whenever relevant), the dot products $a_i^Tv$ are reused.

\subsubsection*{Comparison on Decreasing Objective Values}

In the first experiment we compare the OQA, ASUESA and NEST algorithms (both the standard and adaptive Lipschitz variants) and investigate how the objective function values behave on several test problems. The test problems considered in this experiment are the \texttt{ala} dataset with a squared hinge loss and a value $\lambda = 10^{-4}$, the \texttt{rcv1} dataset with a logistic loss and a value $\lambda = 10^{-4}$, and the \texttt{covtype} dataset with a squared hinge loss and a value $\lambda = 10^{-5}$.
\begin{figure*}[h!]
	\centering
	\includegraphics[scale=.15]{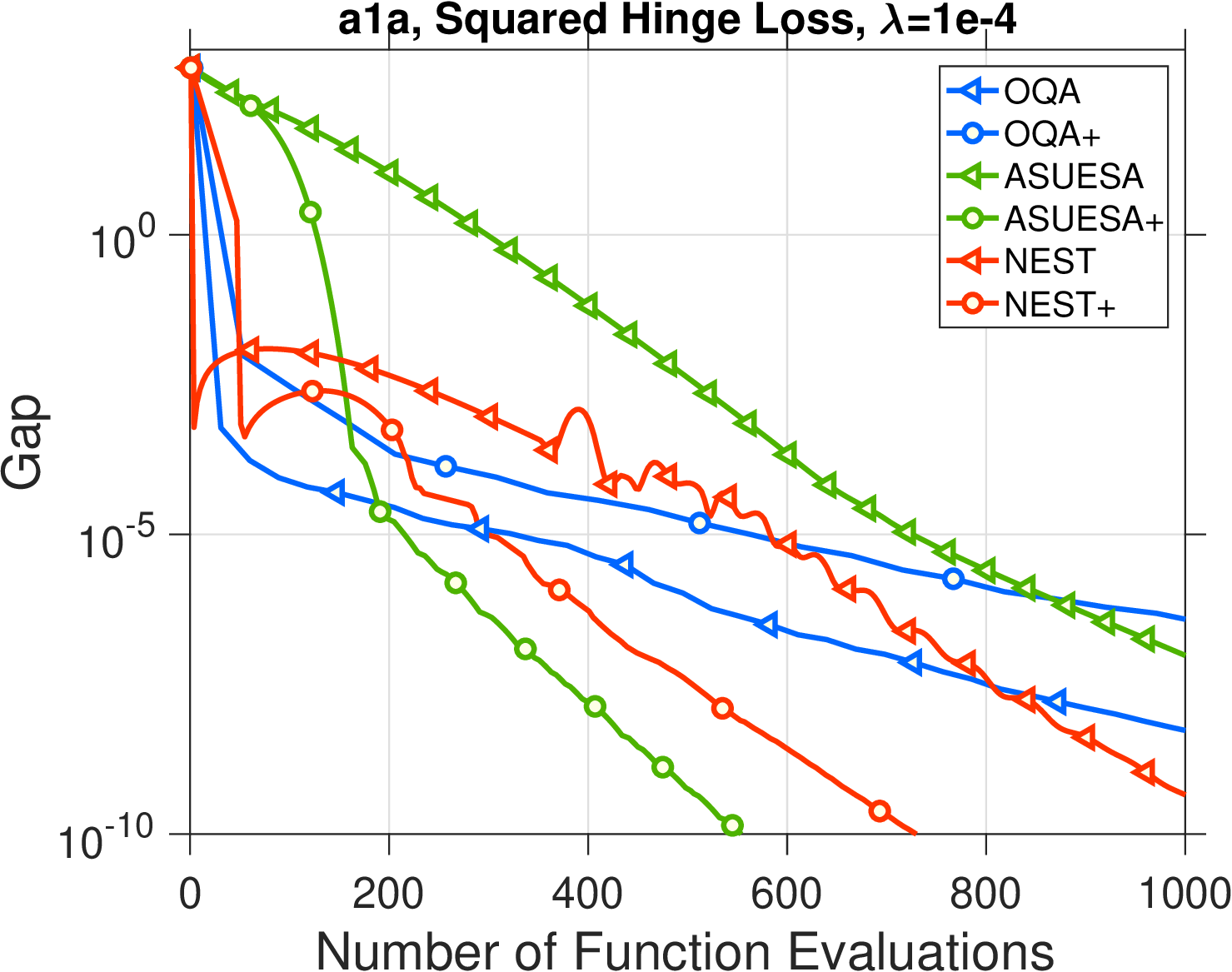}
    \includegraphics[scale=.15]{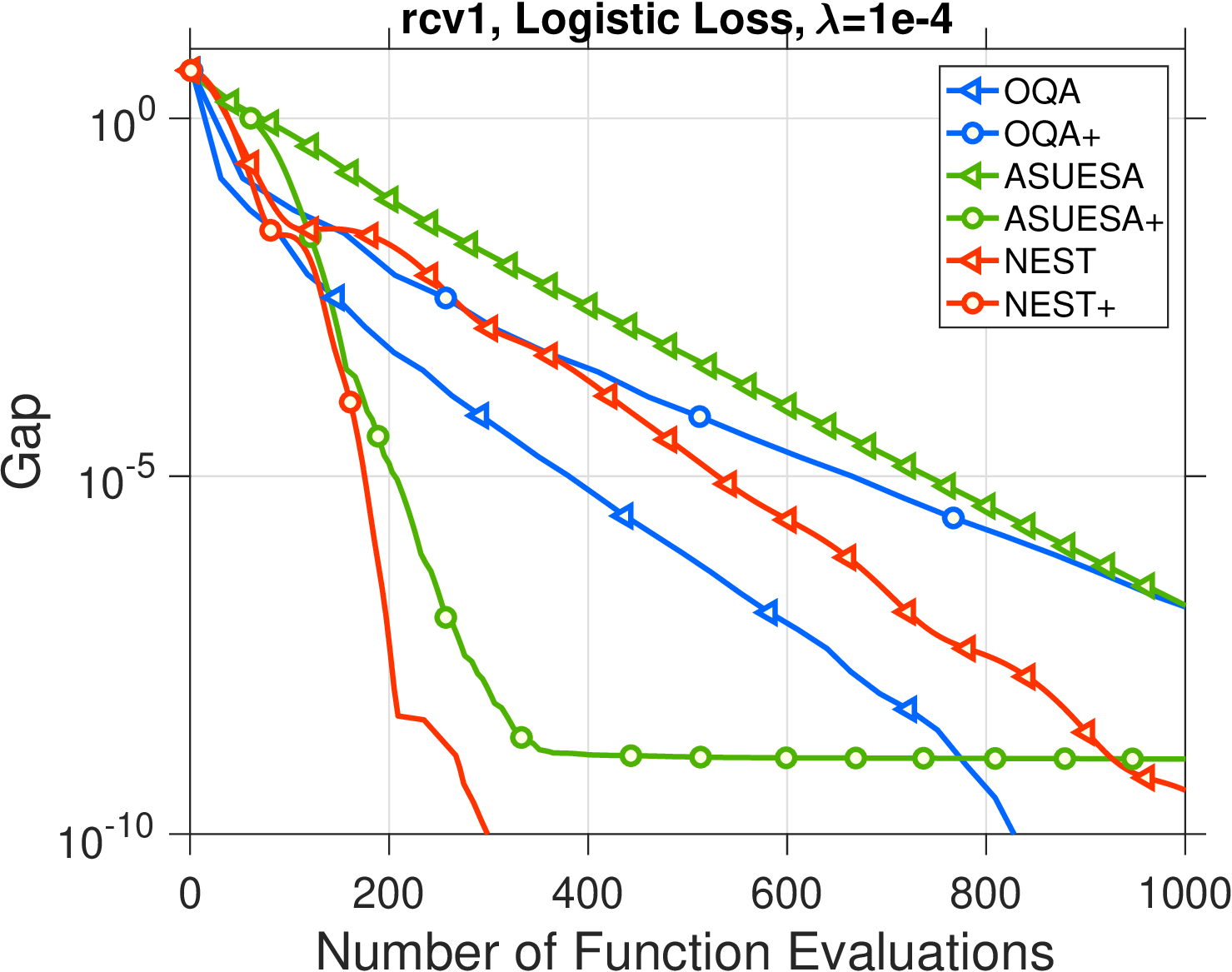}
    \includegraphics[scale=.15]{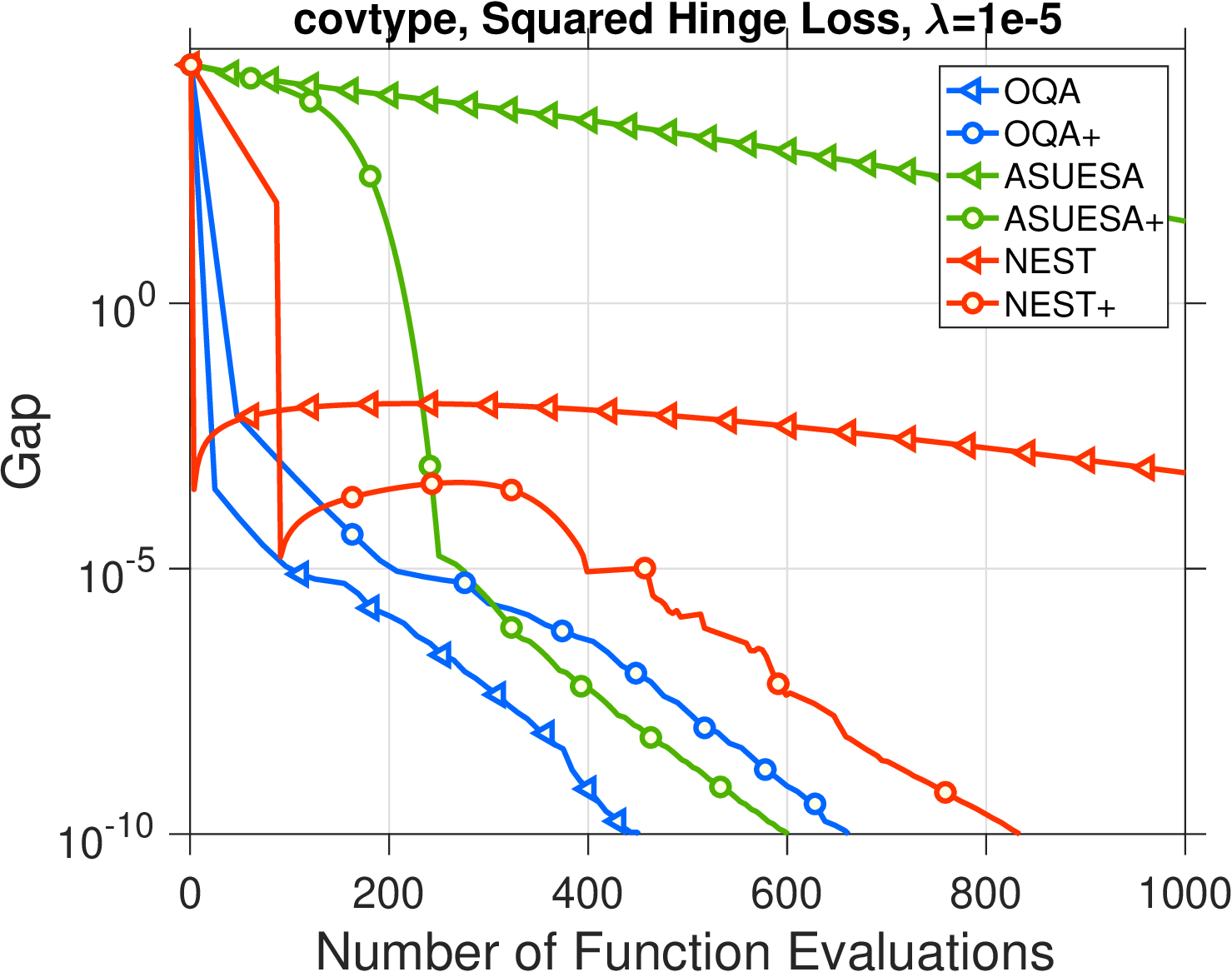}
	\includegraphics[scale=.15]{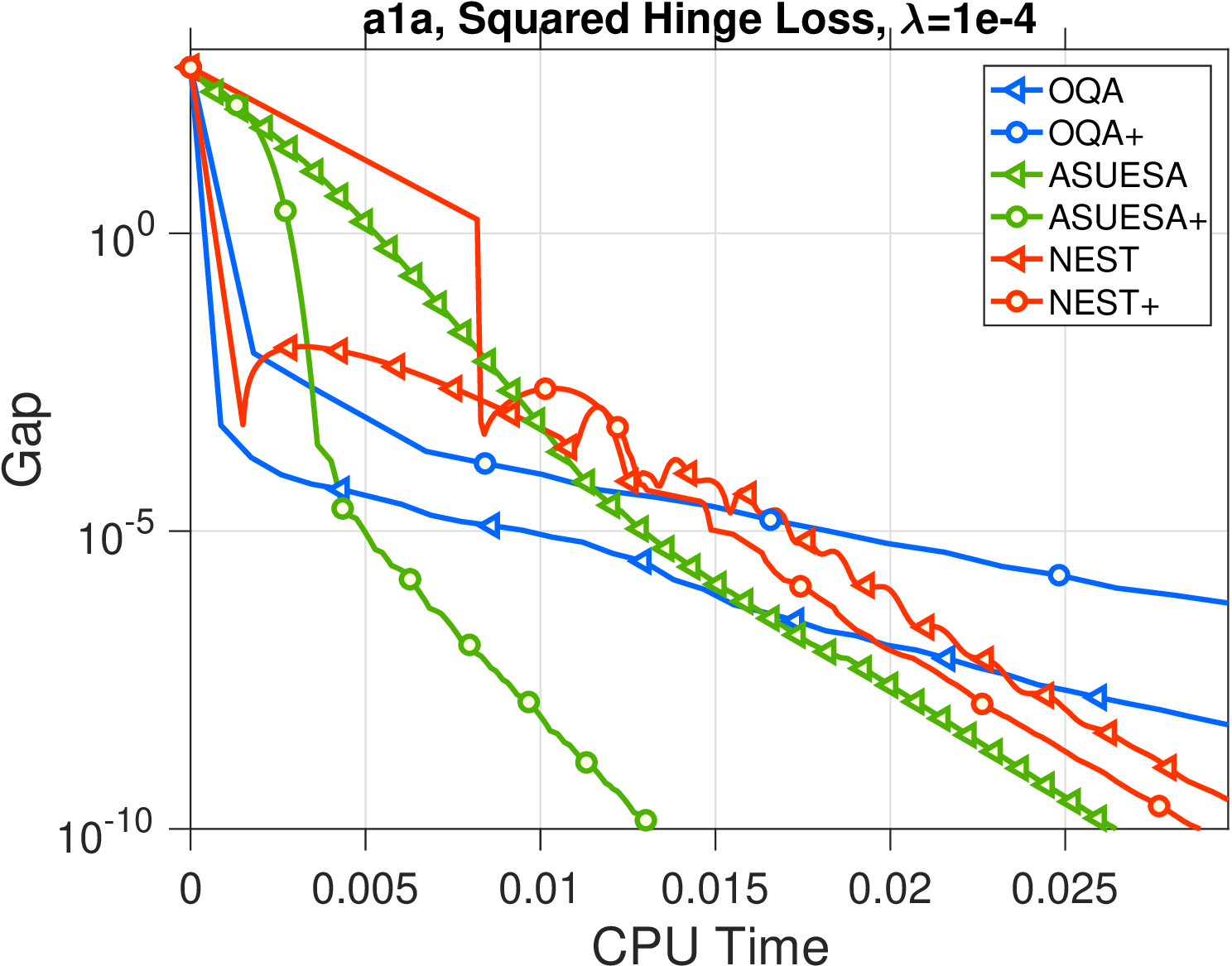}
	\includegraphics[scale=.15]{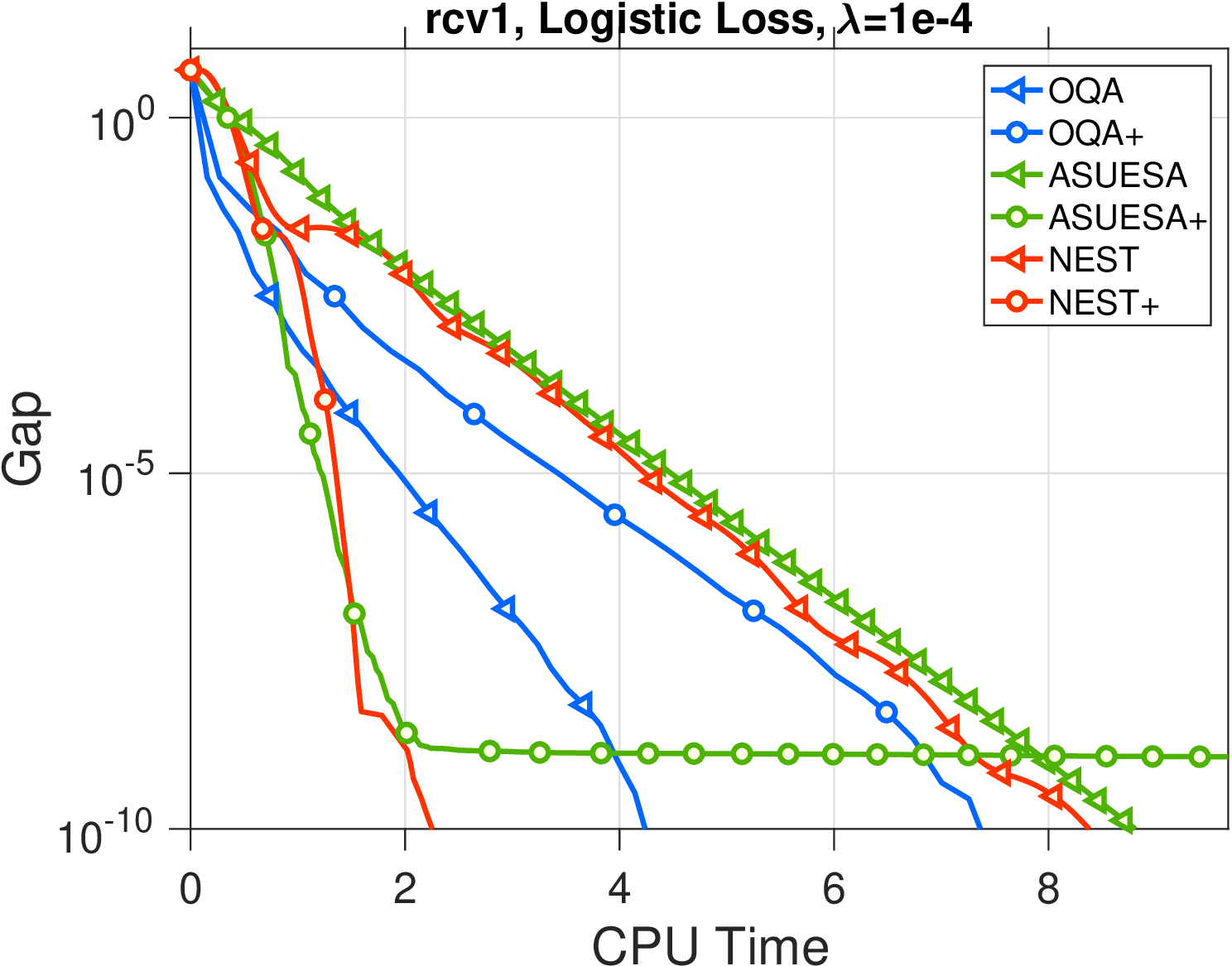}
	\includegraphics[scale=.15]{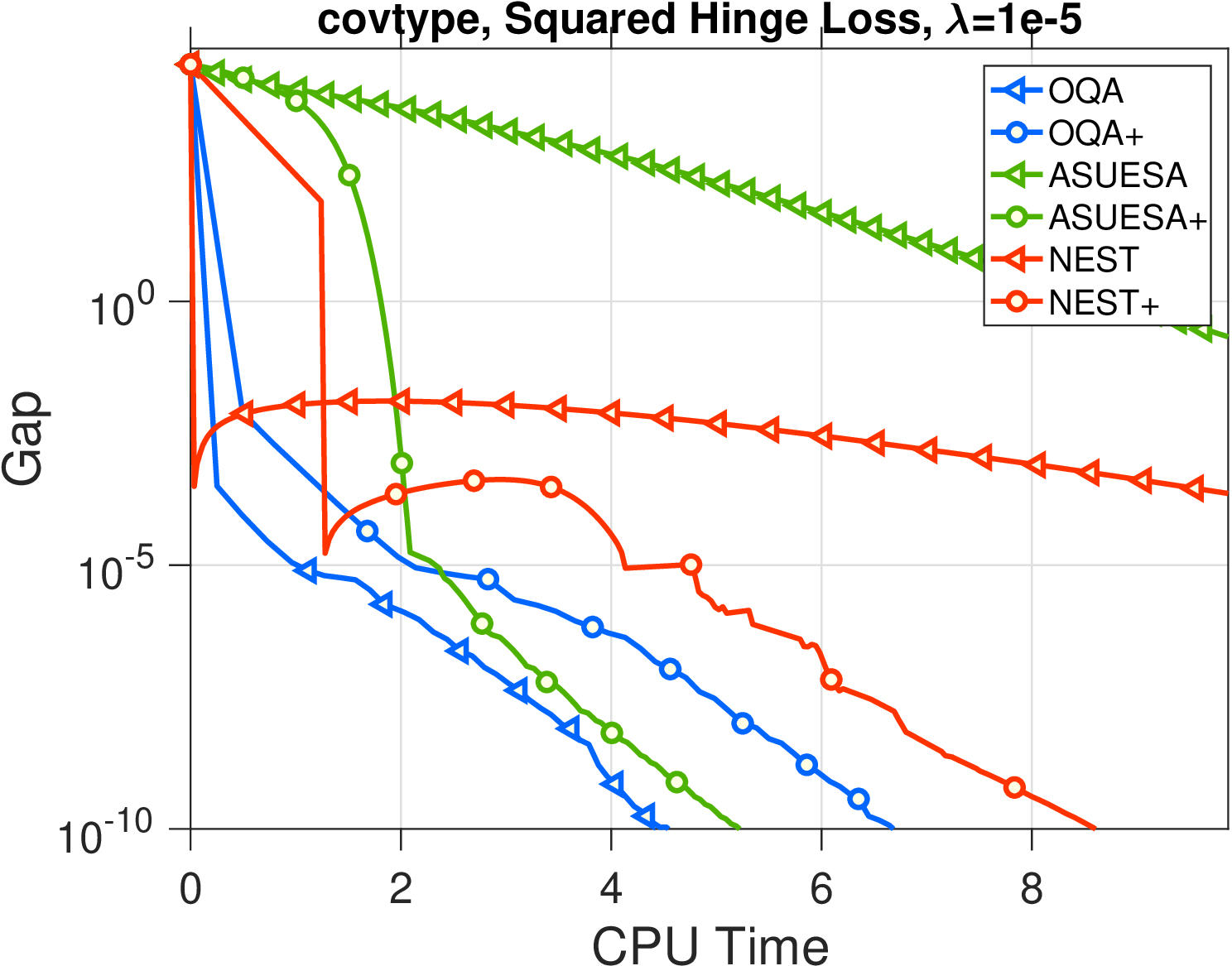}
	\caption{Evolution of the gap $f(x_k)-\phi_k^*$ for each algorithm compared with the number of function evaluations and cputime.}
	\label{fig:exp11}
\end{figure*}
In Figure~\ref{fig:exp11} we plot the gap $f(x_k)-\phi_k^*$ vs the number of function evaluations and the gap $f(x_k)-\phi_k^*$ vs the cpu time. The figure shows the advantages of using an adaptive Lipschitz constant with the adaptive methods performing better than their original versions in most cases. Figure~\ref{fig:exp11} also shows that ASUESA+ performs very well, being the best algorithm on the first dataset, and the second best algorithm on the other two datasets.
\textcolor{black}{The numerical experiments highlight that the gap, $f(x_k) - \phi_k^*$, goes to zero \emph{monotonically} for our proposed methods, which matches our theoretical results. It is crucial to note that this does \emph{not} necessarily mean that the objective function value decreases monotonically. Note also that the results confirm that Nesterov's method is non-monotone as well (which is also observed  in the study \cite{o2015adaptive}).}

\subsubsection*{Theory and Practice for OQA and ASUESA}
In this numerical experiment we study ASUESA and OQA and investigate how their practical performance compares with that predicted by theory. For the OQA algorithm a line search is needed to determine a necessary algorithmic variable, and to ensure that theory for OQA holds, the line search should be exact. In this experiment we will use bisection to compute this variable, but we will restrict the number of bisection steps allowed to $b=2,5,20$. Figure~\ref{fig:bisection} plots the ratio $(f(x_k) - \phi^*_{k})/(f(x_{k-1}) - \phi^*_{k-1})$ for ASUESA and for three instances of OQA, where each instance uses a different number of bisection steps $b=2,5,20$. We also plot $1-\sqrt{\tfrac{\mu}{L}}$ (black dots), which is the amount of decrease in the gap $f(x_k)-\phi_k^*$ at each iteration predicted by the theory. (In theory, we should have $(f(x_k) - \phi^*_{k})/(f(x_{k-1}) - \phi^*_{k-1}) \leq 1-\sqrt{\tfrac{\mu}{L}}$ for all $k\geq 0$.)


\begin{figure*}[h!]
	\centering
	\includegraphics[scale=.15]{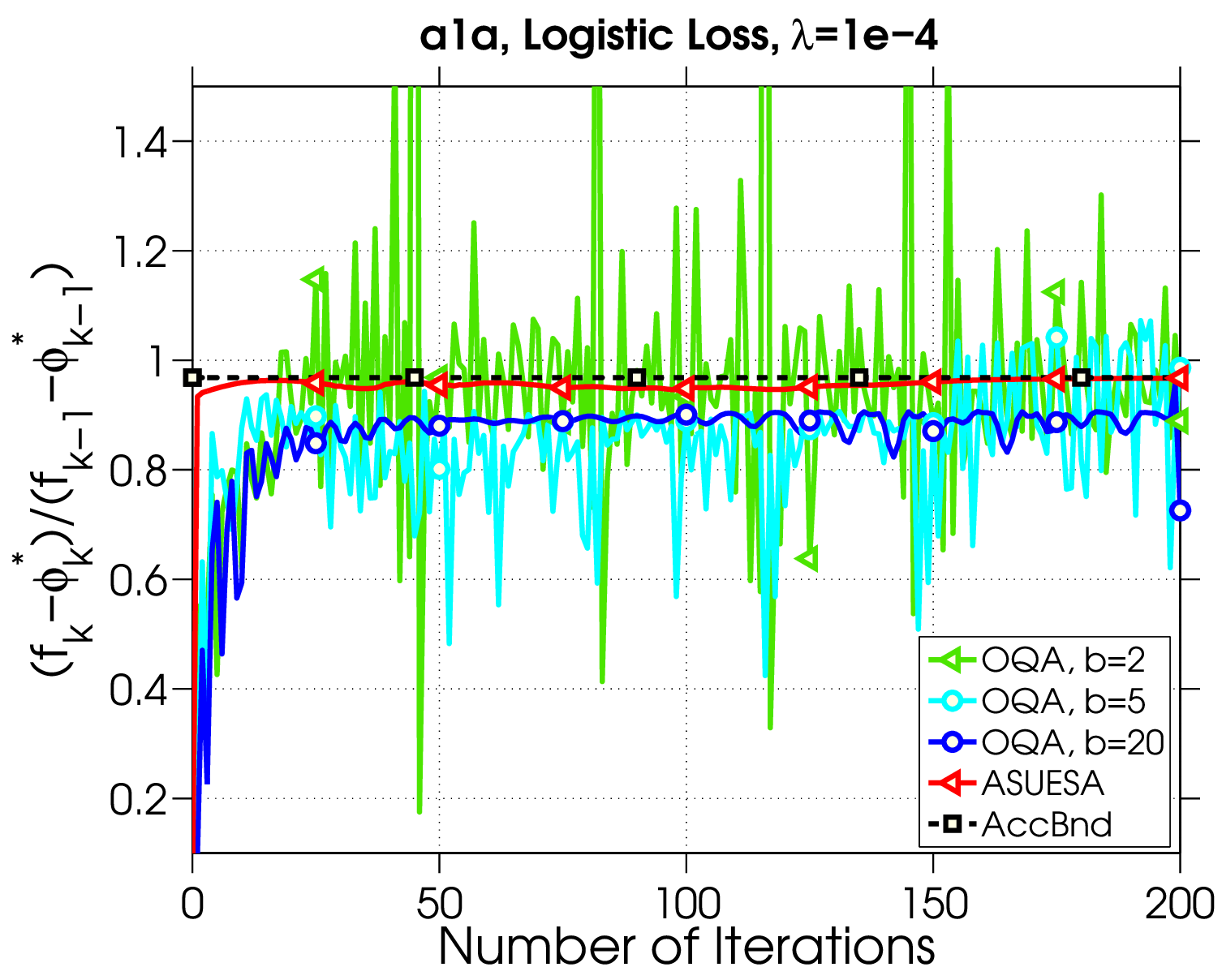}
	\includegraphics[scale=.15]{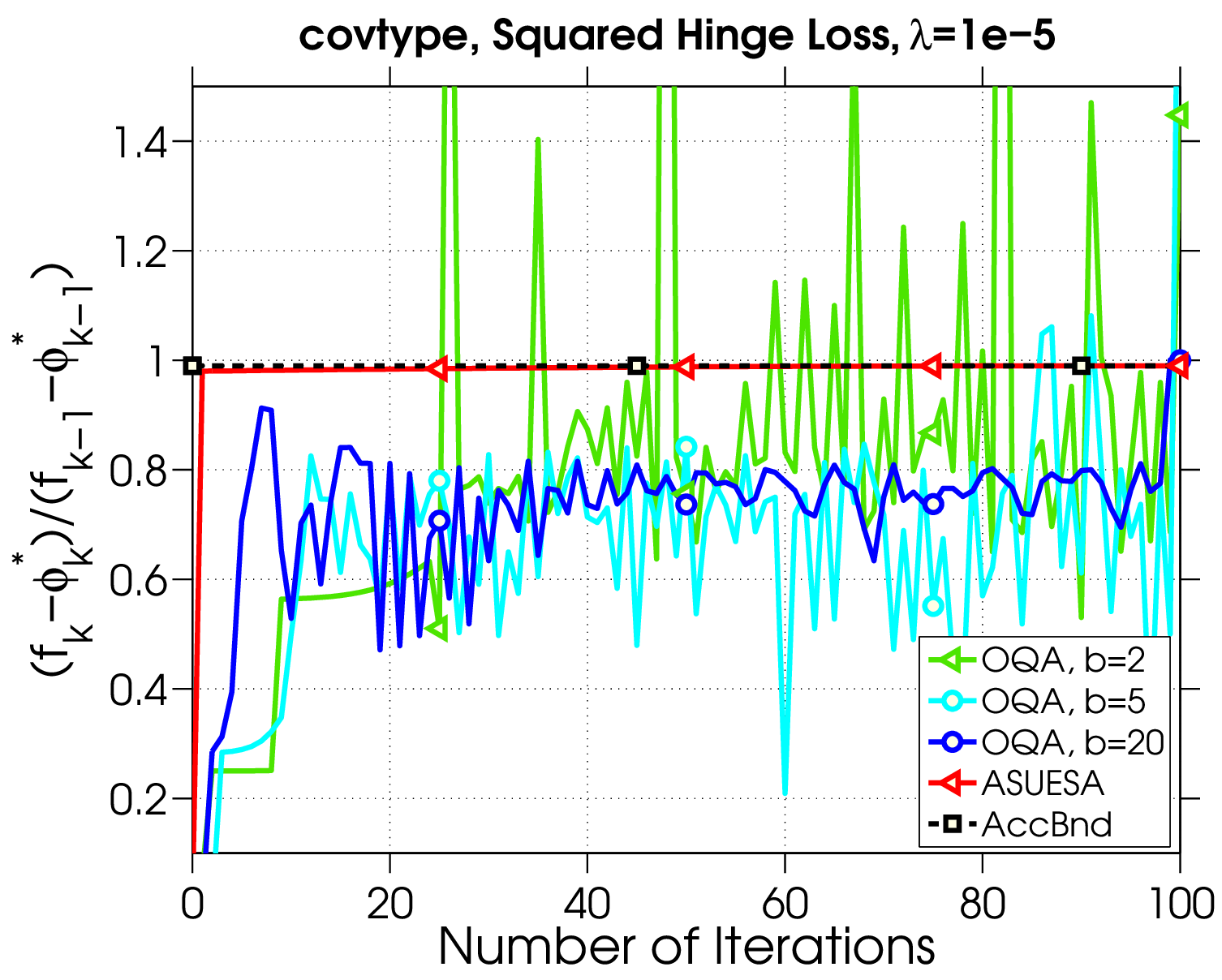}
	\caption{Comparison of  $\tfrac{f(x_k) - \phi^*_{k}}{f(x_{k-1}) - \phi^*_{k-1}}$ for ASUESA and for OQA with different numbers of bisection steps ($b=2,5,20$). The black dots are $1-\sqrt{\tfrac{\mu}{L}}$.}
	\label{fig:bisection}
\end{figure*}

From the plots in Figure~\ref{fig:bisection} we see that ASUESA performs very well, and as predicted by the theory, with the ratio $(f(x_k) - \phi^*_{k})/(f(x_{k-1}) - \phi^*_{k-1})$ always strictly below the theoretical bound. On the other hand, the quality of line search affects OQA significantly. The fewer the number of line search (bisection) iterations, the more likely it is for OQA to violate the theoretical results. Note that this is not necessarily surprising because the theory for OQA requires the exact minimizer along a line segment to be found, so 2 or 5 iterations of bisection may be simply too few to find it. Notice that when $b=2$, the green line shows that OQA behaves erratically, with the ratio $(f(x_k) - \phi^*_{k})/(f(x_{k-1}) - \phi^*_{k-1})$ being greater than 1 on many iterations, indicating that the gap is growing on those iterations. When we use OQA with $b=5$ steps of bisection at each iteration (light blue line), the algorithm performs better, and often, but not always, the ratio is less than 1. Finally, the dark blue line shows the behaviour of OQA when $b=20$ steps of bisection at each iteration. The dark blue line is always below the theoretical bound of $1-\sqrt{\tfrac{\mu}{L}}$, indicating good algorithmic performance (often better than predicted by theory). However, the line search needed by OQA comes at an additional computational cost, which can still mean that the overall runtime is longer for OQA than for ASUESA, as we now show.

Here a similar experiment is performed to compare the theoretical and practical performance of SUESA and ASUESA. We have already seen that the theoretical results for ASUESA give a proportional reduction of $1-\sqrt{\tfrac{\mu}{L}}$ in the gap at every iteration. However, for SUESA, the proportional reduction in the gap is $1-\tfrac{\mu}{L}$. We investigate how these theoretical bounds compare with the practical performance of each of these algorithms. We use the \texttt{ala}, \texttt{rcv1} and \texttt{covtype} datasets for this experiment, and for each of the three datasets we form both a logistic loss, and a squared hinge loss to create 6 problem instances. The results are shown in Figure~\ref{fig:thvsprac}.
\begin{figure*}[h!]\centering
	\includegraphics[scale=.15]{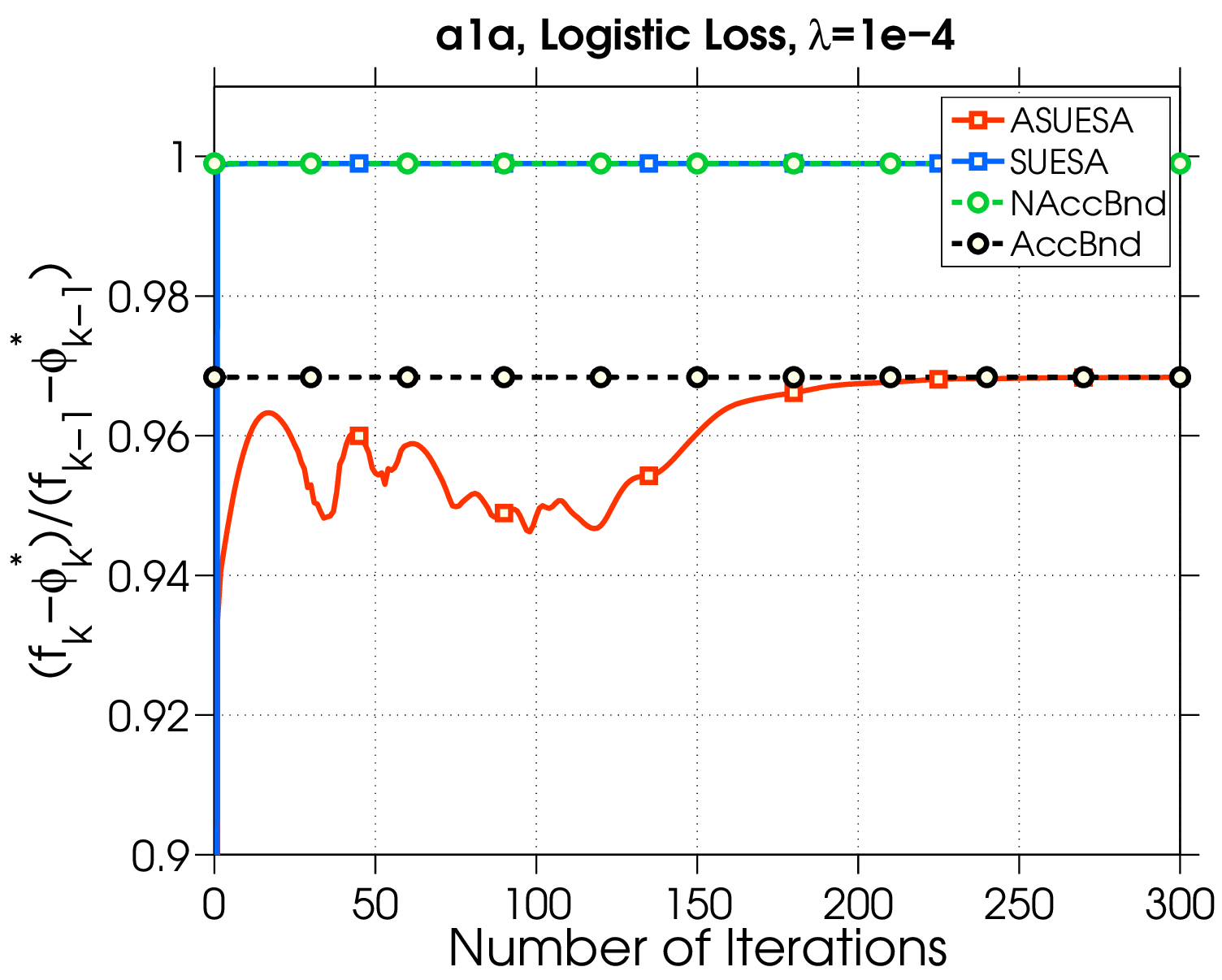}
	\includegraphics[scale=.15]{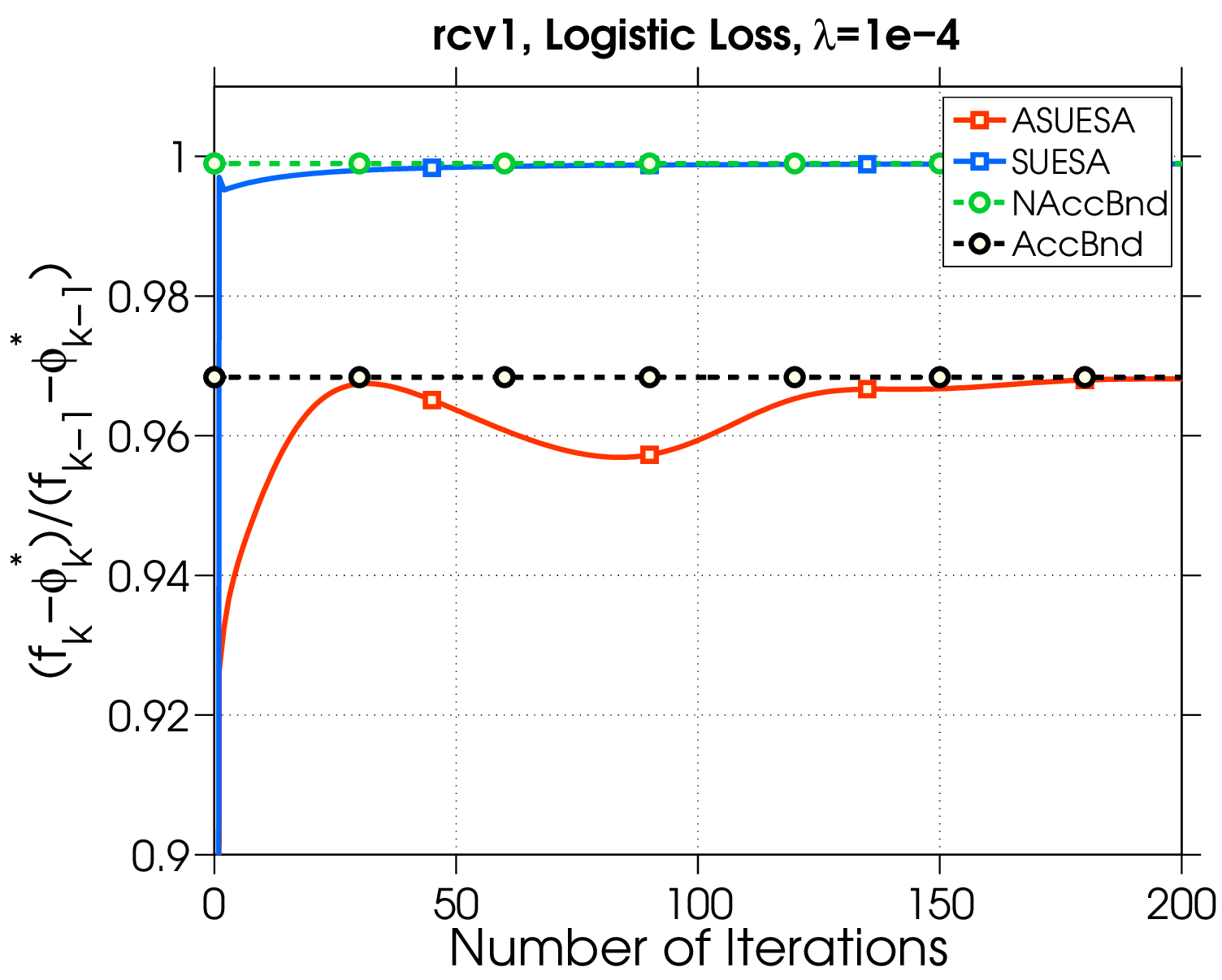}
    \includegraphics[scale=.15]{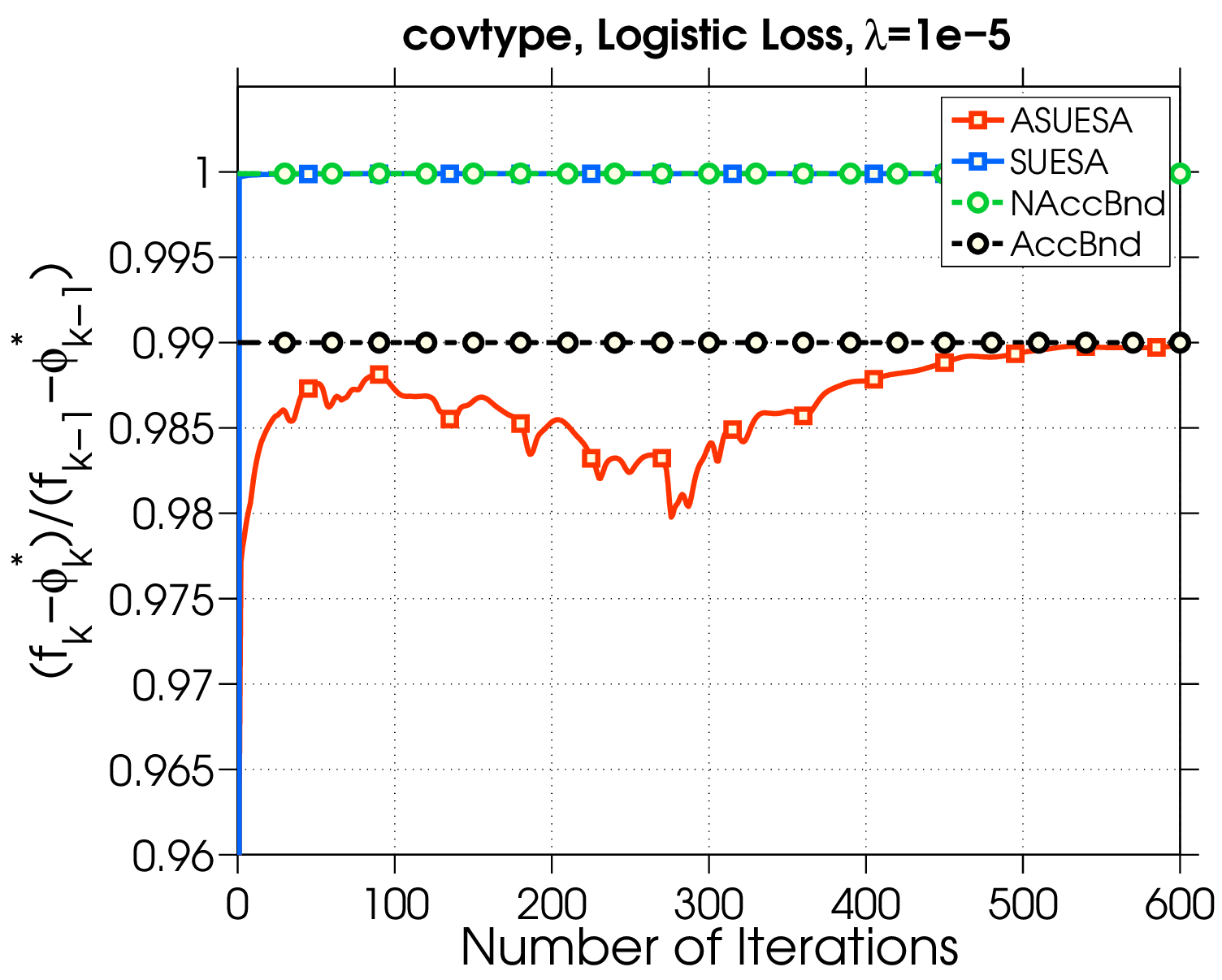}
    \includegraphics[scale=.15]{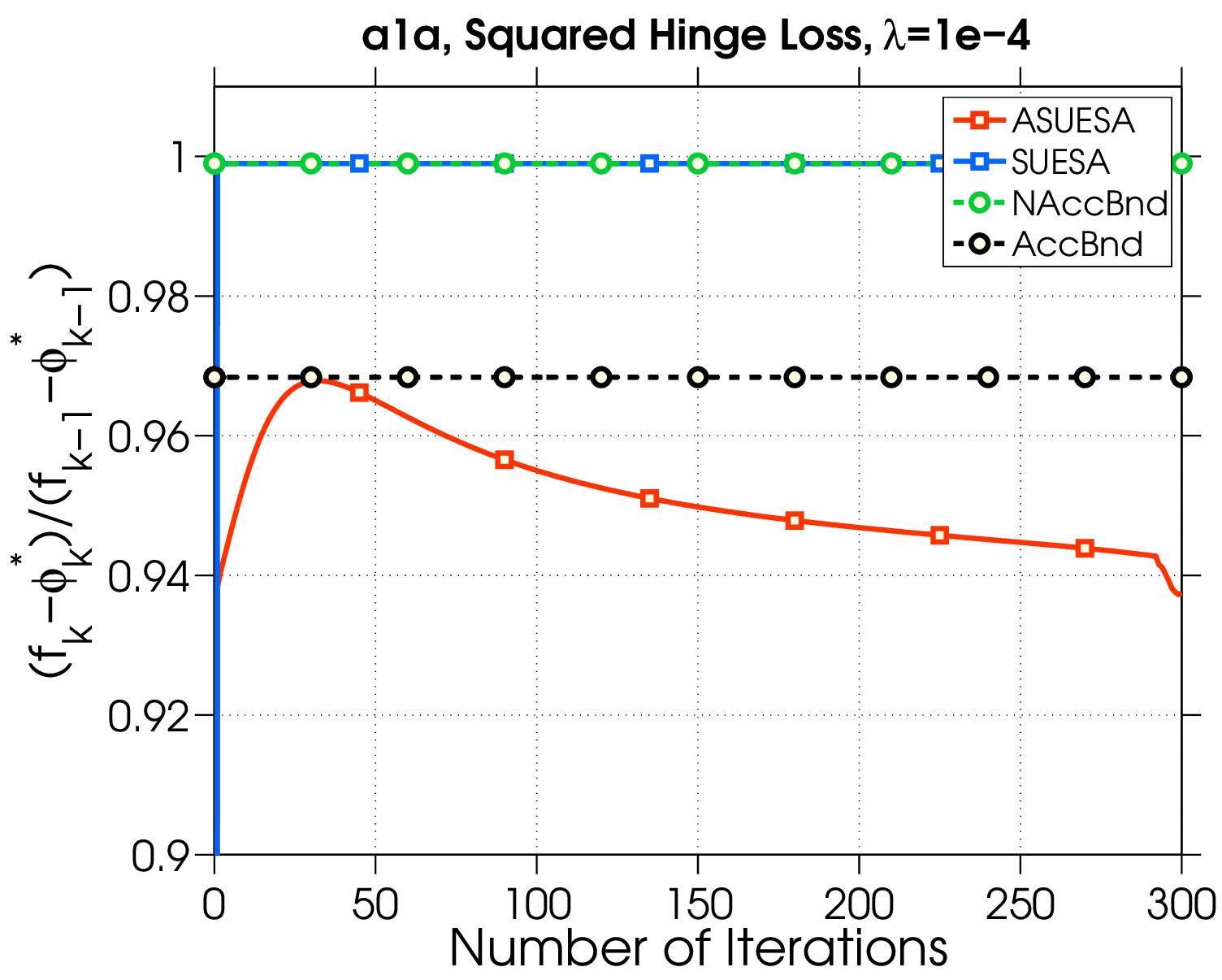}
	\includegraphics[scale=.15]{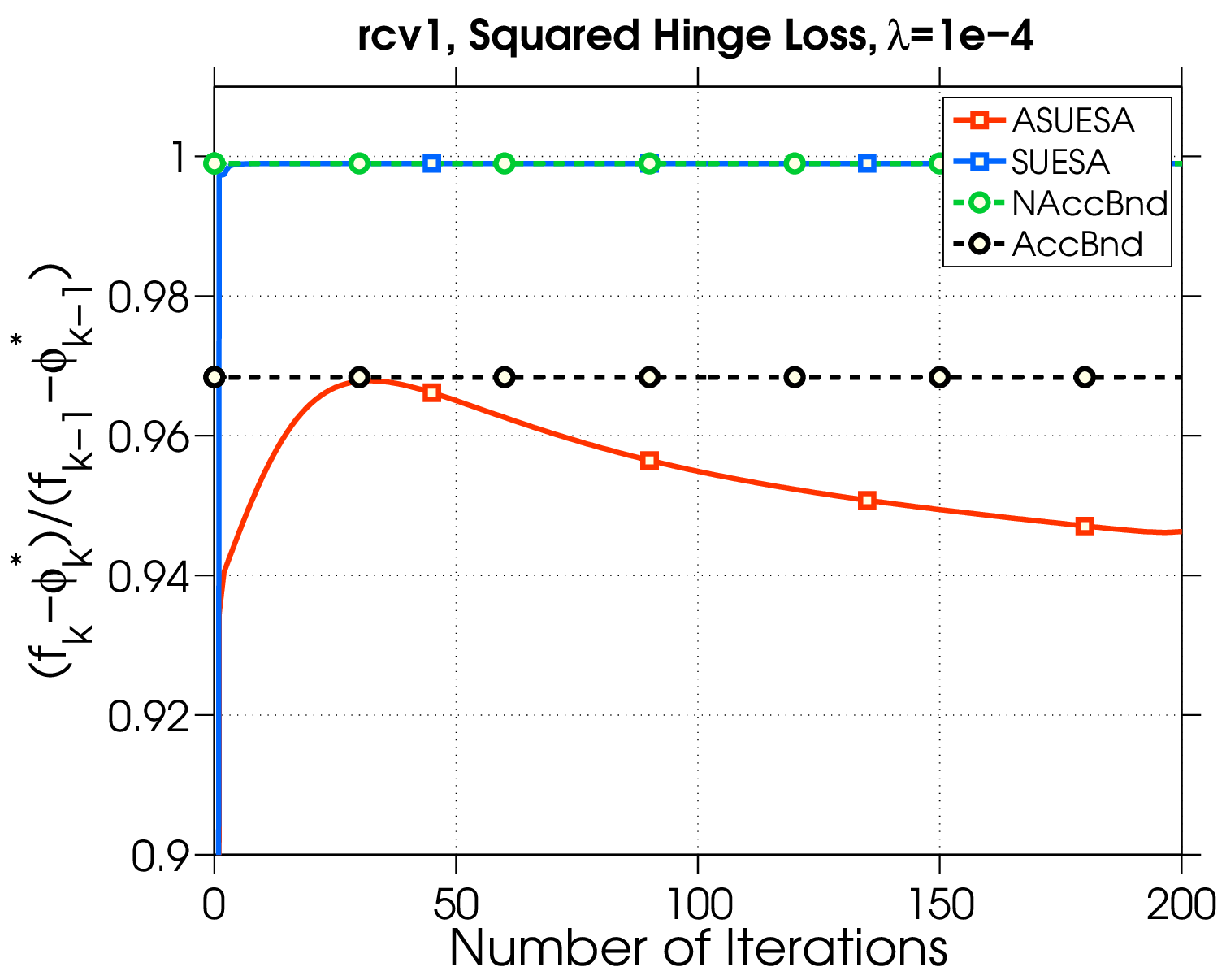}
	\includegraphics[scale=.15]{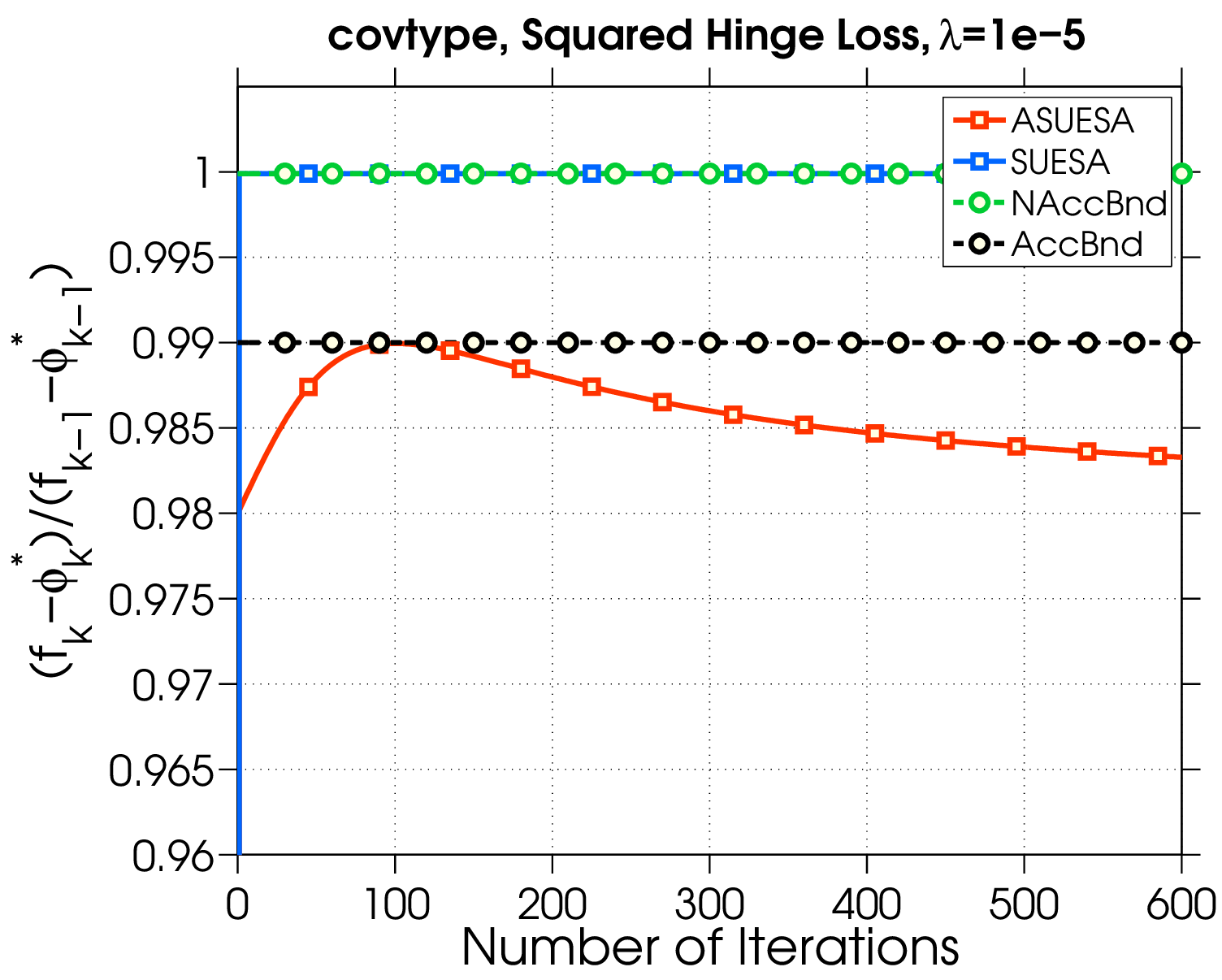}
	\caption{Comparison of $\tfrac{f(x_k) - \phi^*_{k}}{f(x_{k-1}) - \phi^*_{k-1}}$ for SUESA and ASUESA and $1 - \tfrac{\mu}{L}$(green line) and $1-\sqrt{\tfrac{\mu}{L}}$ (black line). }
	\label{fig:thvsprac}
\end{figure*}
Figure~\ref{fig:thvsprac} presents the ratio $\tfrac{f(x_k) - \phi^*_{k}}{f(x_{k-1}) - \phi^*_{k-1}}$ for SUESA and ASUESA. Also displayed is the theoretical (unaccelerated) rate $1 - \tfrac{\mu}{L}$ (the green line) and the theoretical (accelerated) rate $1-\sqrt{\tfrac{\mu}{L}}$ (the black line). One sees that the practical performance of SUESA is very similar to that predicted by the theory because the blue line matches the green line closely. Another observation is that for the accelerated algorithm (ASUESA), in practice, the reduction in the gap $f(x_k) - \phi^*_{k}$ is often more optimistic than the theoretical rate.

\subsection{Experiments on composite functions}

In this section we perform several numerical experiments on problems with a composite objective. Specifically, we consider the elastic net problem, which is problem \eqref{eq:Problem} with
\begin{equation}\label{eq:squaHingeLosscomp}
F(x) =  \frac{1}{n} \sum_{i=1}^{n} (a_i^Tx -y_i)^2 +  \tfrac{\lambda_1}{2}\tnorm{x} + \tfrac{\lambda_2}{2}\|x\|_1.
\end{equation}
Notice that the first two terms in \eqref{eq:squaHingeLosscomp} are smooth, while the $\ell_1$-norm term makes \eqref{eq:squaHingeLosscomp} nonsmooth overall.
We compare our Algorithm \ref{alg:CUESA} and \ref{alg:ACUESA} (CUESA and ACUESA) with the one proposed in \cite{Nesterov07} (CNEST). As stated previously, each of these algorithms can be implemented with either a fixed $L$ or an adaptive $L$, and we will compare each algorithm under both of these two options.

For these experiments we again use the 3 datasets \texttt{ala}, \texttt{rcv1} and \texttt{covtype}. For the \texttt{ala} data the regularization parameters were set to $\lambda_1= \lambda_2 =  10^{-4}$, for the \texttt{rcv1} data the regularization parameters were set to $\lambda_1 = 10^{-4}$ and $ \lambda_2 = 10^{-5}$, and for the \texttt{covtype} data the regularization parameters were set to $\lambda_1 = 10^{-4}$ and $ \lambda_2 = 10^{-6}$.
\begin{figure}[htbp]
	\centering
	\includegraphics[scale=.15]{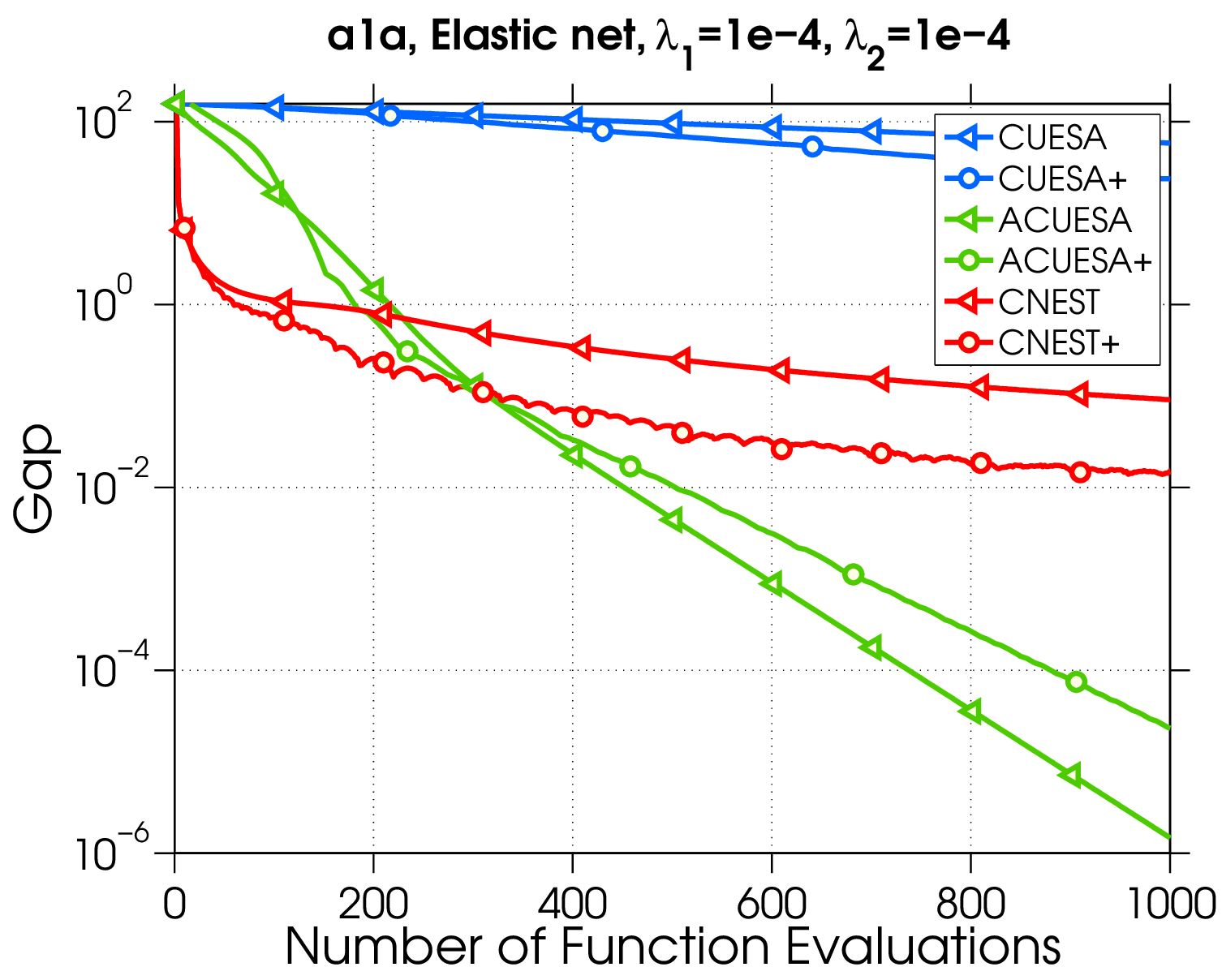}
    \includegraphics[scale=.15]{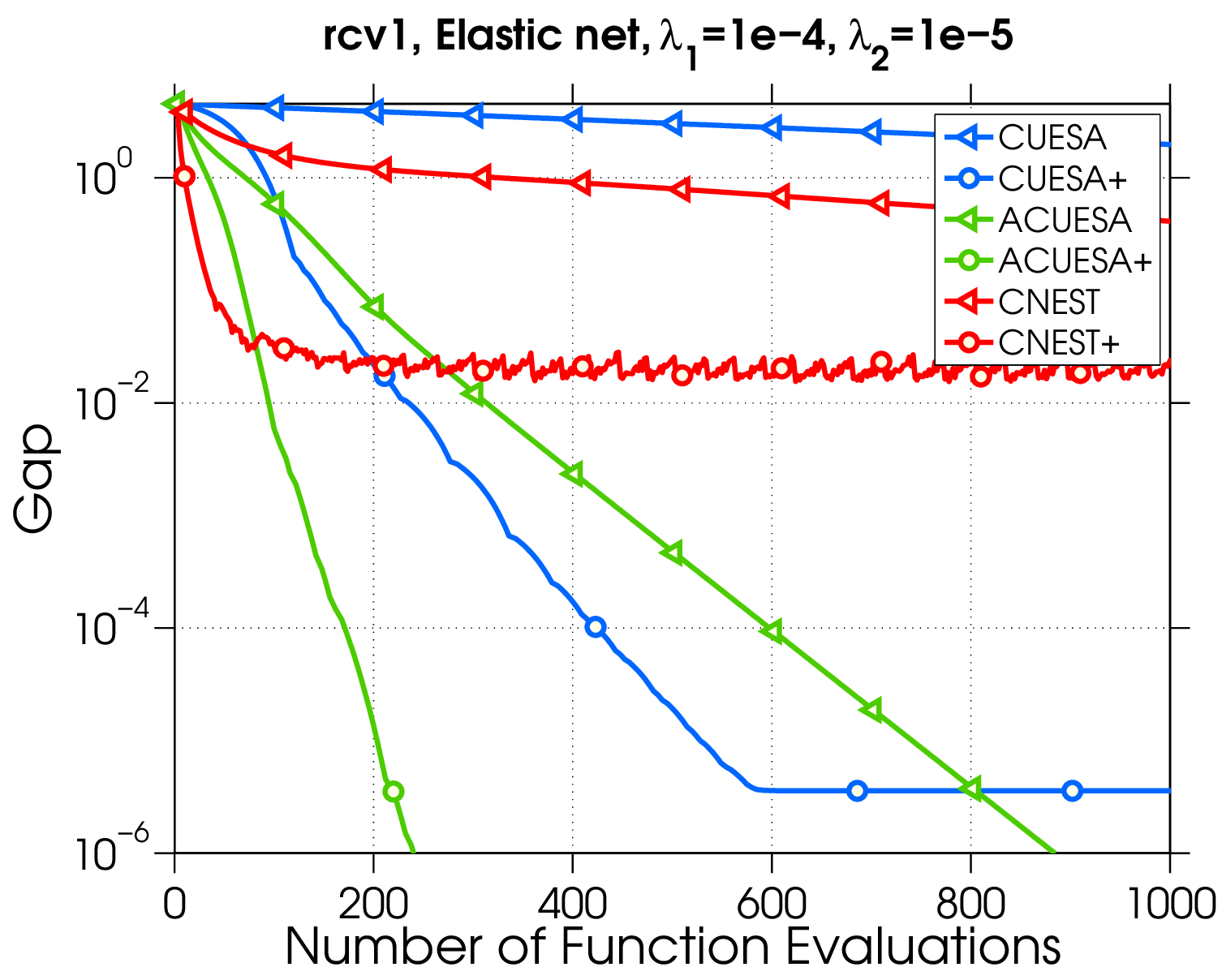}
    \includegraphics[scale=.15]{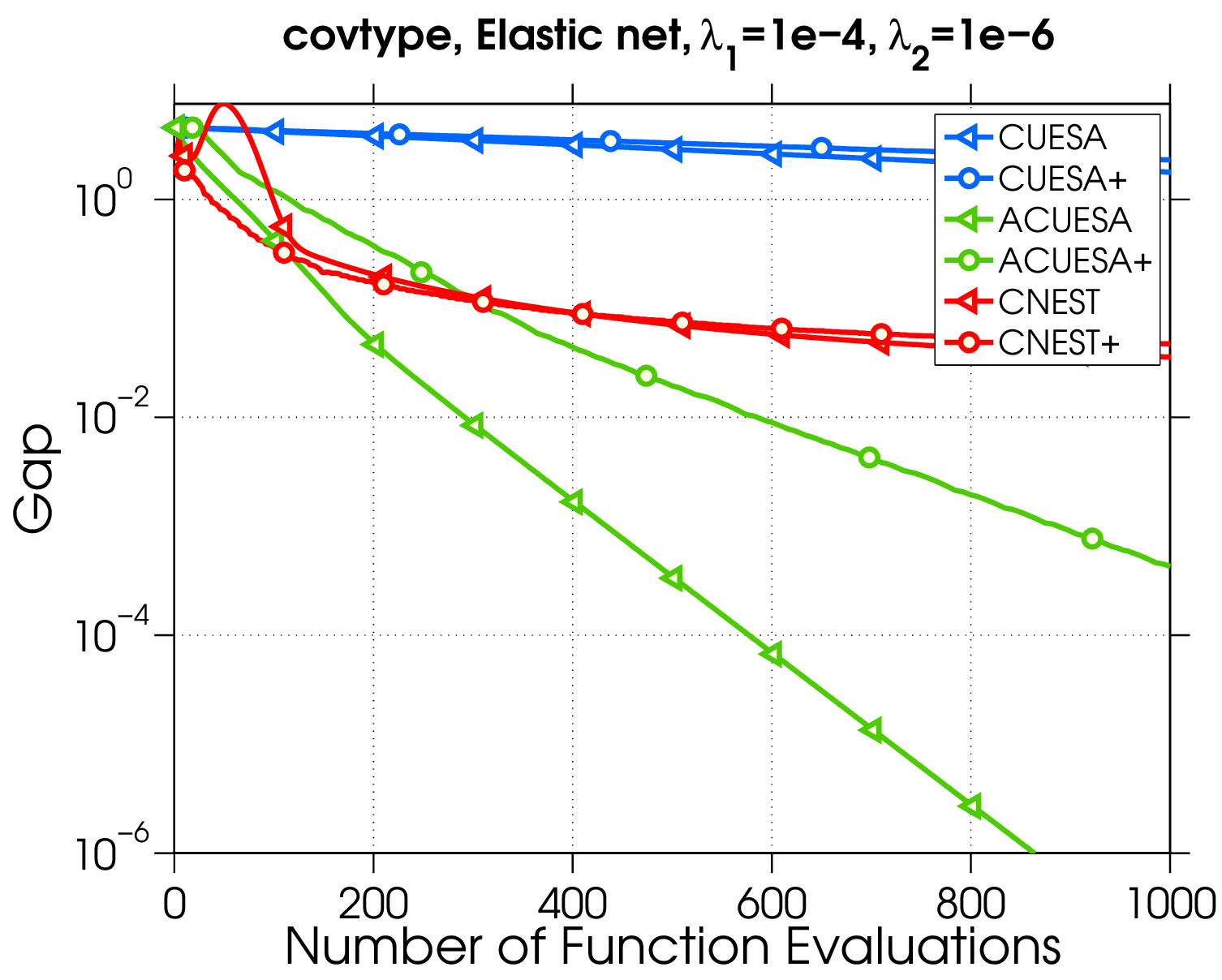}
	\includegraphics[scale=.15]{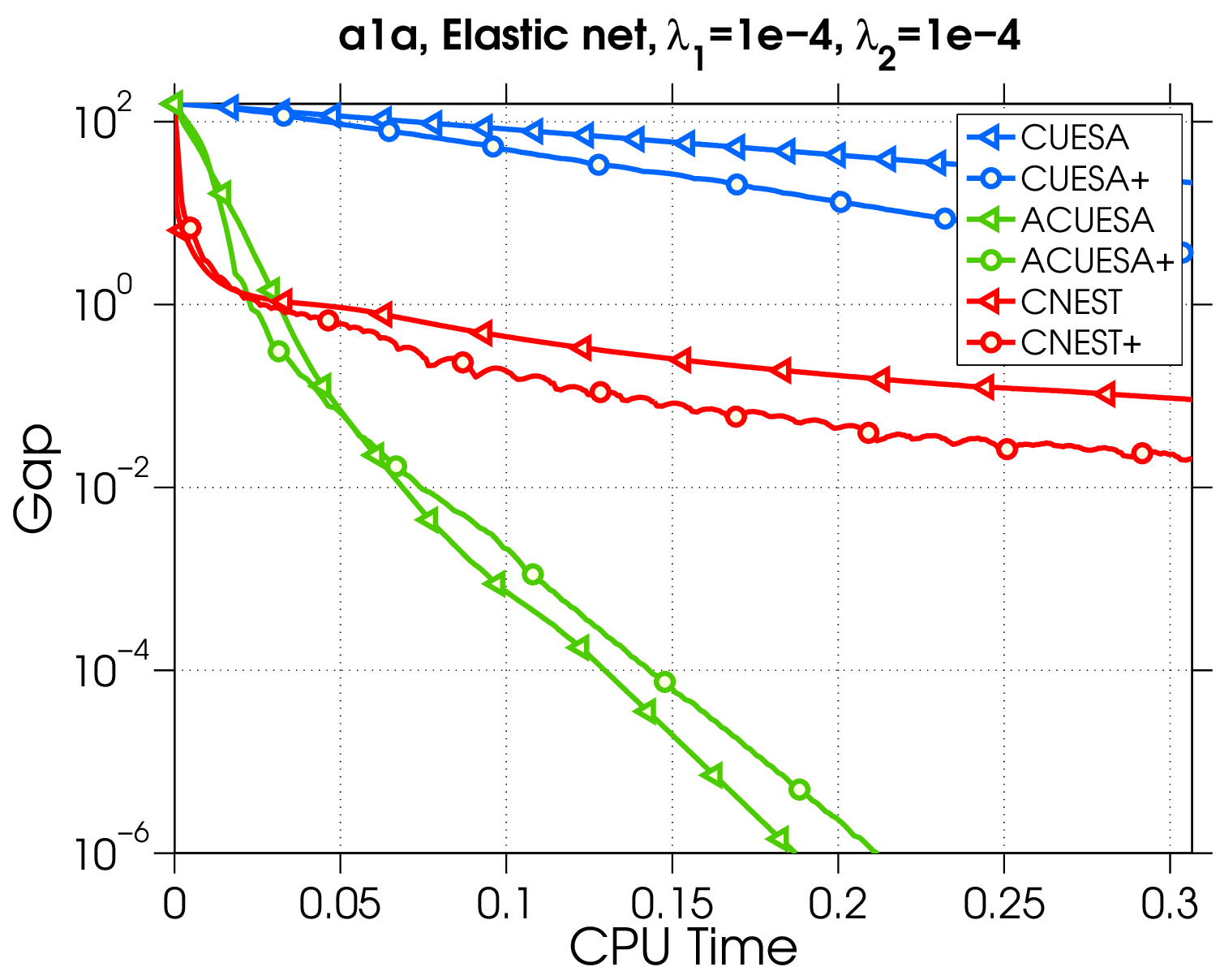}
	\includegraphics[scale=.15]{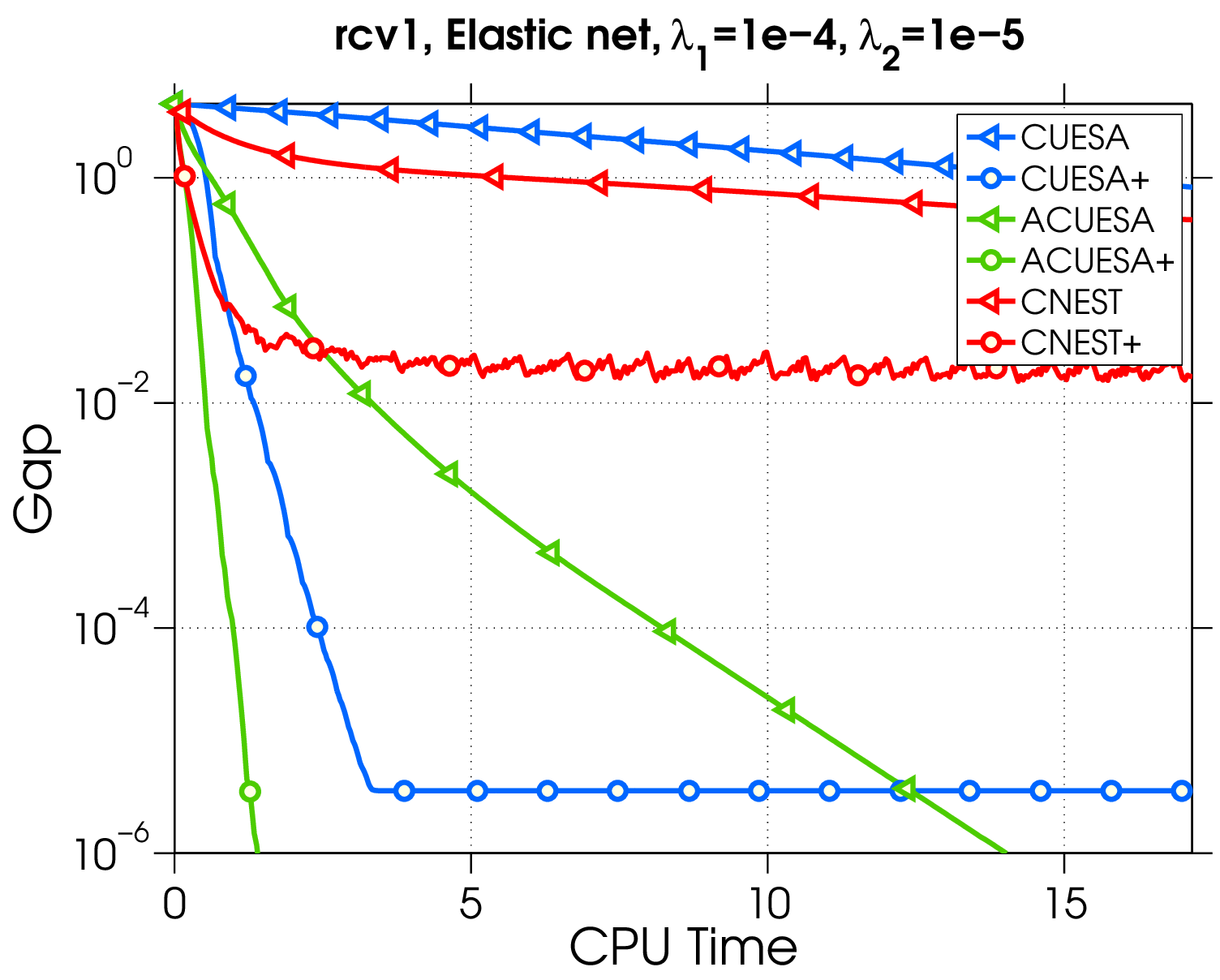}	
	\includegraphics[scale=.15]{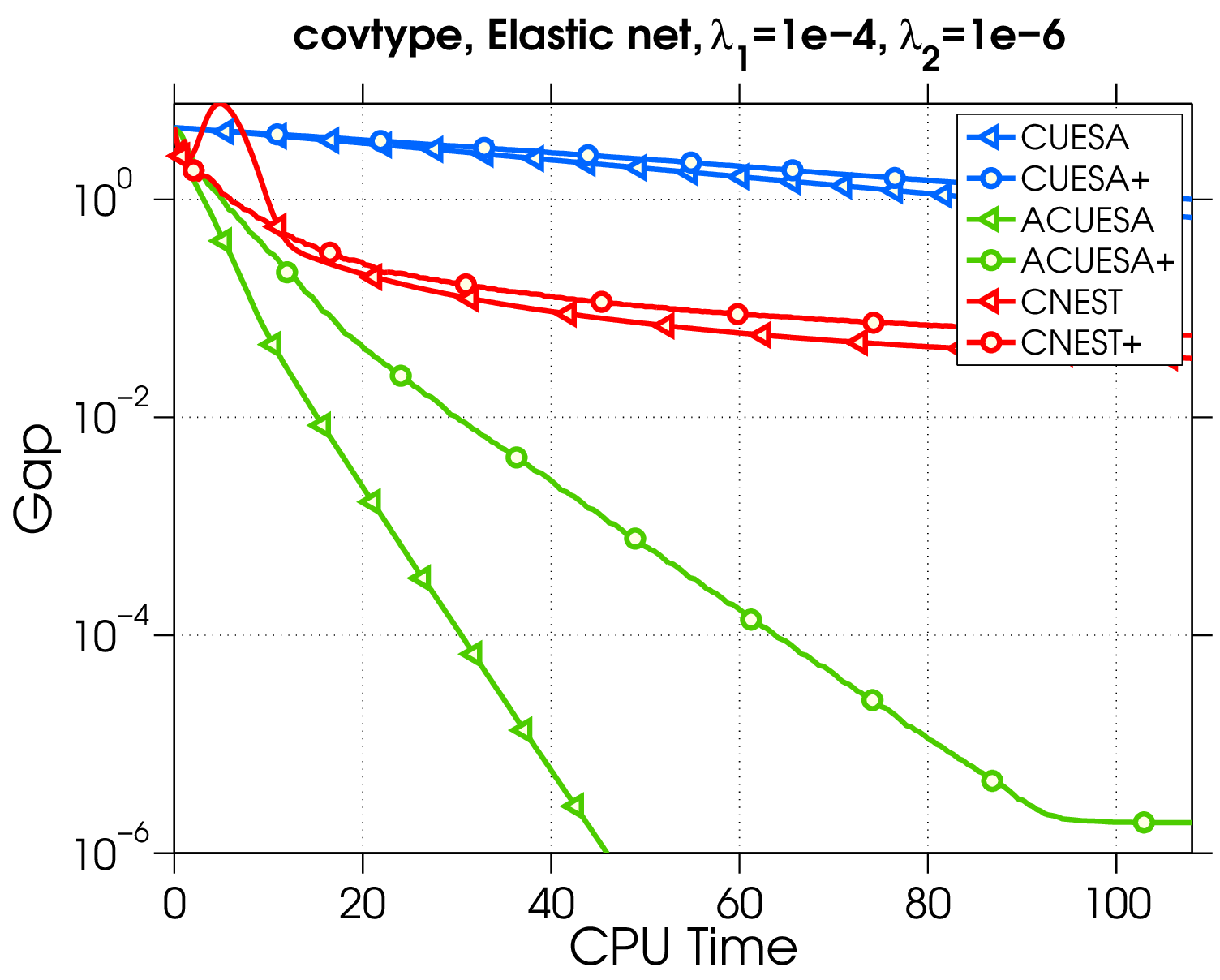}	
 	\caption{Comparison of how gaps between the objective values and the minimum amount of the lower bounds decreases for different algorithms. We observe the advantage of ACUESA+ in both number of function evaluations and running time.}
	\label{fig:comp}
\end{figure}
The results of this experiment are presented in Figure~\ref{fig:comp}, and they show the clear practical advantage of the ACUESA algorithm. The ACUESA algorithm outperforms the CNEST algorithm in all problem instances. Interestingly, on the \texttt{rcv1} dataset, the CUESA+ algorithm (CUESA with an adaptive Lipschitz constant) performs better than the accelerated ACUESA algorithm, although the ACUESA+ (accelerated plus adaptive Lipschitz constant) algorithm is still the best overall.

In the final numerical experiment presented here, we investigate the theoretical vs practical performance of CUESA and ACUESA. We set up three problems using each of the 3 datasets already described, and the results are presented in Figure~\ref{fig:comptheoryvsprac}.
\begin{figure}[H]
	\centering
	\includegraphics[scale=.15]{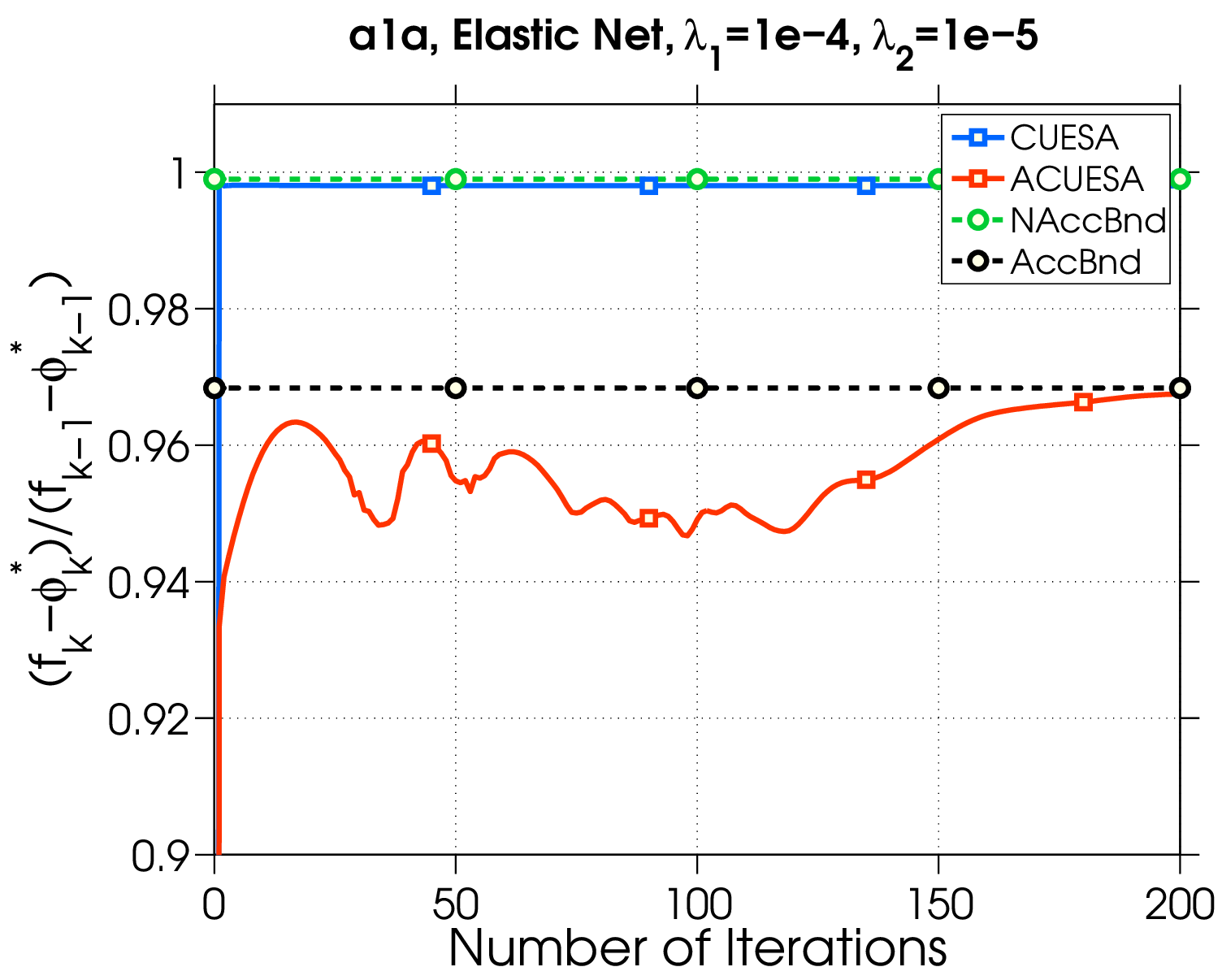}
	\includegraphics[scale=.15]{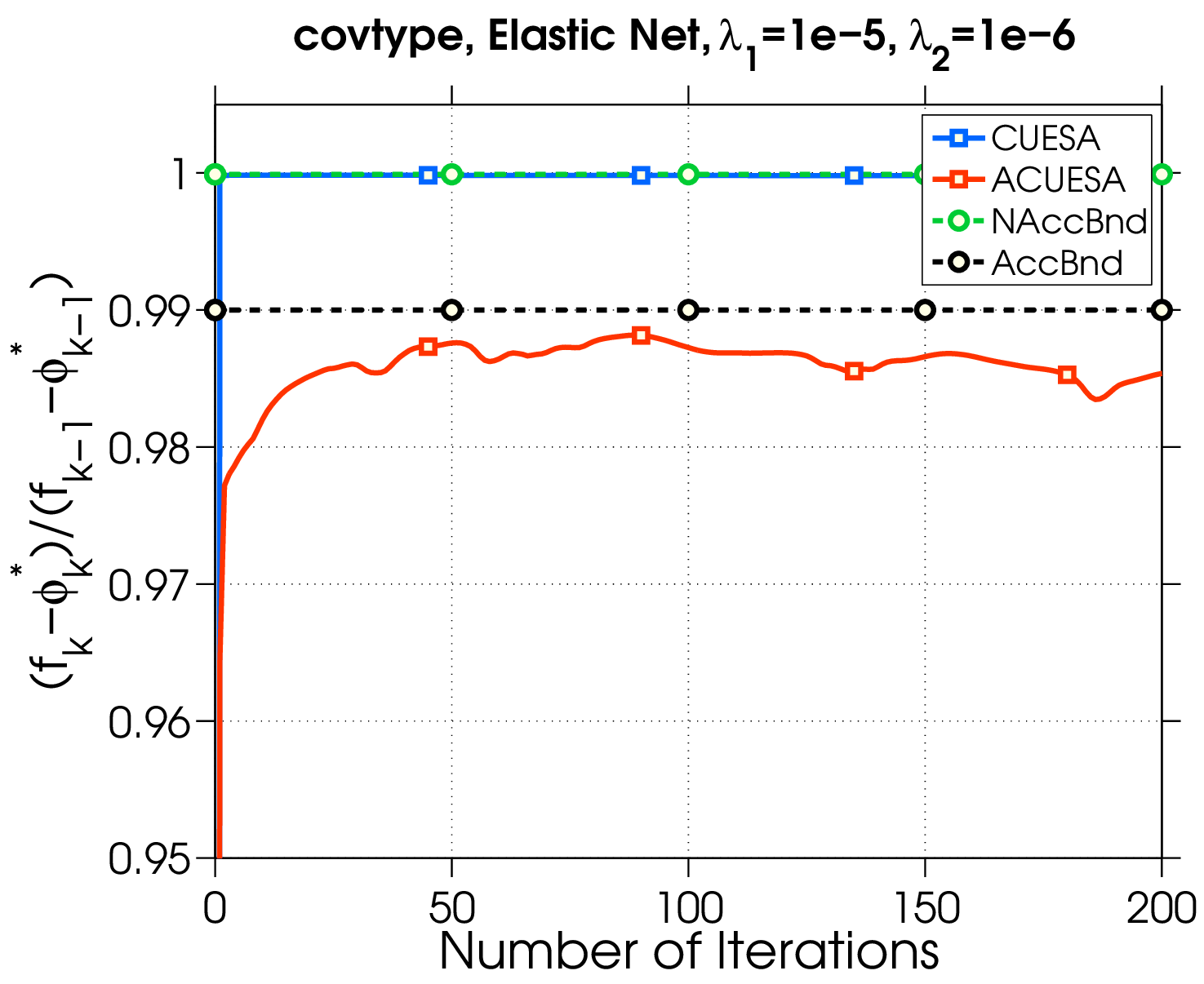}
	\includegraphics[scale=.15]{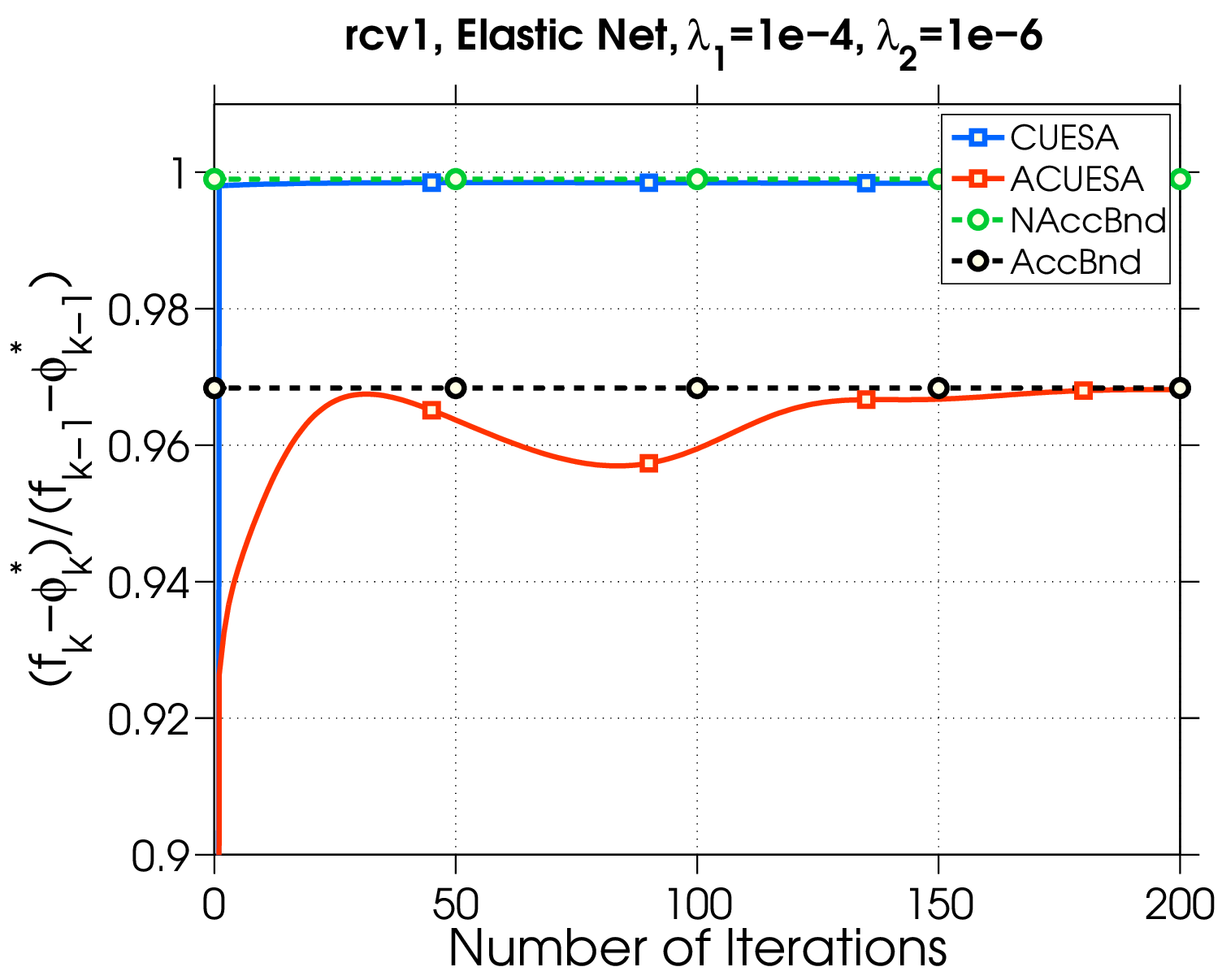}
	\caption{Comparison of  $\tfrac{f(x_k) - \phi^*_{k}}{f(x_{k-1}) - \phi^*_{k-1}}$ for CUESA and ACUESA and $1 - \tfrac{\mu}{L}$(green line) and $1-\sqrt{\tfrac{\mu}{L}}$ (black line). Here we have a similar observation as in Figure \ref{fig:comp}.}
	\label{fig:comptheoryvsprac}
\end{figure}
As before, the green line represents the theoretical (unaccelerated) rate $1 - \tfrac{\mu}{L}$ and the black line represents the theoretical (accelerated) rate $1 - \sqrt{\tfrac{\mu}{L}}$. Note that the practical performance of CUESA closely matches the theoretical rate. We also observe that the practical performance of ACUESA is always at least as good as the theoretical rate, and can often get better decrease in the gap per iteration than $1 - \sqrt{\tfrac{\mu}{L}}$.

All the numerical results presented in this section strongly support the practical success of the SUESA, ASUESA, CUESA and ACUESA algorithms.

\section{Conclusion}\label{sec:conclusion}

In this paper we studied efficient algorithms for solving the strongly convex composite problem \eqref{eq:Problem}. Four new algorithms were proposed --- CUESA, ACUESA, SUESA and ASUESA --- to solve \eqref{eq:Problem} in both the composite and smooth cases. Each algorithm maintains a global lower bound on the objective function value, which can be used as an algorithm stopping condition to ensure and $\epsilon-$optimal solution. Moreover, the definition of a new underestimate sequence that incorporates three sequences, one of which is a global lower bound on the objective function, and this framework was used to establish convergence guarantees for the algorithms proposed here. Our algorithms have a linear rate of convergence, and the two accelerated variants (ACUESA and ASUESA) converge at the optimal linear rate. We also presented a strategy to adaptively select a local Lipschitz constant for the situation when one does not wish to, or cannot, compute the true Lipschitz constant. Numerical experiments show that our algorithms are computationally competitive when compared with other state-of-the-art methods including Nesterov's accelerated gradient methods and optimal quadratic averaging methods.

\bibliographystyle{plain}
\bibliography{citations}

\begin{thebibliography}{10}

\bibitem{allen2016even}
Zeyuan Allen-Zhu, Zheng Qu, Peter Richtarik, and Yang Yuan.
\newblock Even faster accelerated coordinate descent using non-uniform
  sampling.
\newblock In Maria~Florina Balcan and Kilian~Q. Weinberger, editors, {\em
  Proceedings of The 33rd International Conference on Machine Learning},
  volume~48 of {\em Proceedings of Machine Learning Research}, pages
  1110--1119, New York, New York, USA, 20--22 Jun 2016. PMLR.

\bibitem{Baes09}
Michel Baes.
\newblock Estimate sequence methods: extensions and approximations.
\newblock Technical Report Optimization-Online 2372, Universit\'{e} Catholique
  de Louvain, August 2009.

\bibitem{Bubeck15}
S{\'e}bastien Bubeck, Yin~Tat Lee, and Mohit Singh.
\newblock A geometric alternative to {N}esterov's accelerated gradient descent.
\newblock Technical report, Microsoft Research, 2015.
\newblock arXiv:1506.08187[math.OC].

\bibitem{Chang11}
Chih-Chung Chang and Chih-Jen Lin.
\newblock {LIBSVM}: a library for support vector machines.
\newblock {\em ACM Transactions on Intelligent Systems and Technology}, pages
  2:27:1--27:27, 2011.
\newblock \texttt{http://www.csie.ntu.edu.tw/cjlin/libsvmtools/datasets}.

\bibitem{Chen16}
Shixiang Chen, Shiqian Ma, and Wei Liu.
\newblock Geometric descent method for convex composite minimization.
\newblock In I.~Guyon, U.~V. Luxburg, S.~Bengio, H.~Wallach, R.~Fergus,
  S.~Vishwanathan, and R.~Garnett, editors, {\em Advances in Neural Information
  Processing Systems 30}, pages 636--644. Curran Associates, Inc., 2017.

\bibitem{cotter2011better}
Andrew Cotter, Ohad Shamir, Nati Srebro, and Karthik Sridharan.
\newblock Better mini-batch algorithms via accelerated gradient methods.
\newblock In J.~Shawe-Taylor, R.~S. Zemel, P.~L. Bartlett, F.~Pereira, and
  K.~Q. Weinberger, editors, {\em Advances in Neural Information Processing
  Systems 24}, pages 1647--1655. Curran Associates, Inc., 2011.

\bibitem{diakonikolas2019}
Jelena Diakonikolas and Lorenzo Orecchia.
\newblock The approximate duality gap technique: A unified theory of
  first-order methods.
\newblock {\em SIAM Journal on Optimization}, 29(1):660--689, 2019.

\bibitem{Drusvyatskiy16}
Dmitriy Drusvyatskiy, Maryam Fazel, and Scott Roy.
\newblock An optimal first order method based on optimal quadratic averaging.
\newblock {\em SIAM Journal on Optimization}, 28(1):251--271, 2018.

\bibitem{fercoq2015accelerated}
Olivier Fercoq and Peter Richt{\'a}rik.
\newblock Accelerated, parallel, and proximal coordinate descent.
\newblock {\em SIAM Journal on Optimization}, 25(4):1997--2023, 2015.
\newblock 10.1137/130949993.

\bibitem{Fountoulakis18}
Kimon Fountoulakis and Rachael Tappenden.
\newblock A flexible coordinate descent method.
\newblock {\em Computational Optimization and Applications}, 70(2):351--394,
  2018.

\bibitem{Ghadimi12}
Saeed Ghadimi and Guanghui Lan.
\newblock Optimal stochastic approximation algorithms for strongly convex
  stochastic composite optimization {I}: A generic algorithmic framework.
\newblock {\em SIAM Journal on Optimization}, 22(4):1469--1492, 2012.
\newblock doi:10.1137/110848876.

\bibitem{jaggi2014communication}
Martin Jaggi, Virginia Smith, Martin Takac, Jonathan Terhorst, Sanjay Krishnan,
  Thomas Hofmann, and Michael~I Jordan.
\newblock Communication-efficient distributed dual coordinate ascent.
\newblock In Z.~Ghahramani, M.~Welling, C.~Cortes, N.~D. Lawrence, and K.~Q.
  Weinberger, editors, {\em Advances in Neural Information Processing Systems
  27}, pages 3068--3076. Curran Associates, Inc., 2014.

\bibitem{johnson2013accelerating}
Rie Johnson and Tong Zhang.
\newblock Accelerating stochastic gradient descent using predictive variance
  reduction.
\newblock In C.~J.~C. Burges, L.~Bottou, M.~Welling, Z.~Ghahramani, and K.~Q.
  Weinberger, editors, {\em Advances in Neural Information Processing Systems
  26}, pages 315--323. Curran Associates, Inc., 2013.

\bibitem{kingma2014adam}
Diederik Kingma and Jimmy Ba.
\newblock Adam: A method for stochastic optimization.
\newblock In {\em Proceedings of the 3rd International Conference on Learning
  Representations (ICLR)}, 7--9 May 2015, San Diego, USA, 2014.
\newblock arXiv:1412.6980.

\bibitem{Lin15}
Hongzhou Lin, Julien Mairal, and Zaid Harchaoui.
\newblock A universal catalyst for first-order optimization.
\newblock In C.~Cortes, N.~D. Lawrence, D.~D. Lee, M.~Sugiyama, and R.~Garnett,
  editors, {\em Advances in Neural Information Processing Systems 28}, pages
  3384--3392. Curran Associates, Inc., 2015.

\bibitem{Lu15}
Zhaosong Lu and Lin Xiao.
\newblock On the complexity analysis of randomized block-coordinate descent
  methods.
\newblock {\em Mathematical Programming}, 152:615--642, 2015.
\newblock doi:10.1007/s10107-014-0800-2.

\bibitem{ACOCOA}
Chenxin Ma, Martin Jaggi, Frank~E. Curtis, Nathan Srebro, and Martin
  Tak\'a\v{c}.
\newblock An accelerated communication-efficient primal-dual optimization
  framework for structured machine learning.
\newblock Technical report, Lehigh University, USA, 2017.
\newblock arXiv:1711.05305 [math.OC].

\bibitem{ma2015adding}
Chenxin Ma, Virginia Smith, Martin Jaggi, Michael Jordan, Peter Richtarik, and
  Martin Takac.
\newblock Adding vs. averaging in distributed primal-dual optimization.
\newblock In Francis Bach and David Blei, editors, {\em Proceedings of the 32nd
  International Conference on Machine Learning}, volume~37 of {\em Proceedings
  of Machine Learning Research}, pages 1973--1982, Lille, France, 07--09 Jul
  2015. PMLR.

\bibitem{nesterov2012efficiency}
Yu~Nesterov.
\newblock Efficiency of coordinate descent methods on huge-scale optimization
  problems.
\newblock {\em SIAM Journal on Optimization}, 22(2):341--362, 2012.

\bibitem{Nesterov83}
Yurii Nesterov.
\newblock A method for solving the convex programming problem with convergence
  rate $o(1/k^2)$.
\newblock {\em Dokl. Akad. Nauk SSSR}, 269(3):543--547, 1983.

\bibitem{Nesterov04}
Yurii Nesterov.
\newblock {\em Introductory Lectures on Convex Optimization: A Basic Course},
  volume~87 of {\em Applied Optimization}.
\newblock Springer (Originally published by Kluwer Academic Publishers), 2004.
\newblock doi:10.1007/978-1-4419-8853-9.

\bibitem{Nesterov05}
Yurii Nesterov.
\newblock Smooth minimization of non-smooth functions.
\newblock {\em Mathematical Programming}, 103(1):127--152, 2005.
\newblock doi:10.1007/s10107-004-0552-5.

\bibitem{Nesterov07}
Yurii Nesterov.
\newblock Gradient methods for minimizing composite objective function.
\newblock CORE Discussion Paper 2007/76, Universit\'{e} Catholique de Louvain,
  2007.

\bibitem{Nesterov08}
Yurii Nesterov.
\newblock Accelerating the cubic regularization of newton's method on convex
  problems.
\newblock {\em Mathematical Programming}, 112(1):159--181, 2008.

\bibitem{Nesterov13}
Yurii Nesterov.
\newblock Gradient methods for minimizing composite functions.
\newblock {\em Mathematical Programming}, 140(1):125--161, 2013.
\newblock doi:10.1007/s10107-012-0629-5.

\bibitem{nitanda2014stochastic}
Atsushi Nitanda.
\newblock Stochastic proximal gradient descent with acceleration techniques.
\newblock In Z.~Ghahramani, M.~Welling, C.~Cortes, N.~D. Lawrence, and K.~Q.
  Weinberger, editors, {\em Advances in Neural Information Processing Systems
  27}, pages 1574--1582. Curran Associates, Inc., 2014.

\bibitem{o2015adaptive}
Brendan O’donoghue and Emmanuel Candes.
\newblock Adaptive restart for accelerated gradient schemes.
\newblock {\em Foundations of computational mathematics}, 15(3):715--732, 2015.

\bibitem{richtarik2014iteration}
Peter Richt\'{a}rik and Martin Tak\`{a}\v{c}.
\newblock Iteration complexity of randomized block-coordinate descent methods
  for minimizing a composite function.
\newblock {\em Mathematical Programming}, 144(1-2):1--38, 2014.

\bibitem{robbins1951stochastic}
Herbert Robbins and Sutton Monro.
\newblock A stochastic approximation method.
\newblock {\em The annals of mathematical statistics}, pages 400--407, 1951.

\bibitem{schmidt2017minimizing}
Mark Schmidt, Nicolas Le~Roux, and Francis Bach.
\newblock Minimizing finite sums with the stochastic average gradient.
\newblock {\em Mathematical Programming}, 162(1-2):83--112, 2017.
\newblock doi:10.1007/s10107-016-1030-6.

\bibitem{shalev2013accelerated}
Shai Shalev-Shwartz and Tong Zhang.
\newblock Accelerated mini-batch stochastic dual coordinate ascent.
\newblock In C.~J.~C. Burges, L.~Bottou, M.~Welling, Z.~Ghahramani, and K.~Q.
  Weinberger, editors, {\em Advances in Neural Information Processing Systems
  26}, pages 378--385. Curran Associates, Inc., 2013.

\bibitem{Tappenden2016}
R.~Tappenden, P.~Richt\'{a}rik, and J.~Gondzio.
\newblock Inexact coordinate descent: Complexity and preconditioning.
\newblock {\em J Optim Theory Appl}, 170:144–176, 2016.

\bibitem{Tappenden16}
Rachael Tappenden, Peter Richt{\'a}rik, and Jacek Gondzio.
\newblock Inexact coordinate descent: Complexity and preconditioning.
\newblock {\em Journal of Optimization Theory and Applications},
  170(1):144--176, 2016.

\end{thebibliography}

\appendix
\section{Comparison of a UES and Nesterov's ES}\label{sec:appendix}

Here we briefly compare the definition of an underestimate sequence with Nesterov's definition of an estimate sequence. The original definition of an ES only applied to smooth functions that satisfy Assumption~\ref{A_SCL} so we restrict our discussion to this case. Moreover, we will assume that the same sequence $\{\alpha_k\}_{k=0}^{\infty}$ is used when discussing a UES and an ES. Consider the following definition.
\begin{definition}[Definition 2.2.1 in \cite{Nesterov04}]\label{def:NesterovES}
  A pair of sequences $\{\phi_k^N(x)\}_{k=0}^{\infty}$ and $\{\lambda_k^N\}_{k=0}^{\infty}$ $\lambda_k^N \geq 0$ is called an estimate sequence of function $f(x)$ if
  \begin{itemize}
    \item[(i)] $\lambda_k^N \to 0$; and
    \item[(ii)] for any $x \in \R^n$ and all $k\geq 0$ we have $\phi_k^N(x) \leq (1-\lambda_k^N)f(x) +\lambda_k^N \phi_0^N(x).$
  \end{itemize}
\end{definition}
The definition is general and does not say anything about convergence. With this in mind, Nesterov's ES is coupled with the following lemma.
\begin{lemma}[Lemma 2.2.1 in \cite{Nesterov04}]\label{lem:NestConvegence}
   If for some sequence $\{x_k\}$ we have
   \begin{equation}\label{eq:NestUB}
     f(x_k) \leq (\phi_k^N)^* \eqdef \min_{x\in\R^n} \phi_k^N(x),
   \end{equation}
   then $f(x_k) - f^* \leq \lambda_k^N (\phi_0^N(x^*) - f^*)\to 0$.
\end{lemma}

The first observation to make is the clear difference between $f$ and $\phi_k(x)$ for a UES and an ES. In particular, for a UES, at \emph{every iteration} it holds that $\phi_k(x)\leq F(x)$ \eqref{eq:def1} so every $\phi_k(x)$ ($k\geq0$) is a global underestimate (global lower bound) of the objective function \emph{for all} $x\in \R^n$. However, for an ES, the function $\phi_k^N(x)$ is not necessarily a global upper bound for $f$ and it is necessarily not a global lower bound for $f$. What must hold is that, at iteration $k$, the minimizer of the approximation function $\phi_k^N(x)$ must be at least as large as $f(x_k)$ (the value of the objective function at the current point).

A second major differences that one observes is as follows. If an algorithm generates a series of sequences that form a UES (satisfying Definition~\ref{def:UES}), then (assuming $\sum_{k=0}^{\infty} \alpha_k = \infty$), the algorithm is guaranteed to converge (see Proposition~\ref{prop:UES}). On the other hand, if an algorithm generates a series of sequences that form an ES (satisfying Definition~\ref{def:NesterovES}), there is no such algorithm convergence guarantee. This statement is made concrete by considering the next lemma and the text that follows it.

\begin{lemma}[Lemma 2.2.2 in \cite{Nesterov04}]\label{lem:N_ES}
  Assume that (1) $f$ satisfies Assumption~\ref{A_SCL}, (2) $\phi_0^N(x)$ is an arbitrary function on $\R^n$, (3) $\{y_k\}_{k=0}^{\infty}$ is an arbitrary sequence in $\R^n$, (4) $\{\alpha_k\}_{k=0}^{\infty}$ with $\alpha_k \in (0,1)$ and $\sum_{k=0}^{\infty} \alpha_k = \infty$ and (5) $\lambda_0^N = 1$. Then the pair of sequences $\{\phi_k^N(x)\}_{k=0}^{\infty}$ and $\{\lambda_k^N\}_{k=0}^{\infty}$ recursively defined by
  \begin{eqnarray}
    \lambda_{k+1}^N &=& (1-\alpha_k)\lambda_k^N,\label{eq:Nlambda}\\
    \phi_{k+1}^N(x) &=& (1-\alpha_k)\phi_k^N(x) + \alpha_k\Big(f(y_k) + \langle \nabla f(y_k),x-y_k \rangle + \tfrac{\mu}2\tnorm{x-y_k}\Big) \label{eq:Nphik}
  \end{eqnarray}
  is an estimate sequence.
\end{lemma}
Combining \eqref{eq:zzzzz4}, with \eqref{eq:lowerboundDF} and \eqref{eq:normequivcomposite} shows that $\phi_{k+1}(x)$ in \eqref{eq:zzzzz4} is equivalent to $\phi_{k+1}^N(x)$ in \eqref{eq:Nphik} and therefore, \emph{the construction in this work is an estimate sequence} ($\phi_{k+1}(x) \equiv  \phi_{k+1}^N(x)$). However, we will now show that the construction does not satisfy \eqref{eq:NestUB}, and therefore, even though the iterates generated by SUESA/ASUESA form an estimate sequence, Lemma~\ref{lem:NestConvegence} \emph{cannot} be used to prove convergence of SUESA/ASUESA.

\begin{lemma}[Lemma~2.2.3 in \cite{Nesterov04}]\label{lem:NestCanonical}
Let $\phi_0^N(x) = (\phi_0^N)^* + \frac{\gamma_0^N}2\tnorm{x-v_0^N}$. Then the process described in Lemma~\ref{lem:N_ES} preserves the canonical form of functions $\{\phi_k^N(x)\}_{k=0}^{\infty}$:
  \begin{equation}\label{eq:nestcanonicalform}
  \phi_{k+1}^N(x) = (\phi_{k+1}^N)^* + \tfrac{\gamma_k^N}2\tnorm{x - v_{k+1}^N},
 \end{equation}
 where the sequences $\{\gamma_k^N\}_{k=0}^{\infty}$, $\{v_k^N\}_{k=0}^{\infty}$ and $\{(\phi_k^N)^*\}_{k=0}^{\infty}$ are defined as follows:
\begin{eqnarray}
  \gamma_{k+1}^N &=& (1-\alpha_k)\gamma_k^N + \alpha_k \mu,\label{eq:Nestgamma}\\
  v_{k+1}^N &=& \tfrac{1}{\gamma_{k+1}^N}\Big((1-\alpha_k)\gamma_k^Nv_k^N + \alpha_k \mu y_k - \alpha_k \nabla f(y_k)\Big),\label{eq:Nestvk}\\
  (\phi_{k+1}^N)^* &=& (1-\alpha_k)(\phi_k^N)^* + \alpha_k f(y_k) - \tfrac{\alpha_k^2}{2 \gamma_{k+1}^N}\tnorm{\nabla f(y_k)} \notag\\
  &&+ \tfrac{\alpha_k(1-\alpha_k)\gamma_k^N}{\gamma_{k+1}^N}\Big(\tfrac{\mu}{2}\tnorm{y_k - v_k^N} + \langle\nabla f(y_k), v_k^N - y_k\rangle\Big)\label{eq:Nestphik}
\end{eqnarray}
\end{lemma}
Note that in \eqref{eq:nestcanonicalform} the numerator in front of the norm term is $\gamma_k^N$, while in \eqref{eq:phicanonical} it is $\mu$. Substituting $\gamma_k^N = \mu$ into \eqref{eq:Nestgamma} gives $\gamma_{k+1}^N = \gamma_k^N = \mu$ so setting $\gamma_0^N = \mu$ ensures $\gamma_k^N$ is fixed for all $k\geq 0$ in Lemma~\ref{lem:NestCanonical}. Now, using $\gamma_{k+1}^N = \gamma_k^N = \mu$ in \eqref{eq:Nestvk}, and recalling the form of a long step shows that $v_{k+1}^N  \equiv v_{k+1}$ in \eqref{eq:vkcomp}. It remains to observe that \eqref{eq:phistarequiv} (combined with \eqref{eq:normequivcomposite}) is equivalent to \eqref{eq:Nestphik}.

Thus, the difference between the construction in this work and the construction in \cite{Nesterov04} comes down to the minimizer and minimal value of $\phi_0^{N}(x)$. For an ES, it must hold that $f(x_0)\leq (\phi_0^{N})^*$
, and the initialization of scheme (2.2.6) in \cite{Nesterov04}, as well as the proof of Theorem 2.2.1, explicitly mentioned the use of the choice $v_0 = x_0$ for Nesterov's method. However, note that other choices of $v_0$ can still provide the equality $(\phi_0^{N})^* = f(x_0)$, which means that the minimal value of the first element in the ES would be unchanged, i.e., $f(x_0)$, and the minimizer would be shifted from $x_0$ to other points. On the other hand, this contrasts with SUESA/ASUESA, where it is required that $\phi_0(x) \leq f(x)$ and so they are initialized with $v_0 = x_0^{++}$ and $\phi_0^*= f(x_0) - \tfrac{1}{2\mu}\tnorm{\nabla f(x_0)} \textcolor{black}{\leq f(x_0)}$; see \eqref{eq:c0v0}.

Finally, note that Definition~\ref{def:UES} also holds for composite functions. While Definition~\ref{def:NesterovES} only holds for smooth functions, Nesterov has extended the ES framework to the composite setting; see Section~4 in \cite{Nesterov08} and Section~4 in \cite{Nesterov13}. Moreover, the relationship between the OQA method and an ES is discussed Appendix A in \cite{Drusvyatskiy16}.

\textcolor{black}
 {
\section{Comparison of a UES and the study \cite{diakonikolas2019}}
In \cite{diakonikolas2019}, the authors proposed a general scheme for the analysis of first-order methods. For the strongly convex cases (both smooth and nonsmooth), there are some similarities between the methods in \cite{diakonikolas2019} and our proposed methods, but we stress that the approaches in this work are inherently different from those in \cite{diakonikolas2019}. To be more specific, Table \ref{tab:compareWithRev2Pape} summarizes the iterates of ASUESA, and compares them with the iterates (for the smooth and strongly convex setting) in \cite{diakonikolas2019}. The similarities and differences between the corresponding two studies are described as follows:
 \begin{enumerate}
     \item The $\beta_k$s are different for each method, and if the problem is ill-conditioned (large $\kappa$), then the $\beta_k$s are close to each other ($\approx 1-\dfrac{1}{\sqrt{\kappa}}$).
     \item The $y_k$s have the same update structure, but because the $\beta_k$s are different, this results in \emph{different updates} $y_k$ (later we will see that the $v_k$s and $x_k$s are also different as the algorithms progress). 
     \item The $x_{k+1}$s are identical in \emph{structure}. However, again the iterates $x_{k+1}$ are \emph{different} because the $\beta_k$s, and subsequently $y_k$s, are \emph{different} for both methods.
     \item The $v_{k+1}$s are different both in structure and clearly in their values.
     \item The $x_k$s in our proposed study form part of an underestimate sequence, which guarantees the natural stopping criteria, i.e., $f(x_k) - \phi_k^*$ goes to zero in a linear rate. 
 \end{enumerate}
 }
 \begin{table}[htb]
	\centering
	\caption{Comparison of the iterates}
	\begin{tabu}{l|c|c}
		\toprule
		\textbf{Iterates}  & \textbf{ASUESA (current study)}& \textbf{Diakonikolas  and  Orecchia[2019]} \\
		\midrule
		\rowfont{\color{black}} $\beta_k $& $\dfrac{1}{1+ \sqrt{1/ \kappa}}$& $1- \dfrac{\sqrt{4\kappa +1}-1}{2\kappa}$ \\\hdashline
		\rowfont{\color{black}}$y_k$& $\beta_k x_k + (1-\beta_k)v_k$ & $\beta_k x_k + (1-\beta_k)v_k$ \\\hdashline
		\rowfont{\color{black}}$x_{k+1}$& $y_k - \dfrac{1}{L} \nabla f(y_k)$ & $y_k - \dfrac{1}{L} \nabla f(y_k)$\\ \hdashline
		\rowfont{\color{black}}$v_{k+1}$& $(1-\alpha_k) v_k + (\alpha_k)(y_k -\dfrac{1}{\mu} \nabla f(y_k) )$ & $\sum a_i (-\mu y_i - \mu A^{(k+1)}\nabla f(y_i))$\\
		\bottomrule
	\end{tabu}
	\label{tab:compareWithRev2Pape}
\end{table} 

\end{document}